%% file: FrogPhaseTransition.tex
\documentclass[11pt]{article}

\title{Existence and sharpness of the phase transition for the frog model on transitive graphs}

\author{
  Omer Angel
  \thanks{University of British Columbia, Vancouver, Canada.  Research was supported in part by NSERC, an SLMath Clay senior scholarship, and
    Magdalen college, Oxford. Email: angel@math.ubc.ca }
  \and
   Daniel de la Riva
  \thanks{University of British Columbia, Vancouver, Canada. Research was supported in part by NSERC. Email: delariva@math.ubc.ca }
  \and 
   Jonathan Hermon
  \thanks{University of British Columbia, Vancouver, Canada. Research was supported in part by NSERC. Email: jhermon@math.ubc.ca }
  \and 
   Yuliang Shi
  \thanks{\textit{Corresponding author}: University of British Columbia, Vancouver, Canada. Research was supported in part by NSERC. Email: yuliang@math.ubc.ca }
}

\date{}

\input{header.tex}

\usepackage{latexsym, tikz, float, dsfont, mathrsfs, pgfplots, subfigure, bbm, mathabx}
\usetikzlibrary{hobby}
\usepackage{nccmath}
\usepackage{comment}

\usepackage[maxbibnames=99]{biblatex}
\providecommand{\keywords}[1]
{\small\noindent\textbf{Keywords:} #1}

\newcommand{\msc}[1]{\small\noindent\textbf{Mathematics Subject Classification (2020):} #1}

\pgfplotsset{compat=1.18}

\begin{document}

\maketitle

\begin{abstract}
We consider a slight modification of the frog model. For a given graph, each vertex has $\mathrm{Poisson}(\lambda)$ particles (or frogs). At time zero, only the particles at the origin are active, and all the other particles are sleeping. Each active particle performs an independent, continuous-time simple random walk up to a fixed lifetime $t$, after which the particle dies and is removed from the system. 
Once an active frog jumps to a vertex, it activates all of its particles. The survival of active particles can be studied as a dependent percolation model with two parameters $\lambda$ and $t$.  In the present work, we establish the existence of a phase transition with respect to each parameter for non-amenable graphs of bounded degree and quasi-transitive graphs of superlinear polynomial growth, as well as prove the sharpness of the phase transition for transitive graphs.  
\end{abstract}


\msc{60K35, 60G50, 60J80.}

\section{Introduction}\label{introduction}
The frog model has been studied extensively in the probability community since its introduction in \cite{telcs1999branching}. 
In the present work, we consider the following variant of the \emph{frog model with death}, credited to Itai Benjamini and popularized by the works of Alves, Machado, and Popov (see \cite{FrogShape,FrogArt}), and independently by Ram\'irez and Sidoravicius (see \cite{ramirez2004asymptotic}). 
Let $G:=(V,E)$ be a graph, where $V$ is the set of vertices (sometimes called sites), and $E$ is the set of edges. 
Throughout this paper, we assume implicitly that $G$ is infinite, connected, and locally finite. At each site of $G$, there is a random number of particles distributed according to independent $\text{Poisson}(\lambda)$ random variables. 
We let $\mathbf{0} \in V$ denote the origin of the graph. Initially, all particles are inactive except for those at the origin. Each active particle performs a continuous-time simple random walk on the graph $G$ for a fixed lifetime $t$, after which it dies and is removed from the system. 
Once an active particle jumps to a site with inactive particles, it activates all of them.  The newly activated particles start independent continuous-time simple random walks themselves, performing the same dynamics as above. 
Other variations can also be considered; for instance, the particles can have geometric or exponential random lifetimes, but we chose a fixed lifetime since it is consistent with the previous literature.
The main goal of this paper is to establish the existence of a phase transition (the process almost surely dies out or has a positive probability of continuing indefinitely) for this model on large classes of graphs.
This was originally proven on $\Z^{d}$ and the degree-$d$ homogeneous tree $\T_{d}$ in \cite{FrogArt}. 
Moreover, we establish the sharpness of the phase transition for vertex-transitive graphs. 

To state our results precisely, we first define the model rigorously.
For each $x \in V$, the number of particles at site $x$, denoted by $\eta_{x}$, is an independent random variable distributed according to $\mathrm{Poisson}(\lambda)$. We denote the collection of particles at site $x$ by $\mathcal{P}(x)=(\omega_{x}^{1},\ldots,\omega_{x}^{\eta_{x}})$.
For each particle $(\omega_{x}^{i})_{1\leq i \leq \eta_{x}}$, we have $(Y_{x}^{i}(s))_{0\leq s \leq t}$, which is an independent rate 1 continuous-time simple random walk\footnote{That is, each particle has an $\mathrm{Exp}(1)$ clock, and when it rings, the particle uniformly selects one of its neighbors.} starting at $x$ on $G$ with lifespan $t$. Initially, only the particles at the origin (which may form an empty set) are activated and perform their respective random walks. Once a particle visits another vertex, it activates all particles at that vertex, leading to a chain of activations. 
We denote by $\mathbb{P}_{\lambda,t}$ the probability measure associated with this process. The probability measure associated with the simple random walk starting at $x$ is denoted by $\mathbb{P}_{x}$.

We say that a realization of the frog model survives if at all times there exists at least one active particle. 
Equivalently (almost surely), infinitely many vertices are visited.
This can be phrased as a percolation problem as follows.
For $x\in V,$ every particle $\omega_{x}^{i}$ has a set of sites that it visits up to time $t\geq 0$:
\begin{equation}\label{trajectory}
\mathrm{R}^{i}_{x} := \{ Y_{x}^{i}(s): 0\leq s \leq t\}\subset V.
\end{equation}
We define the range of $x$ as the union over the trajectories of all particles starting at $x$:
\begin{align} \label{def: range_of_vertex}
\mathcal{R}_{x}:=\begin{cases}
            \bigcup\limits_{i=1}^{\eta_x}\mathrm{R}_{x}^{i}, & \text{if $\eta_x>0$}\\
            \{x\}, & \text{if $\eta_x=0$}.
         \end{cases}
\end{align}
This range is relevant for the process only when $x$ is visited by an active particle, which activates all its particles. 
However, in some arguments, it is convenient to sample the trajectories of all particles, regardless of whether they are ever activated. With this in mind, one may observe that the survival of the frog model is equivalent to the cluster of $\mathbf{0}$ being infinite in a certain directed percolation model. From each site $x$, we draw an oriented edge represented by $\Rightarrow$ from $x$ to each site in $\mathcal{R}_{x}$. This gives rise to a directed graph on the vertices of $G$.

Associated with this percolation model, we define the following notation:
For any $S \subseteq V,$ we say that $x$ is connected to $y$ in $S$ if there exists a sequence of vertices $\{v_{0},\ldots,v_{k}\}\subseteq S,$ with $v_{0}=x, v_{k}=y$ and each directed edge $(v_{\ell},v_{\ell+1})$ is open, represented by the connection $v_{\ell} \Rightarrow v_{\ell+1}$. We denote this event by $x \xrightarrow{S} y$. 
Furthermore, for $A \subset V,$ we denote by $x\xrightarrow{S} A$ the event that $x$ is connected to some vertex in $A$. 
When $S=V$, we simply write $x\rightarrow y$ and $x\rightarrow A$. Finally, we define $B_{x}(n):=\{y \in V: d_{G}(x,y)\leq n\},$ where $d_{G}(\cdot,\cdot)$ is the graph distance, and set $B(n):= B_{\mathbf{0}}(n)$. We denote by $\mathbf{0} \rightarrow \infty$ the event that $\mathbf{0}$ is connected to $B(n)^{c}$ for every $n\geq1$. 
This event is a.s. the same as the event that, at any given time, there is at least one active particle in the system. 
For this reason, this event is sometimes called \textit{survival of the infection} in the literature (see Definition 1.1 in \cite{FrogArt} and Definition 3.5 in \cite{multiscale}). 

In the following we shall fix one of the parameters $t,\lambda$ of the model and vary the other.
Define the critical parameter $\lambda_{c}(t)$ (respectively $t_{c}(\lambda)$) for the frog model as 
\begin{align*}
    \lambda_{\mathrm{c}}(t) &:= \sup \{ \ \lambda \geq 0: \bP_{\lambda, t}(\mathbf{0} \rightarrow \infty) = 0 \ \}, \\
    t_{\mathrm{c}}(\lambda) &:= \sup \{ \  t \, \geq 0: \bP_{\lambda, t}(\mathbf{0}\rightarrow \infty) = 0 \ \},
\end{align*}
and we observe that these quantities are well defined since the survival of the process is clearly monotone in $\lambda$ and $t.$

The study of phase transitions has been one of the most important problems in probability theory, particularly for Bernoulli percolation and statistical mechanics models such as the Ising model. Many interesting techniques have been developed to address questions about the existence, sharpness, and continuity of phase transitions. Although much work remains to be done, it is in the interest of the community to seek answers for different classes of models. The frog model, which admits a natural percolation interpretation as described above, is often interpreted as a model for the spread of infections or rumors. Thus, it presents a different structure and motivation from what has been analyzed in the percolation literature. Although some previous works on the model answered phase transition questions for $\Z^{d}$ and the degree-$d$ homogeneous tree $\T_{d}$ (see \cite{FrogArt}), other works considered questions of recurrence and transience \cite{hoffman_one,hoffman_two,hoffman_three}, and more recent works studied the model on various finite graphs \cite{froghermon,hermon2018frogs,MR4031108}. This is the first time the model has been systematically studied as a percolation model, expanding both the scope of features of the model that are studied and the collection of graphs considered. We begin by presenting a conjecture, which is partially solved in this paper, concerning the existence of a phase transition for the frog model that further motivates this work. We introduce the following definitions.

Recall that a graph $G$ is called \emph{quasi-transitive} if the action of the automorphism group $\mathrm{Aut}(G)$ on $V$ has only finitely many orbits, and \emph{vertex-transitive} if there is only one such orbit.

For $x\in V$, we let $g_{x}(n) := |B_{x}(n)|$, and we set $g(n):= g_{\mathbf{0}}(n)$. A graph is said to have \textit{linear growth} if $\limsup_{n \to \infty} g(n)/n < +\infty$. Due to connectivity, this limit is independent of the choice of $x$. Conversely, a graph exhibits \textit{superlinear growth} if $\limsup_{n \to \infty} g(n)/n = +\infty$. Our central conjecture regarding the existence of the phase transition for the frog model is formulated as follows:

\begin{conjecture} \label{conj: exist_main}
    If $G=(V,E)$ is a transitive graph with superlinear growth, then for every $t>0$, we have $0 < \lambda_{\mathrm{c}}(t) < \infty$, and for every $\lambda>0$, we have $0 < t_{\mathrm{c}}(\lambda) < \infty$. 
\end{conjecture}

Conjecture \ref{conj: exist_main} is a natural analog of the existence of phase transition results in the percolation literature.
It is well known from a Peierls-type argument that for graphs with degrees bounded by $\Delta$, we have $p_{\mathrm{c}}(G) \ge 1/(\Delta-1)$, where $p_{\mathrm{c}}(G)$ is the critical parameter in the Bernoulli percolation model on $G$ (see Section \ref{preliminaries} for a precise definition).
Benjamini and Schramm first conjectured in \cite{MR1423907} that for any quasi-transitive graph $G$ with superlinear growth, we must have $p_{\mathrm{c}}(G) < 1$. This conjecture has been resolved through the joint efforts of many research papers, including \cite{MR1622785, MR1634419, MR3865659, Easo_Severo_Tassion_2025, MR1338287, MR3616205, MR1841989, MR3607808, MR3520023, timar_2007} and especially the work of \cite{PercolationPhaseTran}.

We formulate Conjecture \ref{conj: exist_main} in the same spirit as the celebrated conjecture of Benjamini and Schramm for percolation, but adapted to the frog model context. In this paper, we establish significant portions of this conjecture for quasi-transitive graphs satisfying different conditions that we introduce below.

\begin{definition}\label{def:poly_growth}
 We say that $G$ has \emph{polynomial growth}\footnote{Note that, since the graph is connected, the definition of polynomial growth is independent of the choice of the center vertex: if there is some constant $C = C_{\boldsymbol{0}}$ such that $g_{\boldsymbol{0}}(n) \le C n^d$, then, for any vertex $x$, there exist $C_x$ such that $g_x(n) \le C_x n^d$.} if there exist $C$ and $d$ such that for all $n$
\begin{align*}
g(n) \le C n^d.
\end{align*}
\end{definition}

\begin{definition}\label{def: amenable}
 We say that a graph of bounded degree is \emph{amenable} if it admits a sequence of finite vertex sets \(\{F_n\}\) satisfying the \emph{Følner condition}, i.e., \(|\partial F_n| / |F_n| \rightarrow 0 \) as $n \rightarrow \infty$, where \(\partial F\) denotes the vertex boundary of \(F\). Such a sequence \(\{F_n\}\) is called a \emph{Følner sequence}. We say that a graph of bounded degree is \emph{non-amenable} if there is no Følner sequence.\par
\end{definition}

\noindent Our existence of the phase transition results are summarized in the following two theorems:

\begin{theorem}\label{main:existence}
    Let $G = (V, E)$ be a fixed graph. Then
    \begin{itemize}[nosep,left=5pt]
        \item[1.] For every $t, \lambda > 0$ we have that  $\lambda_{\mathrm{c}}(t) > 0$ and $t_{\mathrm{c}}(\lambda) > 0$;
        \item[2.] If $G$ has bounded degree and satisfies that $p_{\mathrm{c}}(G)<1$, then $\lambda_{\mathrm{c}}(t) < \infty$ for every $t > 0$;
        \item[3.] If $G$ has bounded degree and is non-amenable, then $t_{\mathrm{c}}(\lambda) < \infty$ for every $\lambda > 0$.
    \end{itemize}
\end{theorem}

Items $(1)$ and $(2)$ are proven in Section \ref{preliminaries} and follow standard arguments of stochastic domination. Since it was shown in \cite{PercolationPhaseTran} that any quasi-transitive graph with superlinear growth satisfies $p_{\mathrm{c}}(G)<1,$ item (2) of the above theorem implies that for any $t>0$ we have $\lambda_{\mathrm{c}}(t)<\infty$ for such graphs.
Item $(3)$ is more involved and follows a strategy similar to the ones developed in \cite{localityconjecture,hermon_morris}, which is shown at the end of \cref{sec: existence}.  
We point out that another possible approach to proving Item $(3)$ can potentially be adapted from \cite{pete2025nonamenablepoissonzoo}, where the authors study a related model to the frog model known as Poisson zoo. Here, a percolation configuration is generated as the union of a Poisson collection of lattice animals.
In the non-amenable setting, it is shown that for certain parameters and measures, there are infinite clusters in the resulting configuration. The exploration argument used in \cite{pete2025nonamenablepoissonzoo} differs from the one we use below, though it appears that it could be adapted to yield a variant of our proof.

Under the assumption of quasi-transitivity, we also derive the following: 

\begin{theorem}\label{main_quasi:existence}
    Let $G = (V, E)$ be a quasi-transitive graph. Then, 
    \begin{itemize}[nosep,left=5pt]
        \item[1.] If $G$ has linear growth, then $\lambda_{\mathrm{c}}(t) = t_{\mathrm{c}}(\lambda) = \infty$ for every $t, \lambda > 0$;
        \item[2.] If $G$ has superlinear polynomial growth, then $t_{\mathrm{c}}(\lambda) < \infty$ for every $\lambda > 0$.
    \end{itemize}
\end{theorem}

Item $(1)$ is proved in \cref{preliminaries} and follows a Borel-Cantelli-type argument, similar to the proof that $p_{c}(G)=1$ for quasi-transitive graphs of linear growth.
Item $(2)$ is the more involved argument and is based on a renormalization approach that is described in \cref{sec: existence}.
This renormalization is essentially a consequence of classical results due to Trofimov (see \cite{Trofimov1985GRAPHSWP}), and was developed in \cite{MR4529920} to prove the locality conjecture for transitive graphs of superlinear polynomial growth.

\begin{remark}\label{remark_heavy_tails}
  Item $(2)$ in Theorem \ref{main_quasi:existence} can be generalized beyond simple random walks: let $G = (V, E)$ be a Cayley graph of superlinear polynomial growth and let $d_{G}(x, y)$ denote the graph distance between $x$ and $y$ on $G$.
  Let $\alpha > 0$ and $d\geq2$ be the growth exponent of $G$ (see Equation \eqref{gromov}). For every $x,y \in V$, define the one-step random walk jump kernel by
    \begin{align*}
      P(x, y) \asymp \frac{1}{d_{G}(x, y)^{d + \alpha}}.
    \end{align*}
   We consider the frog model where active particles follow the rate-1 continuous-time version of $P$. At the end of Section \ref{sec: existence}, we provide further details on why this generalization holds.
 \end{remark}

 \begin{remark}
 We point out that we believe that $p_{c}(G)=1$ implies $\lambda_{c}(t) = t_{c}(\lambda)=\infty$ for graphs of bounded degree. However, since this case is not the main focus of our work, and it does not generalize too much the scope of item (1) of Theorem \ref{main_quasi:existence} (for quasi-transitive graphs, $p_c(G)=1$ implies the graph is virtually $\Z$) we state it as \cref{conj: p_c_one}. 
\end{remark}

One can also ask whether the phase transition is sharp. We show that the answer is positive for vertex-transitive graphs. 
Indeed, the next result is even more general, in the sense that \textbf{reversibility} of the random walks is not required in our proof, a condition that is necessary for the definition of Bernoulli percolation. We let $\vec{G}:= (V,\vec{E})$ denote a directed vertex-transitive graph, and for every $x\in V,$ let  $\mathcal{C}_{x}$ be the (outward) cluster of $x,$ that is, all vertices $y \in V,$ such that $x \rightarrow y$. We denote its cardinality by $|\mathcal{C}_{x}|$. 

\begin{theorem}\label{mainsharpness}
    Let $\vec{G} = (V, \vec{E})$ be a directed vertex-transitive graph. The following statements hold:
    \begin{itemize}[nosep,left=5pt]
        \item[1.] For every $t > 0$, there exists a constant $K_{1}(\lambda_{\mathrm{c}}(t))$ such that for any $\lambda$ in $[\lambda_{\mathrm{c}}(t),2\lambda_{\mathrm{c}}(t)],$ we have
        \begin{align}\label{suplambda}
            \bP_{\lambda, t}(\mathbf{0} \rightarrow \infty) \geq K _{1}(\lambda - \lambda_{\mathrm{c}}).
        \end{align} 
        \item[2.] For every $\lambda > 0$, there exists a constant $K_{2}(t_{\mathrm{c}}(\lambda)) $ such that for any $t$ in $[t_{\mathrm{c}}(\lambda),2t_{\mathrm{c}}(\lambda)],$ we have
        \begin{align}\label{supt}
            \bP_{\lambda, t}(\mathbf{0} \rightarrow \infty) \geq K_{2} (t - t_{\mathrm{c}}).
        \end{align} 
        \item[3.] For fixed $t>0,$ and $\lambda < \lambda_{\mathrm{c}}(t),$ there exist constants $C,c>0,$ such that
        \begin{equation}\label{finitesus}
            \bP_{\lambda,t}\left(|\mathcal{C}_{\mathbf{0}}| \geq n\right) \leq Ce^{-cn}, \text{ for every $n\geq0$}. 
        \end{equation}
        \item[4.] For fixed $\lambda > 0,$ and $t < t_{\mathrm{c}}(\lambda),$ there exist constants $C,c>0,$ such that
        \begin{equation}
           \label{finitesus2} \bP_{\lambda,t}\left(|\mathcal{C}_{\mathbf{0}}| \geq n\right) \leq Ce^{-cn}, \text{ for every $n\geq0$}. 
        \end{equation}
    \end{itemize}
\end{theorem}

It is generally known to be very hard to extend sharpness results to dependent percolation models, as illustrated by \cite{sharpnesspotts}. 
Theorem \ref{mainsharpness} provides another example of a dependent percolation model that exhibits the phenomenon of sharpness.  
Additionally, we conjecture the following generalization on \cref{mainsharpness}, for which we lack a proof.
\begin{conjecture}
[Extending \cref{mainsharpness} from simple random walk to arbitrary transitive walks] 
\label{conj:transitive P}    The assertion of \cref{mainsharpness} remains true for the variant of the frog model in which the particles follow (up to time $t$) a rate-1 continuous-time version of some irreducible and transitive (see \S~2.6.2 in \cite{levin2017markov} for a definition) transition kernel $P$ on some countably infinite set $V$.
\end{conjecture}


We conclude this introduction by posing a few additional questions that we believe could lead to interesting directions for study, along with some insights into them.

\subsection{Questions and related works}

As the reader might observe, the percolation model induced by the frog model shares many similarities with Bernoulli bond percolation.  In fact, even though our model is long-range, it should resemble nearest-neighbor (or finite-range) Bernoulli percolation more closely. That is, since the rate of decay of connections is super-polynomial, one might expect that the results about continuity of the phase transition, and the behavior of the critical exponents in higher dimensions, should be similar to the case of nearest-neighbor Bernoulli percolation. Of course, proving continuity of the phase transition for the frog model would generalize the already challenging result for Bernoulli percolation. It is still natural to conjecture that:

\begin{conjecture}
The phase transition is continuous for all transitive graphs.
\end{conjecture}

An already interesting result would be to consider the case of $\mathbb{Z}^{2}.$ The recent techniques developed in \cite{Slabs} for critical Bernoulli percolation on slabs work for finite range percolation models. As previously mentioned, our case is more subtle since we do allow arbitrary connections, but they have super-polynomial decay. One must also recall that the probabilities of connections are not independent, as in the Bernoulli case, and this introduces great difficulties to many of the arguments. However, we still believe that adapting their proof is a promising direction. 

Many other questions can be asked about the behavior of the critical curves  $\lambda_{c}(t)$ and $t_{c}(\lambda)$. A natural starting point concerns the strict monotonicity of these quantities:

\begin{conjecture}
For all transitive graphs of superlinear growth      $t \mapsto \lambda_{\mathrm{c}}(t)$ and $\lambda \mapsto t_{\mathrm{c}}(\lambda)$ are continuous and strictly decreasing.
\end{conjecture}

We point out that, combined with Theorem \ref{mainsharpness}, the above is a consequence of the continuity of the phase transition. A possible approach to directly prove this on $\mathbb{Z}^{d}$ is to develop the technique of enhancements originally introduced in \cite{enhancements}. Again, this seems possible when we only allow finite-range connections, but we do not know how to extend the results to arbitrary-length connections, which exemplifies the importance of developing a formal mechanism to compare them. Being able to make such a comparison would also help in understanding other properties of Bernoulli percolation, such as locality, which was recently established in \cite{easohutchcroft}. In this sense, we also expect that on $\mathbb{Z}^{d}$, the critical dimension agrees with the case of Bernoulli percolation (see for instance \cite{hutchcroftcriticaldim}), and we believe that the techniques of lace-expansion can be developed to show that the critical exponents agree with the ones of Bernoulli percolation when $d\ge 6.$


Another very interesting direction of study is to extend the results of existence and sharpness to the more general class of models $B+A \rightarrow2B$ (see, for instance, \cite{kesten_vladas,kesten_vladas_shape,kesten_vladas_spread}), which includes the frog model. 
We consider the following version of this model, which we call the $B+A\rightarrow2B$ model ``with death," and more closely resembles our work: initially, there is a single particle of type $B$ at the origin, interpreted as infected. 
Additionally, each site has $\text{Po}(\lambda)$ particles of type $A$. 
Particles of type $A$ move according to independent continuous-time simple random walks with rate $D_{A}$ with infinite lifespan, and the particles of type $B$ do the same but with rate $D_{B}$ up to a finite time $t,$ after which they die.
When a particle of type $B$ collides with a particle of type $A$, the latter instantaneously turns into a particle of type $B$, which is eventually removed from the system after its death. 
It is clear that the frog model is a specific case of the above (with a planted particle at the origin) when $D_{A}=0.$ However, extending our results to this level of generality appears to be extremely challenging, as the $B+A\rightarrow 2B$ model is not \textbf{Abelian}, a key property of the frog model that we discuss at the beginning of Section \ref{preliminaries}. We thus state the following question, which we believe to be of fundamental importance (see \cite{dauvergne2022sirmodelmovingpopulation} for a recent proof of the existence of the phase transition on $\mathbb{Z}^{d}$ for a similar version of this model when $D_{A}=D_B$): 

\begin{question}
    Can Theorems \ref{main:existence}, \ref{main_quasi:existence}, and \ref{mainsharpness} be extended to the more general $B+A\rightarrow2B$  model with death?
\end{question}

Another interesting direction of study, is the following contrasting behavior of amenable and non-amenable graphs when the particles have infinite lifespans:
\begin{question}
  For amenable transitive graphs is it true that conditioned on $\eta_{\mathbf{0}}>0$ one has that $\mathcal{C}_{\mathbf{0}} = V,\mathbb{P}_{\lambda,\infty}-$almost surely? And is it also true that for non-amenable unimodular vertex-transitive graphs, when the lifespan is infinite, the model always undergoes two phase transitions? That is, below $\lambda_{\mathrm{c}}(\text{recurrence}) $ it is transient, i.e., the origin is almost surely visited only finitely many times. Between $\lambda_{\mathrm{c}}(\text{recurrence})$ and $\lambda_{\mathrm{c}}(\text{connectedness})$ it is null-recurrent on $\eta_\mathbf{0} > 0$ and $\inf_{v\in V} \bP_{\lambda,\infty}(0 \rightarrow v) = 0 .$ Above $\lambda_{\mathrm{c}}(\text{connectedness})$ on $\eta_\mathbf{0} > 0$ it is positively recurrent and almost surely $\mathcal{C}_\mathbf{0} = V$. 
\end{question}

The last two predictions suggest that the frog model with infinite lifespan displays a dichotomy between the amenable and non-amenable setups. This is analogous to the following dichotomy, conjectured by Benjamini and Schramm \cite{MR1423907}, to hold for Bernoulli percolation on transitive graphs: There exists a non-uniqueness phase if and only if the graph is non-amenable. A similar dichotomy was established for the somewhat related social network model in \cite{hermon_morris}.

It is standard (see \cite{nina_critical_branching}) that the branching random walk with $\text{Po}(\lambda)$ offspring distribution on unimodular vertex-transitive non-amenable graphs is transient when  $\lambda \rho \leq1,$ where $\rho$ denotes the spectral radius of the walk (see Equation \eqref{def: spectral_radius}). By a straightforward argument, it follows that the same is true for the frog model with infinite lifespan and particle density $\lambda$.

Let us remark on a few variations of the frog model. One obvious variation is to run the frog model in discrete time steps, i.e., all frogs perform an $\ell$-step discrete-time random walk for a deterministic positive integer $\ell$.
The proofs in \cref{subsec_poly} and \cref{sec:nonamenable} carry over for this variation of the model. As a result, we have that for any integer $\ell$ that is sufficiently large (depending on $\lambda$ and the graph $G$), the discrete-time model with lifespan $\ell$ and density $\lambda$ is supercritical. 
Another variant is when each frog has an independent random lifespan $T$.
We consider $T$ to be either positive integer-valued or $(0, \infty)$-valued, where in the first case we consider the discrete-time frog model with $T$-steps and in the second case we consider the continuous-time model with lifespan $T$.
A particularly natural case is a geometric number of steps or an exponential lifespan, but any distribution can be used.
For certain lifespan distributions, the frog model can exhibit phase transitions even when the graph is $\Z$.
The case of $\Z$ was recently studied in detail by Carvalho and Machado \cite{carvalho_machado} when the number of steps is given by a geometric distribution with parameter $p$ distributed according to some $\pi$.


\begin{proposition}
  Let $G$ be a graph such that for all $\lambda>0$ we have $t_c(\lambda)<\infty$.
  Let $\{T_i\}_{i\in I}$ be a family of distributions for the lifespan of the frogs, that is, for a fixed $i\in I$ each frog performs a continuous-time simple random walk up to time $T_{i},$ and assume $\{T_i\}_{i\in I}$ is not tight.
  Then, for every $\lambda$ there is some $T_i$ so that the frog model with lifespans $T_i$ is supercritical.
\end{proposition}

In particular, if $T_i$ is stochastically increasing in $i,$ then the model is supercritical for $i>i_c(\lambda)$.
This includes exponential or geometric lifespans, among others.

\begin{proof}
  Since $\{T_i\}_{i\in I}$ is not tight there exists $\varepsilon>0$ such that for all $M$ there is some $i \in I$ with $\P(T_i\geq M) \geq \varepsilon$.
  Pick any $M > t_c(\lambda\cdot\varepsilon)$.
  In the frog model with lifespan distribution $T_i$, a fraction $\varepsilon$ of the frogs have lifespan at least $M$.
  Thus, the model with random lifespans dominates the frog model with fixed lifespan $M$ and $\lambda\varepsilon$ frogs on average at each vertex, which is supercritical.
\end{proof}

While this is not an explicit characterization of the phase transition as the authors of \cite{carvalho_machado} did for the case of $\mathbb{Z}$, it implies positive probability of survival on general graphs where we know that $t_{c}(\lambda)<\infty$ for the particular family $\text{Beta}(\alpha,\beta)$ of distributions they consider for $\pi$ (the above proof works with minor modifications when we consider a random number of steps as they do). Of course, such a characterization as in their work should not be expected at the level of generality we work on, and the particular case of $\mathbb{Z}$ is of less interest in our context. However, for those interested in studying graphs satisfying $p_{c}(G)= 1$, we state the following additional conjecture: With fixed lifespans for the frogs, we conjecture that the following holds.

\begin{conjecture}\label{conj: p_c_one}
  Let G be a graph of bounded degree and assume that $p_{\mathrm{c}}(G)=1$.
  Then, for any $\lambda,t>0$ we have $t_{\mathrm{c}}(\lambda)= \infty,$ and $\lambda_{\mathrm{c}}(t)=\infty.$
\end{conjecture}

Although Theorems \ref{main:existence} and \ref{main_quasi:existence} can be proven with minor modifications for random lifetimes, Theorem \ref{mainsharpness} is a bit more delicate. In particular, one can naturally extend the definition of $\phi_{\lambda,t}(S)$ for other stochastically increasing families of distributions. Say we do that for $\text{Geom}(p)$. Although the supercritical case does not change much, since the treatment of differential inequalities is analogous to our case,  the current comparison of $\phi_{\lambda,p}(S)$ with $\widetilde{\phi}_{\lambda,p}(S)$ would require a different strategy, as the current proof relies on the fact that the tail of a Poisson distribution decays faster than any exponential. However, as of now, we do not know how to overcome this problem, and so we leave it as a conjecture:
\begin{conjecture}
Consider the version of the frog model with geometric lifetimes. Then, the sharpness of the phase transition holds in the same sense as in Theorem \ref{mainsharpness}.
\end{conjecture}

We finish this discussion by commenting on the remaining cases of Conjecture \ref{conj: exist_main}. These include amenable graphs of exponential growth and graphs of intermediate growth. 
Our current methods do not extend to either of these, mirroring similar difficulties for Bernoulli percolation, where quite a few different strategies were developed to confirm the $p_{c}(G)<1$ conjecture for transitive graphs of superlinear growth.

A promising route for amenable quasi-transitive graphs of exponential growth is adapting the proof of \cite{HUTCHCROFTexpgrowthamenable}: It is not hard to show that for a quasi-transitive graph $G$ of exponential growth, and for any $\lambda>0$ we have that  $\kappa \le \mathrm{gr}(G)^{-1}$, where $\mathrm{gr}(G):=\lim_{n \to \infty} g(n)^{1/n}$ is the growth constant of $G$, and $\kappa:=\lim_{n \to \infty}\kappa(n)^{1/n}$, where 
$\kappa(n):=\min\{\mathbb{P}_{\lambda,t_{\mathrm{c}}(\lambda)}[x \to y] : x,y\in V, d_G(x,y) \le n \}$. To conclude the result, we would need to derive a contradiction by proving that if there is a positive probability of having an infinite cluster at $(\lambda,t_{c}(\lambda)),$ then $\inf_{x,y} \bP_{\lambda,t_{c}(\lambda)}(x \rightarrow y)>0 .$ In the Bernoulli percolation setting, this is a consequence of the Burton-Keane Theorem for amenable graphs. However, the standard Burton-Keane argument fails since, in the current setup, we consider directed percolation, and so, one would need to find an alternative proof for the inequality, which we believe to hold. The remaining case to observe is the intermediate growth regime for which we do not have many insights on how to approach.

\section{Preliminaries}\label{preliminaries}

In this section, we establish some basic properties of the frog model and prove some of the simpler components of our main theorems. 
In particular, we establish here the existence of the phase transition in $\lambda$ for a large class of graphs.

\paragraph{The Abelian property.}
We begin by describing a fundamental property of the frog model, commonly referred to as the \emph{Abelian property}. 
Let $\{\eta_x\}_{x \in V}$ denote the number of particles initialized at each vertex.  
For each particle, we independently sample an infinite discrete-time simple random walk trajectory starting from its initial position and an independent Poisson random variable with mean $t$. 
For $x \in V$ and $i \in \{ 1, \ldots, \eta_x \}$, we denote the discrete-time random walk trajectory of $\omega_x^i$ by $X_x^i$ and we denote by $N_x^i(t)$, a Poisson random variable with mean $t,$ the number of jumps performed by the particle $\omega_{x}^{i}$ within its lifespan $t$. Then, the sequence of vertices traversed by $\omega_x^i$ (allowing repetition if the graph contains a loop)  is given by
\[
\bigl( X_x^i(k) \bigr)_{k=0}^{N_x^i(t)}.
\]

Fix a vertex set $D \subseteq V$ and declare only the particles initialized in $D$ to be active at time $0$. For each particle $\omega_x^i$ with $x \in D$ and $i \in \{ 1, \ldots, \eta_x \}$, we declare $X_x^i(0)$ as \emph{used}. A \emph{legal operation} consists of selecting an active particle $\omega_x^i$ which has not yet reached its final position, considering its next unused position $X_x^i(k{+}1)$ (if no next unused position exists, then no legal operation involving $\omega_x^i$ is possible), and then performing the following steps:
\begin{itemize}[nosep]
    \item move $\omega_x^i$ from its current position $X_x^i(k)$ to $X_x^i(k{+}1)$; The reached $X_x^i(k+1)$ is now considered used.
    \item if $X_x^i(k{+}1)$ contains inactive particles, activate all such particles.
\end{itemize}
For a sequence of legal operations $\mathbf{o} = (o_1, o_2, \ldots)$, let $A_j(\mathbf{o})$ denote the set of active particles after the first $j$ legal operations have been performed, with the initial condition
\[
A_0(\mathbf{o}) := \{\omega_x^i : x \in D,\ 1 \le i \le \eta_x\}.
\]
Since each legal operation can only activate additional particles (never deactivate them), we have $A_j(\mathbf{o}) \subseteq A_{j{+}1}(\mathbf{o})$ for any $j \ge 0$. We define the set of all eventually activated particles as
\[
A_\infty(\mathbf{o}) := \bigcup_{j \ge 0} A_j(\mathbf{o}).
\]

For a particle $\omega$, let $c_\omega(j; \mathbf{o})$ denote the number of times $\omega$ has been selected during the first $j$ legal operations of $\mathbf{o}$.  
We call $\mathbf{o}$ \emph{exhaustive} if every particle $\omega_x^i$ that ever becomes active satisfies
\[
\lim_{j \to \infty} c_{\omega_x^i}(j; \mathbf{o}) = N_x^i,
\]
i.e., every active particle will eventually be selected for exactly $N_x^i$ legal operations.

The Abelian property of the frog model states that for any initially active set $D \subseteq V$, and any two exhaustive sequences of legal operations $\mathbf{o}$ and $\mathbf{o}'$, we have
\[
A_\infty(\mathbf{o}) = A_\infty(\mathbf{o}').
\]
In particular, the set of vertices that eventually become activated is independent of the order in which legal operations are performed. 
We refer the reader to, e.g. \cite{VladasRolla} for a detailed account of the Abelian property for activated random walks.

This property is essential for the arguments in this paper: various exploration procedures will be constructed to reveal the set of eventually activated vertices, and the order in which legal operations are performed can be chosen purely for convenience, without affecting the outcome. 
In particular, instead of running the frogs simultaneously in continuous time, we can pick one active frog at a time, chosen in any way we like, and reveal its entire trajectory.

\paragraph{The subcritical phase.}
As in \cite{FrogArt}, a comparison with a subcritical Galton--Watson branching process shows that for any $\lambda>0$ and $t>0$ with $\lambda t \le 1$, the process dies out almost surely. 
We now state and prove the following proposition, which in particular implies item~(1) of Theorem~\ref{main:existence}. 
Let $\mathcal{T}$ be the total number of particles that are eventually activated, i.e. $\mathcal{T} := \sum_{x} |\{ \omega_x^i : 1 \le i \le \eta_x, \mathbf{0} \rightarrow x\}|.$

\begin{proposition}\label{notzero}
    For any graph $G$ and any positive parameters $\lambda$ and $t$ satisfying $\lambda t \le 1$, we have $\bP_{\lambda,t}(\mathbf{0}\rightarrow\infty)=0$. Moreover, when $\lambda t < 1$, there exists $s > 1$ such that $\bE_{\lambda, t}[s^{\mathcal{T}}] < \infty.$
\end{proposition}

\begin{proof}  
    Let $\zeta_x := \sum_{i=1}^{\eta_x} N_x^i(t)$ be the total number of jumps made by all particles initially located at $x$.  
    The distribution of $\zeta_x$ is the same for every $x\in V$, and we denote this common law on $\N$ by $\mu$.
    We now describe an exploration process.  
    In the first stage, set $A_0 := \{\mathbf{0}\}$ and mark $\mathbf{0}$ as \emph{revealed}.  
    In stage $k\ge 1$, define $A_k$ to be the set of vertices that are \emph{unrevealed} (i.e., not in $\cup_{i=0}^{k-1} A_i$) and that lie in the range of at least one particle initially located at a vertex in $A_{k-1}$.  
    Then mark all vertices in $A_k$ as revealed.
    It is straightforward to check that the sequence $\{ |A_k| \}_{k\ge 0}$ is stochastically dominated by a branching process with offspring distribution $\mu$. Since $\mu$ has mean $\lambda t \le 1$, this branching process is either critical or subcritical, and it dies out in either case. Moreover, when $\lambda t < 1$, since for any $s_0 > 0$, we have $\bE_{\lambda, t}[s_0^{\zeta_x}] = \exp(\lambda (\exp(t(s_0 - 1)) - 1))$, Exercise 5.33 of \cite{MR3616205} gives that $\bE_{\lambda, t}[s^\mathcal{T}]$ is finite for some $s > 1$. 
\end{proof}

\paragraph{Finiteness of $\lambda_c(t)$.} Additionally, when the graph $G$ is quasi-transitive with superlinear growth, for any $t > 0$, if we choose $\lambda$ sufficiently large, we can establish via comparison with supercritical percolation that there is a positive probability of survival of the infection when only the origin is initially activated at time $0$. In other words, for any finite $t > 0$, the critical parameter $\lambda_c(t)$ is finite.

Let $\bP_{p}$ denote the measure with respect to the standard Bernoulli edge percolation on a graph $G$. We define
\begin{equation}
    p_{\mathrm{c}}(G) := \inf \left\{ p \in [0,1]: \exists \, x \text{ such that } \bP_{p}( x\leftrightarrow \infty)>0\right\},
\end{equation}
where $x \leftrightarrow \infty$ denotes the event that there exists an infinite self-avoiding path of open edges starting at $x$.
We now restate item $(2)$ of Theorem \ref{main:existence}:

\begin{proposition} \label{superlinear}
    Let $G$ be a graph of bounded degree $\Delta:=\sup_{v \in V} \deg(v) < \infty$
    and $p_{\mathrm{c}}(G)<1$. Then, for any $t>0$, we have $\lambda_{\mathrm{c}}(t) < \infty$.
\end{proposition}

One natural question is why we choose $\text{Po}(\lambda)$ as the initial particle distribution, rather than some other distribution. The primary reason is the convenience provided by the Poisson thinning property. Proposition \ref{superlinear} exemplifies how the proof leverages this property.

\begin{proof}[Proof of Proposition \ref{superlinear}]
We declare an (undirected) edge $\{x,y\}$ to be \emph{open} if the following two conditions hold:  
(1) there exists a particle initially located at $x$ that makes at least one jump during its lifespan $t$, and whose first jump lands at $y$; and  
(2) condition~(1) also holds with the roles of $x$ and $y$ exchanged.  
By Poisson thinning, the family $\{\mathbf{1}_{\{\{x,y\}\text{ is open}\}}\}_{\{x,y\}\in E}$ consists of independent Bernoulli random variables.  
Moreover, it stochastically dominates an i.i.d.\ collection of Bernoulli random variables with parameter
\[
\left(1 - \exp\!\left(-\frac{\lambda (1 - e^{-t})}{\Delta}\right)\right)^{2}.
\]
By choosing $\lambda$ sufficiently large (depending on $t$ and $\Delta$) so that the quantity above exceeds $p_c(G)$, we obtain that  
$\{\mathbf{1}_{\{\{x,y\}\text{ is open}\}}\}_{\{x,y\}\in E}$ stochastically dominates a supercritical Bernoulli percolation.  
This implies that $\lambda_c(t) < \infty$.
\end{proof}

Combining Proposition \ref{notzero} and Proposition \ref{superlinear}, we have a fairly complete picture for the existence of a phase transition with respect to the particle density $\lambda$. However, proving that $t_{\mathrm{c}}(\lambda)<\infty$ for arbitrarily small $\lambda$ is considerably more challenging and may require different proof strategies for various classes of graphs, such as those of superlinear polynomial growth or non-amenable graphs. While in Conjecture \ref{conj: exist_main}, we state that the phase transition exists for any quasi-transitive graph with superlinear growth, this is not the case for graphs with linear growth, as shown in the following proposition, which establishes item $(1)$ of Theorem \ref{main_quasi:existence}.

\paragraph{The linear growth case.} Before proving the proposition, we define a standard notation for hitting times. For any finite set $S\subset V$ and any point $x\in V$, let $(Y_{x}(s))_{s\geq 0}$ be a continuous-time simple random walk starting at $x$. The hitting time $\tau^{x}_{S}:=\inf\{s\geq0: Y_{x}(s) \in S\}$ is the first time the random walk hits the set $S$. We abbreviate this as $\tau^{x}_{y}$ when $S=\{y\}$, and often drop the superscript $x$ when the starting point is clear from the probability measure $\bP_{x}$. 

\begin{proposition}\label{lineargrowth}
    Let $G$ be a quasi-transitive graph with linear growth. Then, for any $\lambda,t>0$ we have $t_{\mathrm{c}}(\lambda) = \infty,$ and $\lambda_{\mathrm{c}}(t) = \infty$ respectively.
\end{proposition}

\begin{proof}
Fix a root vertex for $G$, and let $S_n$ be the sphere $S(n) = B(n+1)\setminus B(n)$.
Since $G$ has linear growth, there exists an infinite sequence of positive integers $n_{j} \rightarrow\infty$ such that for $|S(n_{j})|\leq K$, for some constant $K=K(G)$. We may assume that this sequence satisfies $n_0 \geq \lceil 2t \rceil$ and $n_{j} \geq 2 \cdot n_{j-1}$ for all $j \geq 1$; otherwise, we simply delete some elements from the sequence to achieve this property.

For each $j\geq 1,$ define the event that no particle initially placed in $B(n_{j-1})$ exits the ball $B(n_{j})$ within its lifespan $t$: 
\begin{equation}
    F_{j} := \left\{ \mathrm{R}_x^i \subseteq B(n_{j}) \text{ for all } x \in B(n_{j-1}) \text{ and all } i \in \{ 1, \ldots, \eta_x \} \right\},
\end{equation}
where recall that $\mathrm{R}_x^i$ is the set of vertices traversed by the particle $\omega_x^i$ during its lifespan $t$. 
Furthermore, define the event that no particle in the annulus $B(n_{j})\setminus B(n_{j-1})$ exits $B(n_{j})$:
\begin{equation}
     D_{j} := \left\{ \mathrm{R}_x^i \subseteq B(n_{j})  \text{ for all } x \in B(n_{j}) \setminus B(n_{j-1}) \text{ and all } i \in \{1, \ldots, \eta_x \} \right\}.
\end{equation}
By Poisson thinning, we have
    \begin{equation*}
        \bP_{\lambda,t}(D_{j}) = \exp\left(- \lambda \sum_{x\in B(n_{j})\setminus B(n_{j-1})} \bP_{x}\big(\tau_{B(n_{j})^{c}}\leq t\big)\right)
    \end{equation*} 
for every $j\geq 1$. We observe that the above sum is upper bounded by:
    \begin{align} \label{eq: linear_growth_equation_1}
        \sum_{x\in B(n_{j})\setminus B(n_{j-1})} \bP_{x}(\tau_{B(n_{j})^{c}}\leq t) \leq \sum_{\substack{x\in B(n_{j})\setminus B(n_{j-1})\\ y \in \partial B(n_{j})}} \bP_{x}(\tau_{y} \leq t). 
    \end{align}
Since $G$ is quasi-transitive, then $\bP_x(\tau_y \le t) \le K_1 \bP_y(\tau_x \le t+1)$ for some constant $K_1$ depending on the graph. Hence,
    \begin{align*}
        \eqref{eq: linear_growth_equation_1} &\leq K_1 |\partial B(n_{j})| \max_{y \in \partial B(n_{j})} \sum_{x\in B(n_{j})\setminus B(n_{j-1})} \bP_{y}(\tau_{x}\leq t+1) \\
        &\leq K_2 \max_{y} \sum_{x \in V} \bP_{y}(\tau_{x}\leq t+1) \leq K_2 \cdot (t+1),
    \end{align*}
where $K_2 = K_2(G)$ is a positive constant, and in the second inequality, we used that $|\partial B(n_j)| \le K$ holds for all $j\geq1$.
Hence, $\inf_{j} \bP_{\lambda,t}(D_{j}) > 0$. Since the events $\{D_{j}\}$ are independent, by the second Borel-Cantelli lemma, we conclude that $\bP_{\lambda,t}(\bigcap_{n=1}^{\infty} \bigcup_{j=n}^{\infty} D_{j}) = 1$. 

Now, for $F_{j}$, note that
\begin{equation*}
    \bP_{\lambda,t}(F^{c}_{j}) = 1 - \exp\left(- \lambda \sum_{x\in B(n_{j-1})} \bP_{x}(\tau_{B(n_{j})^{c}}\leq t)\right).
\end{equation*}
Furthermore, observe that
\begin{align*}
    \sum_{x\in B(n_{j-1})} \bP_{x}\left( \tau_{B(n_{j})^{c}}\leq t \right) \leq |B(n_{j-1})| \, \bP(\mathrm{Po}(t) > n_{j-1}).
\end{align*}
Note that there exists $K_{3}(G)>0$ such that $|B(n_{j})|\leq K_{3}n_{j}$ for every $j\geq 0$. By Chernoff's inequality, for every $L> 0$, we have
\begin{equation}\label{chernoffupper}
\bP\left( |\mathrm{Po}(t)-t |\geq L \right) \leq 2\exp\left( - \frac{L^{2}}{2(t+L)}\right).
\end{equation}
Also, by our choice of $n_j$'s, we have $n_j \ge 2t$ for any $j \ge 0$.
Therefore, we conclude that 
\begin{align*}
&\bP_{\lambda,t}(F_{j}^{c}) \leq \lambda\sum_{x\in B(n_{j-1})} \bP_{x}(\tau_{B(n_{j})^{c}}\leq t) \\
\leq \,&\lambda \, K_{3} \, n_{j-1} \exp\left(-  \frac{n_{j-1}^{2}}{4(2t+n_{j-1})} \right).
\end{align*}
It is straightforward to verify that the last quantity above is summable over $j \ge 1$. By the first Borel-Cantelli lemma, we have $\bP_{\lambda,t} ( \, \cup_{n = 1}^{\infty} \cap_{j = n}^{\infty} F_{j} \, )=1$. 
Combining both conclusions, we obtain that for any $\lambda,t>0$, we have $\bP_{\lambda,t}(\mathbf{0}\rightarrow\infty ) = 0$, which concludes the proof.
\end{proof}


\paragraph{Preliminary discussion and notation for sharpness.} Another important question is whether the phase transition exhibits sharpness with respect to $\lambda$ and $t$. 
Inspired by the elegant proof of sharpness in \cite{MR3477351}, we demonstrate in Section \ref{sharpness} that this is indeed the case for any directed vertex-transitive graph $\vec{G}.$
The proof follows a similar approach to that work by considering a local criterion for criticality, which is then proved to agree with the usual definitions. 
To provide a lower bound for the percolation probability in the supercritical regime, the authors of \cite{MR3477351} derived a differential inequality for the probability that the origin connects to the boundary of a ball of radius $n$, and we aim to do the same.

However, the frog model presents additional technical difficulties, as the directed connections between vertices, $\{ x \rightarrow y \}_{x, y \in V}$, are not independent. 
For instance, models with correlations, such as the Potts model, require more sophisticated techniques to prove exponential decay of correlations in the subcritical regime (see \cite{sharpnesspotts}).
The introduction of two critical parameters --- one for particle lifespan and another for density --- also presents an interesting challenge for the frog model, requiring the derivation of distinct differential inequalities: one with respect to $\lambda$ and another with respect to $t$ for the study of the supercritical regime.

It is worth noting that the argument of \cite{MR3477351} only establishes exponential tail bounds for the intrinsic radius of the cluster in Bernoulli percolation, whereas in the present work, we establish exponential tail bounds for the cluster size of the frog model. 
We point out that other works have proved exponential tail estimates for the cluster size in Bernoulli percolation using stochastic domination techniques, e.g.  \cite{AHL_2025__8__101_0}. 
We also note that our argument is proven for directed vertex-transitive graphs $\vec{G}$. We now provide a few key definitions used in the proof of \cref{mainsharpness}.

Let $S \subseteq V$ be a finite set.
For each $x \in S$, we define $A_{S}(x):=\{1\leq i\leq \eta_{x} : \mathrm{R}_{x}^{i}\cap S^{c} = \varnothing\}$, 
This records the indices of particles starting at $x$ whose trajectories remain inside $S$ throughout their lifespan. 
Moreover, we define 
\begin{align} \label{def: modified_arrow}
\mathcal{R}^{S}_{x}:=\begin{cases}
            \bigcup\limits_{i\in A_S(x)}\mathrm{R}_{x}^{i}, & \text{if $A_S(x) \neq \varnothing$}\\
            \; \; \, \{x\}, & \text{otherwise}.
         \end{cases}
\end{align}
We write $x \xRightarrow{S} y$ if $y \in \mathcal{R}^{S}_{x}$. Then, we say that $x$ activates $y$ internally to $S$, denoted $x \xrightharpoonup{S} y$ if 
\begin{center}
there exists a finite sequence of vertices $x_0, x_1, \ldots, x_{m}$ in $S$ \\
such that $x_0 = x$, $x_m = y$, and $x_{i} \xRightarrow{S} x_{i+1}$ for every $0 \leq i \leq m-1$.
\end{center}
Note that for the event $x \xrightharpoonup{S} y$, we require that all particles involved in the chain of activations remain within $S$ throughout their entire lifespan $t$.
In particular, this event is not monotone in $t$, since particles living longer are more likely to exit the set $S$.
We say $x \xrightharpoonup{S} A$ for $A \subseteq S$ if there exists $y \in A$ such that $x \xrightharpoonup{S} y$. We emphasize that the event $x \xrightharpoonup{S} y$ plays a crucial role in our derivation of the sharpness result.
One should compare the above definition with the previously defined event $x \xrightarrow{S} y$, which is a less restrictive event, as it only requires the existence of a chain of activations from $x$ to $y$ whose vertices are in $S$, but the particles involved in the activation may leave $S$ and return inside. \footnotemark
\footnotetext{We note that these definitions work exactly as stated for directed graphs, since the events $x\xrightarrow{S}y$, $x\xrightharpoonup{S}y,$ and $\tau_{S^{c}}^{x} \leq t$ implicitly account for the graph structure, or alternatively, the random walk jump kernel on $\vec{G}.$ }

Throughout the rest of this paper, we require that $S \subseteq V$ be finite and contain the origin $\boldsymbol{0}$. We define  
\begin{align}\label{phiformula}
    \phi_{\lambda, t}(S) := \sum_{x \in S } \lambda \, \bP_x(\tau_{S^c} \leq t) \, \bP_{\lambda, t}(\mathbf{0} \xrightharpoonup{S} x). 
\end{align}
This quantity can be verbally described in the following way: it is the expected number of \textbf{particles} (thus, the right-hand side is multiplied my $\lambda$) satisfying two conditions: (1) the initial position of the particle, denoted by $x$, is activated by a chain of trajectories that remain completely contained in $S$, i.e., $\boldsymbol{0} \xrightharpoonup{S} x$, and (2) the particle exits the set $S$ within its lifetime $t$. That is, we have $\phi_{\lambda,t}(S) = \bE_{\lambda,t}\big[\big|\mathfrak{N}(S)\big | \big]$, where 
$$\mathfrak{N}(S) := \bigcup_{x \in S} \, \left\{ \omega_{x}^{i} : 
1 \leq i \leq \eta_x,\,
\mathbf{0}\xrightharpoonup{S}x,\,  
\mathrm{R}_{x}^{i} \cap S^{c} \neq \varnothing
\right\}.$$

When considering transitive directed graphs in the proof of \cref{mainsharpness}, let $\Delta$ denote the outer degree, that is, $\Delta := |\{y \in V: (\mathbf{0}, y) \in \vec{E}\}|$. 
Let $c = c(\Delta, \lambda, t) := 1/C(\Delta, \lambda, t) \in (0, 1)$, where $C(\Delta, \lambda, t)$ is the function introduced in Lemma \ref{constantcomparison}.

The explicit form of this function is not important; however, we note that the for a fixed $\Delta,$ the map $(\lambda, t) \mapsto c(\Delta, \lambda, t)$ satisfies three properties that are used in Section \ref{sharpness}: it is continuous on $[0, \infty) \times [0, \infty)$, strictly positive on $(0,\infty) \times (0,\infty),$ and takes value $0$ at $t=0$ or $\lambda=0.$

We introduce two new critical parameters:
\begin{align*}
    \tilde{\lambda}_c(t) &:= \sup\big\{\lambda \geq 0: \inf_{\substack{S \ni \mathbf{0}  \\ S \text{ finite }}} \phi_{\lambda, t}(S) \leq \mathrm{c}(\Delta,\lambda,t) \big\}, \\
    \tilde{t}_c(\lambda) &:= \sup\big\{t \geq 0: \inf_{\substack{S \ni \mathbf{0}  \\ S \text{ finite }}} \phi_{\lambda, t}(S) \leq \mathrm{c}(\Delta,\lambda,t) \big\}.
\end{align*}

Our goal, as in \cite{MR3477351}, is to prove the sharpness of the phase transition for the frog model with respect to the parameters $\tilde{\lambda}_{c}(t)$ and $\tilde{t}_{c}(\lambda)$. This, in turn, will imply that $\lambda_{\mathrm{c}}(t) = \tilde{\lambda}_{c}(t)$ and $t_{\mathrm{c}}(\lambda)=\tilde{t}_{c}(\lambda)$, thus establishing the result. One of the key ideas is to differentiate the connection probability $\bP_{\lambda,t}(\mathbf{0} \rightarrow \Lambda^{c})$ with respect to $\lambda$ and with respect to $t$ separately for any given finite set $\Lambda$. \par

\begin{remark}
(1) In principle, the quantities $\tilde{\lambda}_{c}(t)$ and $\tilde{t}_{c}(\lambda)$ take values in $[0,\infty]$.  
If $\tilde{\lambda}_{c}(t)=\infty$, our proofs show that the model is always subcritical when the lifespan is chosen as $t$;  
conversely, if $\tilde{\lambda}_{c}(t)=0$, then the model is always supercritical, but as a consequence of Proposition \ref{notzero}, which also holds for directed graphs, we will show that $\tilde{\lambda}_{c}(t)\geq 1/t$. 
The same conclusions hold for $\tilde{t}_{c}(\lambda)$.  
In all of these boundary cases, the frog model exhibits no phase transition.



\smallskip

(2) By a coupling argument, one can verify that the map
\(\lambda \mapsto \phi_{\lambda,t}(S)\) is non-decreasing for all $t$ and $S$.
In contrast, we believe that the map
\(t \mapsto \phi_{\lambda,t}(S)\) 
is not non-decreasing in \(t\).
For a fixed set \(S\), as \(t \to \infty\), with high probability every particle initially inside \(S\) exits \(S\) during its lifetime \(t\).
As a result, only particles starting at the origin contribute to \(\phi_{\lambda,t}(S)\), and consequently \(\phi_{\lambda,t}(S)\) converges to \(\lambda\).
On the other hand, for a fixed supercritical value of \(t\), we expect that \(\phi_{\lambda,t}(S)\) is proportional to the size of the boundary of \(S\).
In particular, by choosing \(S\) sufficiently large, one can ensure that \(\phi_{\lambda,t}(S)\) exceeds \(\lambda\).
This heuristic is motivated by known results on the isoperimetric profile of infinite clusters in supercritical Bernoulli percolation (see, for instance, \cite{MR4243018, MR4536689}).
We note that while it is
certainly plausible that the map
\(t \mapsto \inf_{S} \phi_{\lambda,t}(S)\) is non-decreasing,
our arguments do not rely on such monotonicity.

Regardless of whether such monotonicity holds, the quantities \(\tilde{t}_{c}(\lambda)\) and \(\tilde{\lambda}_{c}(t)\) are well defined (possibly equal to \(0\) or \(\infty\)).
For every \(t > \tilde{t}_{c}(\lambda)\) or \(\lambda > \tilde{\lambda}_{c}(t)\), by definition we have that
$\inf_{S} \phi_{\lambda,t}(S) > c(\Delta,\lambda,t)$.
When \(t < \tilde{t}_{c}(\lambda)\) or $\lambda < \tilde{\lambda}_{c}(t) $, we have no a priori information on whether
\(\inf_{S} \phi_{\lambda,t}(S)\) lies above or below \(c(\Delta,\lambda,t)\).
This delicate regime is addressed in \cref{sec:subcritical}. 

If \(t < \tilde{t}_{c}(\lambda)\) (respectively, $\lambda < \tilde{\lambda}_{c}(t) $) then there exists $t' \in [t,\tilde{t}_{c}(\lambda))$ such that $\inf_{S} \phi_{\lambda,t'}(S) \le c(\Delta,\lambda,t')$ (respectively, $\lambda' \in [\lambda,\tilde{\lambda}_{c}(t))$ such that $\inf_{S} \phi_{\lambda',t}(S) \le c(\Delta,\lambda',t)$). We use this to establish \eqref{finitesus2} for $(\lambda,t')$ (respectively, \eqref{finitesus} for $(\lambda',t)$), which by monotonicity also hold at $(\lambda,t)$.
As a consequence, the exponentially decaying tails in \eqref{finitesus} and \eqref{finitesus2} imply that $\inf_{S} \phi_{\lambda,t}(S) = 0$,
for all $t<\tilde{t}_{c}(\lambda)$ and $\lambda < \tilde{\lambda}_{c}(t)$,
which rules out any fluctuation of $\inf_{S} \phi_{\lambda,t}(S)$ around the curve \(c(\Delta,\lambda,t)\). 

Something similar happens in~\cite{MR3477351,duminilzedd}\footnote{Recall that the authors of~\cite{MR3477351,duminilzedd} defined
$\varphi_{p}(S) := p \sum_{xy \in \partial_{E} S} \bP_{p}(0 \xleftrightarrow{S} x)$,
where $\partial_{E} S$ denotes the edge boundary of $S$, and defined $\tilde{p}_{c} := \sup \{ p\in [0,1]: \inf_{\substack{S \ni \mathbf{0}  \\ S \text{ finite }}} \varphi_{p}(S) <1 \} $.} with their proof implying that for $p< \tilde{p}_{c},$ the infimum of $\varphi_{p}(S)$ over finite $S$ is equal to zero. However, we remark that the choice of a function $c(\Delta,\lambda,t)$ with $\lim_{\lambda \to 0}c(\Delta,0,t)=0$ instead of the constant $1$ as in \cite{duminilzedd} is not merely an artifact of the proof. Indeed, even if in the definition of $\phi_{\lambda,t}(S)$ we allow the trajectories along the chain of infection to exit $S$, we would have that it is at most $\lambda |S| \to 0$, as $\lambda \to 0$.

\end{remark}

Before we finish this section, we will present a brief overview of the proofs of the main results as well as discuss the related literature. 

The first sharpness result for percolation was proven initially by \cite{aizenman1987sharpness} and \cite{mensikov1986coincidence} for Bernoulli percolation on $\mathbb{Z}^{d}.$ However, our proof draws more inspiration from the recent works of Duminil-Copin and Tassion (see \cite{MR3477351,duminilzedd}). Indeed, as previously observed, the idea of introducing the new critical parameters $\tilde{\lambda}_{c}(t)$ and $ \tilde{t}_{c}(\lambda) $ is directly inspired by their auxiliary quantity $\tilde{p}_{c}$. However, the frog model presents extra technicalities that require new ideas. More fundamentally, unlike Bernoulli percolation, the probabilities of connections in the frog model are not independent. As in \cite{MR3477351}, we analyze two regimes: the supercritical when $\lambda>\tilde{\lambda}_{c}(t)$ (respectively $t > \tilde{t}_{c}(\lambda)$) and the subcritical when $\lambda < \tilde{\lambda}_{c}(t)$ (respectively $t < \tilde{t}_{c}(\lambda)$). In the analysis, we derive that $\lambda_{c}(t)=\tilde{\lambda}_{c}(t)$ (respectively $t_{c}(\lambda) = \tilde{t}_{c}(\lambda)$), which implies the sharpness result. 

To study the supercritical regime, we derive novel differential inequalities for $\bP_{\lambda,t}(\mathbf{0}\rightarrow\Lambda^{c})$, where $\Lambda$ is an arbitrary finite set. The differential inequalities are now taken with respect to $\lambda$ and $t$ separately, which are somewhat different, but still, via a derivation analogous to the Margulis-Russo formula (for Poisson and Bernoulli random variables), we will compare the resulting quantities to the infimum over $\phi_{\lambda, t}(S) $ for finite sets, for which we apply our hypothesis $\lambda > \tilde{\lambda}_{c}(t)$ (and respectively $t > \tilde{t}_{c}(\lambda)$).

In the subcritical regime, we define an exploration process which, under the assumption that
\(\phi_{\lambda,t}(S) < c(\Delta,\lambda,t)\) for some finite set \(S\), is stochastically dominated by a subcritical branching process.
The construction of this exploration process is a non-trivial generalization of the one employed in the Bernoulli percolation literature.
We emphasize several key features of our exploration process.
First, the offspring are particles rather than vertices.
Second, offspring particles may have initial positions either inside or outside \(S\); in contrast, in Bernoulli percolation, the offspring consist solely of vertices in the external vertex boundary of \(S\) (see Section~2.1 of \cite{duminilzedd}).

In the definition of \(\phi_{\lambda,t}(S)\) (see \eqref{phiformula}), we consider only chains of infection that are strictly contained in the set \(S\).
In other words, \(\phi_{\lambda,t}(S)\) counts the number of ``first-exit trajectories/particles", namely particles that exit \(S\) and are reached by a chain of particles whose trajectories remain entirely within \(S\).
We denote this collection of particles by \(\mathfrak{N}(S)\).
The offspring of the first generation are then chosen to be the particles whose initial positions lie along the trajectories in \(\mathfrak{N}(S)\).
This quantity is later denoted by \(\widetilde{\phi}_{\lambda,t}(S)\).
We will show that
$\widetilde{\phi}_{\lambda,t}(S) < \phi_{\lambda,t}(S)/c(\Delta,\lambda,t)$. Whenever \(\phi_{\lambda,t}(S) < c(\Delta,\lambda,t)\),  the expected number of offspring of each ancestor is less than 1 and the resulting branching process is subcritical.
 
The function $c(\Delta,\lambda,t)$ is defined so that $1/c(\Delta,\lambda,t)$ is a uniform upper bound, over all finite sets $S \subset V$, on $\lambda$ times a weighted average (over starting vertices that are activated internally in $S$) of the expected number of jumps made by a random walk up to time $t$, conditioned on walk exits $S$ by time $t$.
The inequality $\widetilde{\phi}_{\lambda,t}(S) < \phi_{\lambda,t}(S)/c(\Delta,\lambda,t)$ is the only place in the proof in which we rely on the assumption that the particle performs a simple random walk on a transitive digraph of degree $\Delta$, rather than using an arbitrary transitive jump kernel $P$ as in Conjecture \ref{conj:transitive P}.

The proof will be presented in two subsections in Section \ref{sharpness}, divided into the supercritical and subcritical phases. 
We highlight that our result does not rely on reversibility and is proved for directed vertex-transitive graphs. 
Theorem \ref{mainsharpness} is the first result of its kind for the frog model to the knowledge of the authors. 


In \cref{sec: existence}, we prove that for any $\lambda>0,$ we have that $t_{c}(\lambda)<\infty$ for quasi-transitive graphs of superlinear polynomial growth, and non-amenable graphs of bounded degree. These two classes of graphs exhibit drastically different behaviors, so we divide our proof into two subsections, each dedicated to one setup.

In the first subsection, we focus on quasi-transitive graphs of superlinear polynomial growth. 
In this setup, our proof relies on a renormalization structure for quasi-transitive graphs of polynomial growth similar to the one developed in \cite{MR4529920}. The renormalization structure is a consequence of classical results on quasi-transitive graphs of polynomial growth due to Trofimov \cite{Trofimov1985GRAPHSWP} and the fact that there exists a surjective group homomorphism from any nilpotent group to $\mathbb{Z}^{2}$.
When the graph is of superlinear growth, we can define a ``renormalized graph" such that the Bernoulli site percolation on it has a non-trivial phase transition. 
For a fixed $\lambda>0$, we look at sufficiently large balls around each vertex on the renormalized graph and prove that there are many ``good" frogs inside the ball, in the sense that they activate a constant fraction of the ball. Moreover, we show that with high probability, once a constant fraction of the ball is activated, all the ``neighboring" balls in the renormalized lattice are also activated. In this way, we can compare the frog model with the Bernoulli site percolation on the renormalized graph and conclude that the frogs survive whenever the percolation cluster is infinite.

In \cref{sec:nonamenable}, we define an exploration argument to prove positive probability of survival for a large enough lifetime $t$. 
This uses an idea introduced in \cite{localityconjecture}, which lower bounds the probability of escaping a set in terms of the spectral radius of the random walk on $G$. 
The proof resembles the proof of existence of a phase transition for social networks from \cite{hermon_morris}. 
Essentially, we show that at any stage of the process, there are some frogs that are likely to reach previously unvisited vertices, and thereby activate additional frogs. We now begin the proof of our main results.

\section{Sharpness of the Phase Transition}\label{sharpness}
In this section, we establish the sharpness of the phase transition for the frog model on all directed vertex-transitive graphs. This result has its full strength only when a phase transition is present, since otherwise one of the two regimes becomes vacuous. 
As discussed in \cref{preliminaries}, for every $t>0$ we have 
$0 < \lambda_{\mathrm{c}}(t) < \infty$
whenever the underlying graph is vertex-transitive and has superlinear growth.  
On the other hand, we conjecture that for every $\lambda > 0$ one has $0 < t_{\mathrm{c}}(\lambda) < \infty,$ but in this paper, we could only verify the latter assertion for more restricted classes of graphs.  
In contrast, the sharpness result we prove here holds for all directed vertex-transitive graphs. 
We emphasize that \cref{mainsharpness} remains valid if $\tilde{\lambda}_{c}$ and $\tilde{t}_{c}$ are used in place of $\lambda_{\mathrm{c}}$ and $t_{\mathrm{c}}$.  
We begin the proof by analyzing the supercritical regime.

\subsection{Supercritical Phase}
Recall the definition of $\phi_{\lambda,t}(S)$ in \eqref{phiformula}. We start with a differential inequality with respect to $\lambda$. 
\begin{lemma}\label{partiallambda}
For any $\lambda, t > 0$ and any finite set $\Lambda \subset V$, we have
\begin{align}
    \frac{\partial}{\partial \lambda} \bP_{\lambda, t}(\mathbf{0} \xrightarrow[]{}\Lambda^c) \ge \frac{1}{\lambda} \inf_{\substack{  S \ni \mathbf{0}   \\ S \subseteq \Lambda}} \phi_{\lambda, t}(S) \left( 1 - \bP_{\lambda, t}(\mathbf{0} \xrightarrow[]{}\Lambda^c) \right) \, .
\end{align}
\end{lemma}

In order to prove this result, we first introduce some notation. 
We restrict ourselves to a finite set of vertices $\Lambda$ and identify all the points in $\Lambda^{c}$ as one, denoted by $\mathcal{O}$. 
The restriction is made since its finiteness facilitates a straightforward application of Russo's formula.
Let $\bar{\Lambda}:=\Lambda\cup \{ \mathcal{O} \}.$ 
We consider the following restriction of the particles' traces to $\bar{\Lambda}:$ for each $x\in V$ and $1\leq i \leq \eta_{x}$, we define

\begin{align} \label{restricted_trace}
\mathrm{R}^{\bar{\Lambda},i}_{x}:=\begin{cases}
            \, \, \mathrm{R}_{x}^{i}, & \text{if $\mathrm{R}_{x}^{i} \cap \Lambda^{c} = \varnothing$,}\\
          (\mathrm{R}_{x}^{i}\cap \Lambda) \cup \{\mathcal{O}\}, & \text{if $\mathrm{R}_{x}^{i} \cap \Lambda^{c} \neq \varnothing$}.
         \end{cases}
\end{align}

Let $\Omega_{\Lambda}:= \big\{\big(x,B\big) : x \, {\in} \, \Lambda,B \, {\subseteq} \, \bar{\Lambda}\big\}.$ Let $Y_{(x,B)}:=|\{ 1 {\le} i {\le} \eta_x : \mathrm{R}^{\bar{\Lambda},i}_{x} \, {=} \, B\}|,$ which is the number of particles at $x$ whose (restricted) traces equal $B.$\footnote{We note that we could have instead considered $Z_{(x,B)}:=\mathbbm{1}_{Y_{(x,B)} \ge 1}$. By Poisson thinning, $(Z_{(x,B)}:(x,B) \in \Omega_{\Lambda})$ is a finite collection of independent Bernoulli random variables, with corresponding parameters $p(x,B,\lambda)=1-\exp(-\lambda \bP_x(\mathrm{R}^{\bar{\Lambda}} = B))$. This means that below we could have defined pivotality as in Bernoulli percolation, and applied the fairly standard version of Russo's formula for independent Bernoulli random variables whose parameters are labeled by $\lambda$ with the above form.} 
By Poisson thinning, $Y := \{Y_{(x, B)}\}_{(x, B) \in \Omega_\Lambda}$ forms a finite collection of independent Poisson random variables, where for every $(x, B) \in \Omega_\Lambda$, $Y_{(x, B)}$ has mean $\lambda \cdot \bP_x(\mathrm{R}^{\bar{\Lambda}} = B)$, where $\mathrm{R}$ is the trace by time $t$ and $\mathrm{R}^{\bar{\Lambda}}$ is the corresponding restricted trace as in \eqref{restricted_trace}.
Note that $Y$ takes values in $\N^{\Omega_{\Lambda}}$.
We observe that the event $\{\mathbf{0} \rightarrow \Lambda^{c}\}$ is measurable with respect to the sigma-algebra $\sigma( Y )$. Therefore, considering the map $Y: \Sigma \rightarrow \mathbb{N}^{\Omega_{\Lambda}},$ where $\Sigma$ denotes the space of configurations of the frog model, there exists $A\in \mathcal{P}(\mathbb{N}^{\Omega_{\Lambda}}) ,$ such that $\{\mathbf{0} \rightarrow \Lambda^{c}\} = Y^{-1}(A)$.

Let  $\omega, \omega' \in \N^{\Omega_{\Lambda}}$ and write $\omega = \{\omega_{(x, B)}\}_{(x, B) \in \Omega_\Lambda}$. 
We say $\omega \leq \omega^{\prime}$ if $\omega_{(x, B)} \le \omega_{(x, B)}^{'}$ for every $(x, B) \in \Omega_\Lambda$.
We say that $E\subseteq \N^{\Omega_\Lambda}$ is \textit{increasing} if $\omega \in E$ and $\omega \leq \omega^{\prime}$ imply that $\omega^{\prime} \in E$. 
It is clear that $A$ is increasing.
Let $\omega^{+}_{x,B} := \omega + \vec{\delta}_{(x,B)}$, where $\vec{\delta}_{x,B}$ denotes the vector that is one for the pair $(x,B)$ and zero everywhere else.
For every $(x, B) \in \Omega_\Lambda$, we define the following pivotal event:
$$\text{Piv}^{\mathbf{0}\rightarrow \Lambda^{c}}_{x,B} := \left\{ Y^{+}_{x,B} \in A, Y \not\in A \right\}.$$



\begin{proof}
In a derivation analogous to Russo's formula for Poisson processes (see, for instance, Equation (19.2) of \cite{LecturesPoisson}), we have that: 

\begin{equation}
\begin{split}
    \frac{\partial}{\partial\lambda}\bP_{\lambda,t}\left(\mathbf{0} \rightarrow \Lambda^c\right)
    = \sum_{x\in \Lambda} \sum_{B \subseteq \bar{\Lambda}}\bP_{x}\left(\mathrm{R}^{\bar{\Lambda}} = B\right) \bP_{\lambda,t}\left(\text{Piv}^{\mathbf{0} \rightarrow \Lambda^c}_{x,B}\right) \, .
    \end{split} \label{eq: russo_lambda_1}
\end{equation}
To analyze $\bP_{\lambda,t}\big(Y^{+}_{x,B} \in A, Y \not\in A\big),$ let $\mathcal{L}:=\{x \in \Lambda: x \not\rightarrow\Lambda^{c}\}$ be the set of vertices $x$ in $\Lambda$ such that $x \nrightarrow \Lambda^{c}$. On the event $\{\mathcal{L} = S\}$, observe that $(x,B)$ is pivotal if and only if 
\begin{enumerate}[nosep]
    \item $B \cap (\bar{\Lambda} \, \setminus \, S) \neq \varnothing$ and 
    \item $\mathbf{0}$ activates $x$ internally to $S$.
\end{enumerate}
Note that condition (2) implies that $\mathbf{0}, x \in S$. In particular, $Y_{(x, B')} = 0$ for any $B' \cap S^c \neq \varnothing$. Then, from (1) and (2) it follows that
\begin{align}\label{pivequality}
\frac{\partial}{\partial\lambda}\bP_{\lambda,t}\left(\mathbf{0} \rightarrow \Lambda^c\right) &= \sum_{  \substack{S \ni \mathbf{0}  \\ S \subseteq \Lambda}}  \sum_{x\in \Lambda}\sum_{B \subseteq \bar{\Lambda}}\bP_{x}\left(\mathrm{R}^{\bar{\Lambda}} = B\right)\bP_{\lambda,t}\left(Y^{+}_{x,B} \in A, Y \not\in A, \mathcal{L}=S \right) \\
&= \sum_{  \substack{S \ni \mathbf{0}  \\ S \subseteq \Lambda}} \sum_{x \in S} \bP_{x}\Big(\tau_{S^c} \le t\Big) \, \bP_{\lambda, t}\left(\mathbf{0} \xrightharpoonup{S} x, \mathcal{L} = S\right) \, .
\end{align}
We now claim that the events $\{\mathbf{0} \xrightharpoonup{S} x\}$ and $ \{\mathcal{L} = S\}$ are independent.  Indeed, we can write
\begin{equation}\label{sdependency}
\{ \, \mathcal{L} = S \, \} =  \bigg\{ \text{for all } y \in S, \text{ and } 1\leq i \leq \eta_{y}, \  \mathrm{R}_{y}^{i} \cap S^{c} = \varnothing \bigg\} \bigcap \, \left\{ \, \forall \, y \in S^{c}, \, y \xrightarrow{S^c} \Lambda^c \right\} \, .
\end{equation}
Note that the rightmost event depends only on the particles whose initial positions are in $S^c$. 
Additionally, by Poisson thinning, we can split the $\text{Po}(\lambda)$ number of particles at each $y \in S$ into two disjoint and independent collections according to whether the particles exit the set $S$ or not. 

Now, it is a straightforward observation that the event $\{\mathbf{0} \xrightharpoonup{S} x\}$ depends on the particles which do not exit $S$, while the first set in \eqref{sdependency} only depends on particles that do exit $S$, as it is the event that the number of particles starting at $S$ that exit $S$ is zero, implying the claim. It then follows that:
\begin{align}
    \frac{\partial}{\partial\lambda} \bP_{\lambda, t}(\mathbf{0} \rightarrow \Lambda^c) & = \sum_{ \substack{S \ni \mathbf{0}  \\ S \subseteq \Lambda}}\sum_{x \in S} \bP_{x}(\tau_{S^{c}} \le t) \bP_{\lambda, t}(\mathbf{0} \xrightharpoonup{S} x) \bP_{\lambda, t}(\mathcal{L} = S), \nonumber \\
    & \ge \frac{1}{\lambda}\inf_{ \substack{S \ni \mathbf{0} \\ S \subseteq \Lambda}} \phi_{\lambda, t}(S) \left( 1 - \bP_{\lambda, t}(\mathbf{0} \xrightarrow[]{}\Lambda^c) \right),
\end{align}
which completes the proof.
\end{proof} 

\begin{proof}[Proof of item $(1)$ of  Theorem \ref{mainsharpness}]
Throughout the proof, fix an arbitrary $t>0.$
First, when $\tilde{\lambda}_c(t) = \infty$, we will show in \cref{sec:subcritical} that $\lambda_c(t) = \infty$, which makes item $(1)$ vacuous. 

Second, we show that $\tilde\lambda_c(t) \ge 1/t$. In particular, we have $\tilde\lambda_c(t) > 0$. We argue by contradiction. Let $\tilde{\lambda}_c(t) < \lambda < 1/t$. Take a sequence of balls, $\vec{B}_{\mathbf{0}}(n):=\{x \in V: d_{\vec{G}}(0,x)\leq n\}.$
Consider $|\mathfrak{N}(\vec{B}_{\mathbf{0}}(n))|$ for $n \ge 1$. On the one hand, it follows from $\lambda > \tilde{\lambda}_c(t)$ that for any $n$, we have $\bE_{\lambda, t}[|\mathfrak{N}(\vec{B}_{\mathbf{0}}(n))|] > c(\Delta, \lambda, t)$. On the other hand, it follows from Proposition \ref{notzero} (the proof also holds for directed vertex-transitive graphs) that $|\mathfrak{N}(\vec{B}_{\mathbf{0}}(n))| \rightarrow 0$ almost surely. Moreover, let $\mathcal{T}$ be as in Proposition \ref{notzero}. Then it follows from $\bE_{\lambda,t}(s^{\mathcal{T}}) < \infty $ for some $s>1$, that $\bP_{\lambda, t}(\mathcal{T} \ge k)$ decays exponentially in $k$ (see Equation \eqref{eq: exp_decay} for more details). In particular, $\mathcal{T}$ is integrable. Therefore by the Dominated Convergence theorem, we have $\lim_{n \rightarrow \infty}\bE_{\lambda, t}[|\mathfrak{N}(\vec{B}_{\mathbf{0}}(n))|] = 0$, thus implying a contradiction.

Throughout the rest of the proof, assume that $\tilde{\lambda}_c := \tilde{\lambda}_c(t) < \infty$.
We now observe that Lemma \ref{partiallambda} implies Equation \eqref{suplambda}. Let $f(\lambda):= \bP_{\lambda, t}(\mathbf{0} \rightarrow \Lambda^c)$ and note that $\inf_{ \substack{S \ni \mathbf{0} \\ S \subseteq \Lambda}} \phi_{\lambda,t}(S)> \mathrm{c}(\Delta,\lambda,t)$ holds for any $\lambda > \tilde{\lambda}_c$. From the lemma above, we get the following differential inequality:
$$\frac{f^{\prime}(\lambda)}{1-f(\lambda)} > \frac{\mathrm{c}(\Delta,\lambda,t)}{\lambda}\, , \quad \text{for $\lambda\in (\tilde{\lambda}_{c}, +\infty)$} \, . $$ 


Now, fix $t> 0$ and let $\lambda_0 := \sup \{\lambda > 0: f(\lambda) \leq 1/2 \}$. It follows from the definition of $\lambda_0$, that $f(\lambda) > 1/2$ for any $\lambda > \lambda_0$.
Then, for every $\lambda$ such that  $2\tilde{\lambda}_{c}\geq\lambda > \text{max}\{\lambda_{0},\tilde{\lambda}_{c}\}$, we have that
\begin{align*}
    f(\lambda) > \frac{1}{2} \ge \frac{\lambda - \tilde{\lambda}_{c}}{2\tilde{\lambda}_{c}} \, .
\end{align*}
Furthermore, we have that for every $\lambda$ such that  $\text{min}\{\lambda_0,2\tilde{\lambda}_{c}\} \geq \lambda \geq \tilde{\lambda}_{c}$, 
\begin{align*}
    f'(\lambda) > \frac{\mathrm{c}(\tilde{\lambda}_{c})}{4\tilde{\lambda}_{c}}, 
\end{align*}
where we observe that $\tilde \lambda_{c}$ implicitly depends on the graph $\vec{G}$ and $t$, and we write $\mathrm{c}(\tilde{\lambda}_{c})$  for the smallest value for $\mathrm{c}(\Delta, \lambda, t)$ on the interval $[\tilde{\lambda}_{c}(t),2\tilde{\lambda}_{c}(t)]$, which is a strictly positive quantity by uniform continuity. So, by integrating from $\tilde{\lambda}_{c}$ to $\lambda,$ we get 

\begin{align*}
    f(\lambda) >  \frac{\mathrm{c}\big(\tilde{\lambda}_{c}\big)}{4\tilde{\lambda}_{c}}(\lambda - \tilde{\lambda}_{c}) \, .
\end{align*}
By taking $K\big(\tilde\lambda_{c}(t)\big):= \text{min}\Big\{\frac{1}{2\tilde{\lambda}_{c}},\frac{\mathrm{c}(\tilde{\lambda}_{c})}{4\tilde{\lambda}_{c}}\Big\}$ we conclude that 

\begin{align*}
    f(\lambda) \geq K(\lambda - \tilde{\lambda}_c) \, ,
\end{align*}
for $\lambda\in [\tilde{\lambda}_{c}(\lambda),2\tilde{\lambda}_{c}(\lambda)].$  Finally, we complete the proof by taking the sequence of finite sets $\Lambda_n:= \vec{B}_{\mathbf{0}}(n)$ and letting $n\rightarrow\infty$.
\end{proof}

We now begin to prove the second result in this subsection.

\begin{lemma}\label{partialt}
For any $\lambda, t > 0$ and any finite set $\Lambda \subset V$, we have
\begin{align}
    \frac{\partial}{\partial t} \bP_{\lambda, t}(\mathbf{0} \xrightarrow[]{} \Lambda^c) \ge \frac{\lambda e^{-t}}{ t} \inf_{\substack{  S \ni \mathbf{0}  \\ S \subseteq \Lambda}} \phi_{\lambda, t}(S) \left( 1 - \bP_{\lambda, t}(\mathbf{0} \xrightarrow[]{}\Lambda^c) \right).
\end{align}
\end{lemma}

\begin{proof}
We now derive a Russo-type formula with respect to $t$ analogously to the case of Bernoulli percolation  (see \cite{GRIbook}).
For each fixed $\lambda>0$, we now describe a monotone coupling in the lifespan $t$ which allows particles initialized on different vertices to have different lifespans. For each $x$, let $t_x > 0$ be the lifespan of all particles initialized at $x$. 
For each vertex $x \in V$, we consider the natural coupling of the model in which $\eta_{x}$ is distributed according to $\text{Po}(\lambda)$. And, for each particle $(\omega_{x}^{i})_{1\leq i \leq \eta_{x}},$ we sample i.i.d. continuous-time simple random walks starting from $x$ with infinite lifespan denoted by $(Y_{x}^{i}(s))_{s \geq 0}.$ Now, given a fixed $t_x>0$, we consider for each $x\in V$ the restriction of the random walks $1 \leq i \leq \eta_x$ up to time $t_x$, i.e. $(Y_{x}^{i}(s))_{0\leq s \leq t_x}.$ We denote the probability measure associated with the frog model with particle density $\lambda$ and lifespan $\mathbf{t} := (t_x)_{x \in \Lambda}$ as $\bP_{\lambda,\mathbf{t}}$ and the measure of the coupled probability space as $\bP_{\lambda}$. Moreover, for a fixed $\mathbf{t}$, we add the subscript to the previously defined notation $\mathcal{L}$ and write $\mathcal{L}_{\mathbf{t}}$ to denote the (random) vertex set that is not connected to $\Lambda^c$ under lifespan $\mathbf{t}$. Also, with a slight abuse of notation, we write $\mathbf{t}$ above the arrows to specify its dependency on the lifespan in the coupled probability space, for example $\xrightarrow{\mathbf{t}}$ and $\xrightharpoonup{S, \, \mathbf{t}}$.

Let $\varepsilon > 0$ and $\vec{\varepsilon}_x := (\varepsilon \1_{x = y})_{y \in \Lambda}$. One can understand the frog model as a directed percolation model, and on the event $\{\mathcal{L}_{\mathbf{t}}=S\}$, adding lifespan $\varepsilon$ to the particles of $x\in S $ is pivotal to $\mathbf{0} \rightarrow \Lambda^c$ when (1) $\mathbf{0} \xrightharpoonup{S,\, \mathbf{t}} x$ and (2) this causes an edge from $x$ to $\Lambda^c$ to be added, i.e. a frog in $x$ jumps to $\Lambda^c$ between $t_x$ and $t_x + \varepsilon$. Observe that (2) happens with probability $1 {-} \exp(-\lambda \bP_{x}(\tau_{S^{c}} \in (t,t+\varepsilon])$.
We add the superscript $\mathbf{t}$, and denote by $A^{\mathbf{t}}_{S}(x)$ (see Section \ref{preliminaries}) as the set of particles that belong to $x \in S$ and do not exit $S$ in their lifespan $t_x$. 
Observe that:
\begin{align*}
    & \bP_{\lambda,\mathbf{t} + \vec{\varepsilon}_x}\Big(\mathbf{0} \rightarrow \Lambda^c\Big) - \bP_{\lambda,\mathbf{t}}\Big(\mathbf{0} \rightarrow \Lambda^c\Big)
	= \bP_{\lambda} \left( \left\{ \mathbf{0} \xrightarrow{\mathbf{t} + \vec{\varepsilon}_x} \Lambda^c \right\} \setminus \left\{ \mathbf{0} \xrightarrow{\mathbf{t}} \Lambda^c \right\} \right) \\[0.7em]
    & \; \; = \sum_{\substack{S \ni \boldsymbol{0} \\ S \subseteq \Lambda}} \bP_{\lambda}\left( \{ \mathcal{L}_{\mathbf{t}} = S \} \cap \left( \{ \mathbf{0} \xrightarrow{\mathbf{t} + \vec{\varepsilon}_x} \Lambda^c \} \setminus \{ \mathbf{0} \xrightarrow{\mathbf{t}} \Lambda^c \} \right) \right) \\
    & \; \; = \sum_{\substack{S \ni \boldsymbol{0} \\ S \subseteq \Lambda}} \bP_{\lambda} \left( \left\{ \mathcal{L}_\mathbf{t} = S \right\} \cap \left\{ \boldsymbol{0} \xrightharpoonup{S, \, \mathbf{t}} x \right\} \cap \left\{ \exists i\in A^{ \mathbf{t}}_{S}(x), \tau^{x}_{S^{c}}(i)\in (t_{x},t_{x}+\varepsilon] \right\} \right) \\
    & \; \; = \sum_{\substack{S \ni \boldsymbol{0} \\ S \subseteq \Lambda}} \bP_{\lambda,t} \left(  \mathcal{L} = S  \right) \bP_{\lambda,t} \left(  \boldsymbol{0} \xrightharpoonup{S} x  \right) \left( 1 - e^{- \lambda \bP_x(\tau_{S^c}\in (t_{x}, t_{x}+\varepsilon]) } \right).
\end{align*}
The last identity is valid, since by the discussion following Equation \eqref{sdependency}, the event $\{ \mathcal{L}_{\mathbf{t}} = S\}$ is independent of  $\{\mathbf{0} \xrightharpoonup{S}x\}$ and
$\big\{{\exists i\in A^{\mathbf{t}}_{S}(x), \tau^{x}_{S^{c}}(i)\in (t_{x},t_{x}+\varepsilon]}\big\}$. Moreover, the last two events are also independent, since the first one does not depend on particles at $x$ while the second one concerns only particles at $x$. 
Additionally, let $\vec{t} := (t, \ldots, t)$, then by the chain rule:

\begin{align*}
    &\frac{\partial}{\partial t}\bP_{\lambda,t}(\mathbf{0} \rightarrow \Lambda^c) = \sum_{x \in \Lambda } \frac{\partial}{\partial t_x} \bP_{\lambda, \vec{t} }\,(\mathbf{0} \rightarrow \Lambda^c) \\
    = & \sum_{x \in \Lambda}\sum_{\substack{S \ni \boldsymbol{0} \\ S \subseteq \Lambda}} \bP_{\lambda, t} \left(  \mathcal{L} = S  \right) \bP_{\lambda, t} \left(  \boldsymbol{0} \xrightharpoonup{S} x \right) \lim_{\varepsilon \rightarrow 0}\frac{\lambda\bP_{x}\big(\tau_{S^{c}} \in (t,t+\varepsilon] \big)}{\varepsilon} \, .
\end{align*}

Let $(Y_{x}(s))_{s \ge 0}$ be a continuous-time simple random walk starting at $x$, and let $(X(n))_{n=0}^{\infty}$ be the discrete-time simple random walk extracted from its jumps. Given a continuous-time random walk starting at $x,$ let $U_x(k)$ be the sum of the first $k$ jumps of said random walk. Observe that $U_{x}(k)$ is the sum of $k$ independent $\mathrm{Exp}(1)$ random variables, and hence, it is distributed according to a Gamma distribution with shape $k$ and rate $1$, denoted by $\Gamma(k, 1)$. Let $U_k \sim \Gamma(k, 1)$ and let $\nu_{S^c} := \{ n \ge 0: X(n)\in S^c \}$. Additionally, let $\bP$ be the measure associated with such a random variable for each $k$. Observe that 
\begin{align} \label{eq: cts-dis-exit}
    \bP_{x}\big(\tau_{S^{c}} \in (t,t+\varepsilon] \big) &= \sum_{k = 1}^{+\infty} \, \bP_x(\nu_{S^c} = k) \, \bP\big(U_k \in (t, t+\varepsilon]\big) \, .
\end{align}
Let $f_{U_k}$ be the density function of $U_k$. Therefore, by dividing by $\varepsilon$ and taking limits as $\varepsilon\rightarrow0$ we have that \footnote{We can exchange limits with sums since $f_{U_k}(t) = \frac{1}{(k-1)!} t^{k-1} e^{-t}$.}
\begin{align}
    \frac{\partial}{\partial t} \, \bP_{\lambda,t}(\mathbf{0} \rightarrow \Lambda^c) = \sum_{\substack{  S \ni \mathbf{0}  \\ S \subseteq \Lambda}} \bP_{\lambda, t}(\mathcal{L} = S) \cdot\lambda\sum_{x \in S}  \bP_{\lambda, t}(\mathbf{0}\xrightharpoonup{S} x) \sum_{k = 1}^{+\infty} \, \bP_x(\nu_{S^c} = k) \, f_{U_k}(t) \, .
\end{align}
It is a straightforward calculation to check that for any $t \ge 0$ and $k \ge 1$, we have $f_{U_k}(t) \ge \frac{e^{-t}}{t} \bP(U_k \le t)$. We may conclude that 
\begin{align*}
    \frac{\partial}{\partial  t}  \bP_{\lambda,t}(\mathbf{0} \rightarrow \Lambda^c) &\geq \sum_{\substack{  S \ni \mathbf{0}  \\ S \subseteq \Lambda}} \bP_{\lambda, t}(\mathcal{L} = S) \cdot\lambda\sum_{x \in S}  \bP_{\lambda, t}(\mathbf{0}\xrightharpoonup{S} x) \sum_{k = 1}^{+\infty} \, \bP_x(\nu_{S^c} = k) \, \frac{e^{-t}}{t} \bP(U_k \le t) \\
    &= \frac{\lambda e^{-t}}{t} \sum_{\substack{  S \ni \mathbf{0}  \\ S \subseteq \Lambda}} \bP_{\lambda, t}(\mathcal{L} = S) \cdot \sum_{x \in S}  \bP_{\lambda, t}(\mathbf{0}\xrightharpoonup{S} x) \bP_x(\tau_{S^c} \le t) \\
    &\ge \frac{\lambda e^{-t}}{t} \inf_{ \substack{S \ni \mathbf{0} \\ S \subseteq \Lambda}} \phi_{\lambda, t}(S) \left( 1 - \bP_{\lambda, t}(\mathbf{0} \xrightarrow[]{}\Lambda^c) \right),
\end{align*}
where in the second line we used Equation \eqref{eq: cts-dis-exit} with $(t, t+\varepsilon]$ replaced by $(0, t]$.
\end{proof}

\begin{proof}[Proof of item $(2)$ of Theorem \ref{mainsharpness}]
Throughout the proof, fix an arbitrary $\lambda>0.$
First, when $\tilde{t}_c(\lambda) = \infty$, we will show in \cref{sec:subcritical} that $t_c(\lambda) = \infty$, which makes item $(2)$ vacuous. 
Observe that analogously to the proof of item (1), we have that $\tilde{t}_{c}(\lambda)\geq 1/\lambda.$ In particular, we have $\tilde{t}_{c}(\lambda) > 0$.

Throughout the rest of the proof, assume that $\tilde{t}_c := \tilde{t}_c(\lambda) < \infty$.
Now, we observe how we can derive the Equation \eqref{supt} from Lemma \ref{partialt}. Let $h(t):= \bP_{\lambda, t}(\mathbf{0} \rightarrow \Lambda^c).$ Then, 
$$\frac{h^{\prime}(t)}{1-h(t)} > \frac{\lambda \cdot\mathrm{c}(\Delta,\lambda,t)}{e^{t}\cdot t}, \text{\ for $t\in (\tilde{t}_{c}, +\infty )$}. $$
Let $t_0 := \sup \{t > 0: h(t) \leq 1/2 \}$. By definition, we have $h(t) > 1/2$ for any $t > t_0$.  Then, for every $t$ such that  $2\tilde{t}_{c}\geq t > \text{max}\{t_{0},\tilde{t}_{c}\}$, we have that 
\begin{align*}
    h(t) > \frac{1}{2} \ge \frac{t - \tilde{t}_{c}}{2\tilde{t}_{c}} \, .
\end{align*}
Furthermore, we have that for every $t$ such that  $\text{min}\{t_{0},2\tilde{t}_{c}\} \geq t \geq \tilde{t}_{c}$, 
\begin{align*}
    h'(t) > \frac{ \lambda \cdot \mathrm{c}(\tilde t_{c})}{4  \tilde{t}_{c} \cdot e^{2\tilde{t}_{c}}},
\end{align*}
where we write $\mathrm{c}(\tilde{t}_{c})$ for the smallest value for $\mathrm{c}(\Delta, \lambda, t)$ on the interval $[\tilde{t}_{c}(\lambda),2\tilde{t}_{c}(\lambda)]$.
Integrating from $\tilde{t}_{c}$ to $t$ gives us 
$$ h(t) > \frac{ \lambda \cdot \mathrm{c}(\tilde t_c)}{4 \tilde{t}_{c} \cdot  e^{2\tilde{t}_{c}}}\big(t-\tilde{t}_{c}\big).$$ 
By taking $K\big(\tilde{t}_{c}(\lambda)\big):= \text{min}\Big\{\frac{1}{2\tilde{t}_{c}},\frac{\lambda \cdot  \mathrm{c}({\tilde t_{c}})}{4\tilde{t}_{c} \cdot e^{2\tilde{t}_{c}}} \Big\}$ we conclude that 
\begin{align*}
    h(t) \geq K(t - \tilde{t}_c) \, .
\end{align*}
for $t\in [\tilde{t}_{c}(\lambda),2\tilde{t}_{c}(\lambda)].$ The proof is once again completed by taking the sequence $\Lambda_n:= \vec{B}_{\mathbf{0}}(n)$ and letting $n$ go to infinity.  
\end{proof}

With the proof of the results for the supercritical regime concluded, we now study the subcritical regime. 

\subsection{Subcritical Phase} \label{sec:subcritical}
In order to study the subcritical regime, we first compare $\phi(S)$ with the following quantity. Let $S \subset V$ be a finite connected set containing the origin, and define

\begin{align}\label{phiformulaplus}
     \widetilde{\phi}_{\lambda, t}(S) := \sum_{ \substack{x \in S}} \lambda \bP_x(\tau_{S^c} \leq t) \bP_{\lambda, t}(\mathbf{0} \xrightharpoonup{S} x)  \bE_x\left[ N(t) \, \big\vert \, \tau_{S^c} \le t \right],
\end{align}
where $N(t),$ defined in Section \ref{preliminaries}, denotes the number of jumps  a random walk (starting at $x$) performs by time $t$. We begin with the following statement:

\begin{lemma}\label{constantcomparison}
   Fix a transitive directed graph $\vec{G}$ with outer degree $\Delta$. For any $\lambda, t > 0$, there exists a positive and finite function $C(\Delta,\lambda,t)$ such that
    \begin{align}
        \widetilde{\phi}_{\lambda, t}(S)< C(\Delta,\lambda, t) \, \phi_{\lambda, t}(S) \, ,
    \end{align}
    holds for any finite, connected $S \subset V$ containing the origin. Moreover, $c(\Delta,\lambda,t) := 1/C(\Delta,\lambda,t)$ is continuous in $\lambda$ and $t$ on $[0, \infty) \times [0, \infty).$ 
\end{lemma}
The proof of the above lemma is a consequence of  Lemmas \ref{lemma: exponential_bound_Nt}, and \ref{lemma: exponential_bound_Ar}, which we now state and prove in detail.

\begin{lemma} \label{lemma: exponential_bound_Nt}
Let $S \subset V$ be a finite set containing the origin. For $x \in S$ define $D_x := d_{\vec{G}}(x, S^c).$ We have that for any $t > 0$ and $x\in S:$
\begin{equation}
\bE_x\left[ N(t) \, \big\vert \, \tau_{S^c} \le t \right] \leq \Delta^{D_x} (t + D_x) \, .
\end{equation}
\end{lemma}

\begin{proof}
First, observe that:
\begin{align*}
    \bE_x\left[ N(t) \, \big\vert \, \tau_{S^c} \le t \right] &\le \frac{\bE_x\left[ N(t) \, \1_{N(t) \geq D_x} \right]}{\bP_x(\tau_{S^c} \le t)}.
\end{align*}
We proceed by upper-bounding the numerator and lower-bounding the denominator. For the numerator, using $\bP_{x}(N(t)=k) = t^{k}e^{-t}/k!,$ we have that
\begin{align*}
    \bE_x\left[ N(t) \, \1_{N(t) \geq D_x} \right] &= \sum_{k = D_x}^{+\infty} k \, \bP_x(N(t) = k) \\
    &= t \sum_{k = D_x-1}^{+\infty} \bP_{x}(N(t) = k) = t \, \bP_x(N(t) \ge D_x - 1) \, .
\end{align*}
To lower bound the denominator, we fix a shortest path $\gamma = (v_0, v_1, \ldots, v_{D_x})$ from $v_0 := x$ to $v_{D_x} \in S^c$ with $\{ v_i, v_{i+1} \} \in \vec{E}$. Then
\begin{align*}
    \bP_x(\tau_{S^c} \le t) \ge \bP_x(N(t) \ge D_x, \text{the first $D_x$ jumps are along $\gamma$}) = \Delta^{-D_x} \, \bP_{x}(N(t) \ge D_x) \, .
\end{align*}
Therefore, we have 
\begin{align*}
    \bE_x\left[ N(t) \, \big\vert \, \tau_{S^c} \le t \right] &\le \Delta^{D_x} t \cdot \frac{\bP_{x}(N(t) \ge D_x - 1)}{\bP_{x}(N(t) \ge D_x)} \\
    &\le \Delta^{D_x} t \cdot  \bigg( 1 + \frac{\bP_{x}(N(t) = D_x - 1)}{\bP_{x}( N(t) = D_x)} \bigg) = \Delta^{D_x} t \cdot \bigg(1 + \frac{D_x}{t}\bigg). \qedhere
\end{align*}
\end{proof}

For $r \geq 1$, let $S_r := \{x \in S : d_{\vec{G}}(x, S^c)=r \}$ and $A_r := \{ x \in S_r: \boldsymbol{0} \xrightharpoonup{S} x \}$. Before we move on to the next result, we explain why the Harris-FKG inequality is applicable in our setup.  

Fix some particle density and lifespan $\lambda,t \in (0,\infty)$. For $v \in V$ and a finite $B \subset V$ containing $v$, let $Y_{(v,B)}=Y_{(v,B)}(t)$ be the number of particles whose initial location is $v$ whose trace (for the lifespan $t$) is $B$. Let $Z_{(v,B)}=\mathbbm{1}_{Y_{(v,B)}(t) \ge 1}$. The probability that there is a particle whose trace  is infinite is 0, and we may construct the probability space such that this event is the empty set. When this is done, all events considered below concerning the frog model with parameters $(\lambda,t)$ are measurable with respect to \ $\sigma(Z_{(v,B)}: v \in B \subset V, \,  |B| < \infty)$. By Poisson thinning, the $Y$'s are independent Poisson random variables, and thus the $Z$'s are independent Bernoulli random variables. Below, and throughout the paper, we apply Harris-FKG with respect to the $Z$'s.

\begin{lemma} \label{lemma: exponential_bound_Ar}
There exists $K = K(\Delta, \lambda, t)$ such that for all finite $S$ containing the origin, we have
$$\bE_{\lambda, t}\big[|A_r|\big] \le K^{r-1} \bE_{\lambda, t}\big[|A_1|\big] \, .$$
\end{lemma}

\begin{proof}
   We define for $x,y \in V$ and $d_{\vec{G}}(x,y)=1$, the event $x \rightsquigarrow y$  that there exists a particle initialized on $x$ that makes exactly one jump within time $t$, and this jump is from $x$ to $y$.  
    By definition, we have 
    $$\bP(x\rightsquigarrow y) = 1 - e^{-\frac{\lambda}{\Delta}\frac{t}{e^{t}}} \, , $$
    and we denote the quantity on the right-hand side as $\delta.$ For any $u \in S_1$, let $\mathcal{D}_r(u) := \{ x \in S_{r} : d_{\vec{G}}(u, x) = r-1 \}$. 
    For any $u \in S_1$ and any $x \in \mathcal{D}_r(u)$, there exists a path $x=v_{1}, v_{2}, \ldots, v_{r}=u$ such that $d_{\vec{G}}(v_{\ell},v_{\ell+1})=1$ for all $\ell \in \{1, \ldots, r-1 \}$. Hence, we have by the Harris-FKG inequality that:
    \begin{align*}
          \bP_{\lambda,t}(\mathbf{0}\xrightharpoonup{S}u) \geq \bP_{\lambda,t}(\mathbf{0}\xrightharpoonup{S}v_{1}, v_{1} \rightsquigarrow v_{2}, \ldots, v_{r-1} \rightsquigarrow v_{r}) \geq \bP_{\lambda,t}(\mathbf{0}\xrightharpoonup{S}x)\delta^{r-1}.
     \end{align*}
     Therefore, 
     \begin{align*}
          &\sum_{x \in S_{r}} \bP_{\lambda,t}(\mathbf{0}\xrightharpoonup{S}x)\delta^{r-1} 
          \leq \sum_{u \in S_1} \sum_{x \in \mathcal{D}_r(u)} \bP_{\lambda,t}(\mathbf{0}\xrightharpoonup{S} x)\delta^{r-1} \\
          \leq \sum_{u \in S_1} \sum_{x \in \mathcal{D}_r(u)} &\bP_{\lambda,t}(\mathbf{0}\xrightharpoonup{S} u) \leq \sum_{u \in S_1} \bP_{\lambda,t}(\mathbf{0}\xrightharpoonup{S} u) |\mathcal{D}_r(u)| \leq \sum_{u \in S_{1}} \bP_{\lambda,t}(\mathbf{0}\xrightharpoonup{S} u)\Delta^{r-1} \, .
     \end{align*}
     Taking $K := \Delta/\delta \leq \max \Big\{ 4\Delta, 2\frac{\Delta^{2}e^{t}}{\lambda t}\Big\}$ completes the proof.  
\end{proof}

We can now turn ourselves back to the proof of Lemma \ref{constantcomparison}.

\begin{proof}[Proof of Lemma \ref{constantcomparison}]
    Consider the ratio
    \begin{align*}
        \frac{\tilde{\phi}_{\lambda, t}(S)}{\phi_{\lambda, t}(S)} = \frac{\sum_{x \in S} \bP_{\lambda, t}\left( \boldsymbol{0} \xrightharpoonup{S} x \right) \bP_x(\tau_{S^c} \le t) \bE_x\left[ N(t) \, \big\vert \, \tau_{S^c} \le t \right]}{\sum_{x \in S} \bP_{\lambda, t}\left( \boldsymbol{0} \xrightharpoonup{S} x \right) \bP_x(\tau_{S^c} \le t)} \, .
    \end{align*}
    It follows from Lemma \ref{lemma: exponential_bound_Nt} and Lemma \ref{lemma: exponential_bound_Ar} that the numerator is upper bounded by 
    \begin{align*}
        \sum_{r=1}^{\infty} \bE_{\lambda, t}\left[ |A_r| \right] \, \bP_x(\tau_{S^c} \le t) \, \Delta^{r} (t + r) \le \bE_{\lambda, t}\left[ |A_1| \right] \sum_{r=1}^{\infty} K^{r-1} \, \bP(\text{Po}(t) \ge r) \, \Delta^{r} (t + r) \, .
    \end{align*}
    On the other hand, the denominator is lower bounded by
    \begin{align*}
        \sum_{x \in S_1} \bP_{\lambda, t}\left( \boldsymbol{0} \xrightharpoonup{S} x \right) \bP_x(\tau_{S^c} \le t) \ge \frac{1-e^{-t}}{\Delta} \, \bE_{\lambda, t}[|A_1|] \, .
    \end{align*}
    We finally derive that 
    
    \begin{align*}
    \frac{\tilde{\phi}_{\lambda, t}(S)}{\phi_{\lambda, t}(S)}  &\leq \frac{ \Delta}{1-e^{-t}} \sum_{r=1}^{\infty} K^{r-1} \bP(\text{Po}(t)\geq r) \Delta^{r} (t + r)\\ 
    & < \frac{ \Delta}{1-e^{-t}} \cdot 2(t+1)^{2}\cdot K^{t} \Delta^{t+1} + \frac{ \Delta}{1-e^{-t}} \sum_{r=\lfloor t+1 \rfloor }^{\infty} K^{r-1} e^{r-t} \bigg(\frac{t}{r}\bigg)^{r}  \Delta^{r}2r \\
     & < \frac{  2(t+1)^{2}\Delta}{1-e^{-t}}\cdot \bigg(4\Delta^{2}+2\frac{\Delta^{3}e^{t}}{\lambda t}\bigg)^{t+1}  + \frac{ \Delta^{2} 2t}{1-e^{-t}} \sum_{r=\lfloor t+1\rfloor }^{\infty} \bigg(4\Delta^{2}+ 2\frac{\Delta^{3}e^{t}}{\lambda t}\bigg)^{r-1} e^{r-t} \bigg(\frac{t}{r}\bigg)^{r-1}  \\ 
     &=: \mathrm{C}(\Delta,\lambda,t) \, ,
    \end{align*} 
and conclude the proof by observing that since the above sum converges absolutely for any $\lambda,t>0$ we have that $c(\Delta,\lambda,t)$ is continuous and strictly positive on $(0,\infty)\times (0,\infty),$ and by taking limits we can observe that $c(\Delta,0,t)=c(\Delta,\lambda,0)=0$ for any $\lambda,t \geq 0$. 
\end{proof}

We now begin the proof of Item (3) of Theorem \ref{mainsharpness}. 
Fix $t >0$ and let $\lambda < \tilde{\lambda}_c(t)$. By the definition of $\tilde{\lambda}_{c}(t),$ the observation above, and the monotonicity of $\widetilde\phi_{\lambda,t}$ with respect to $\lambda,$ there exists a finite set $S \subset V$ such that $\widetilde\phi_{\lambda, t}(S)< 1$. We now define an exploration process which will later be shown to stochastically dominate the frog model. This exploration process will have expected number of children less than $\widetilde\phi_{\lambda,t}(S),$ thus, being subcritical. Define for every $x \in V$ a graph isomorphism $\Psi_{x}: V \rightarrow V$ such that $\Psi_x(\mathbf{0}) = x$, which exists due to the transitivity of $\vec{G}$.

\textbf{Definition of the exploration process:} We define the following exploration process starting at the origin, denoted by $\text{EP}(\mathbf{0})$. For each generation $i\geq 0$ of the exploration process, we let $\mathcal{Z}_{i}$ denote the total number of ``children" belonging to generation $i$. At the beginning $\mathcal{Z}_{0} = 1$ and the original ``parent" is $\mathbf{0}.$ We let $\mathcal{A}$ denote the set of active vertices, that is, vertices that will have their particles' trajectories fully explored by the end of the current iteration, and $\mathcal{V}$ be the set of revealed vertices, that is, the vertices whose particles' trajectories were already fully explored in the current iteration. We define them iteratively in the following way. Initially $\mathcal{A}_{0}:= \mathbf{0}$ and $\mathcal{V}_{0} = \varnothing$. Now, for the vertices in $\mathcal{A},$ in this case $\mathbf{0},$ label each of the $\text{Po}(\lambda)$ particles at the origin in arbitrary order $\{\omega^{1}_{\mathbf{0}},..., \omega_{\mathbf{0}}^{\eta_{\mathbf{0}}}\}$. According to this order, we reveal the trajectory of the random walks. Recall that $A_{S}(x) = \{ 1 \le i \le \eta_x : \mathrm{R}^i_x \cap S^c = \varnothing \}$ and $\mathcal{R}^S_x$ is the union of particle trajectories in $A_{S}(x)$ (see \eqref{def: modified_arrow}). For every particle $i \in A^{c}_{S}(\mathbf{0})$ we count $\cup_{i \in (A_{S}(\mathbf{0}))^c} \mathrm{R}^{i}_{\mathbf{0}}\setminus \{\mathbf{0}\}$ as children of $\mathbf{0}$. Moreover, for every particle $i \in A_{S}(\mathbf{0}),$ we will explore the vertices along their trajectories. That is, we now define $\mathcal{A}_{1}:= \mathcal{R}_{\mathbf{0}}^{S} \setminus\{\mathbf{0}\},$ and $\mathcal{V}_{1} : = \mathbf{0}.$ We repeat the same procedure as above: for every $x \in \mathcal{A}_1 = \mathcal{R}_{\mathbf{0} }^{S}\setminus\{\mathbf{0}\}$, add $\cup_{i \in A^{c}_{S}(x)}\mathrm{R}_{x}^{i} \setminus \{x\}$ to the children of $\mathbf{0}$, and further add $\mathcal{R}_{x}^{S} \setminus \{x\}$ to the set of active vertices $\mathcal{A}_{2}$. Observe that since the vertices in $\mathcal{A}$ all belong to $S,$ this exploration eventually terminates. This determines the end of the iteration starting at $\mathbf{0},$ and thus the end of the exploration of the children belonging to the first generation. See Figure \ref{firststep} for an illustration of it.

Let us now analyze the expected size of the children of $\mathbf{0}$. Define $\mathcal{K}(S):= \{ x \in S: \mathbf{0} \xrightharpoonup{S} x\}$. We now observe that the number of children belonging to the first generation is given by  

$$ \mathcal{Z}_{1}=  \Bigg\vert \bigcup_{x \in \mathcal{K}(S)} \, \bigcup_{i \in A^{c}_{S}(x)} \, \mathrm{R}^{i}_x \setminus \{ x \} \Bigg \vert \, .$$
We observe that $\bE_{\lambda,t}[\mathcal{Z}_{1}]$ can be further upper bounded bounded by

\begin{equation}\label{expectationbound}
\begin{split}
    \bE_{\lambda, t} \left[ \mathcal{Z}_{1} \right] 
    &\le \bE_{\lambda, t}\left[ \sum_{x\in \mathcal{K}(S)} \sum_{i \in A_{S}^{c}(x)} \left( |\mathrm{R}^{i}_x| - 1 \right) \right] \\
    &= \sum_{x \in S} \bE_{\lambda, t} \left[ \1_{ \left\{ \boldsymbol{0} \xrightharpoonup{S} x \right\} } \sum_{i = 1}^{\eta_x} \1_{ \left\{ \tau_{S^{c}}^{(i)}\leq t \right\} } \left( |\mathrm{R}^{i}_x| - 1 \right) \right] \\
    &= \sum_{x \in S} \bP_{\lambda, t} \left( \boldsymbol{0} \xrightharpoonup{S} x  \right) \bE_{\lambda, t} \left[ \sum_{i = 1}^{\eta_x} \1_{ \left\{ \tau_{S^{c}}^{(i)}\leq t \right\} } \left( |\mathrm{R}^{i}_x| - 1 \right) \right]\\
    &= \sum_{x \in S} \bP_{\lambda, t} \left( \boldsymbol{0} \xrightharpoonup{S} x \right) \lambda \bE_{x}\left[\1_{ \left\{ \tau_{S^{c}}\leq t \right\} } \left( |\mathrm{R}_{x}| - 1\right) \right],
\end{split}
\end{equation}
where the second equality uses the independence between $\{ \boldsymbol{0} \xrightharpoonup{S} x \} $ and $\sum_{i = 1}^{\eta_x} \1_{ \{ \tau_{S^{c}}^{(i)}\leq t \} } ( |\mathrm{R}^{i}_x| - 1 )$, and the last identity follows from an application of Wald's equation. From \eqref{expectationbound}, since we have $$\bE_{x}\left[\1_{ \left\{ \tau_{S^{c}}\leq t \right\} } \left( |\mathrm{R}_{x}| - 1\right) \right]\le  \bP_x(\tau_{S^c} \le t) \bE_x\left[ N(t) \, \big\vert \, \tau_{S^c} \le t \right],$$ 
we conclude that $\bE_{\lambda, t} \left[ \mathcal{Z}_{1} \right] \le \widetilde{\phi}_{\lambda, t}(S)<1$.

\begin{figure}
\centering
\begin{tikzpicture} 
\coordinate (a1) at (0,0);
\coordinate (a2) at (3.3,-0.7);
\coordinate (a3) at (4.5,0.2);
\coordinate (a4) at (3,1.5);
\coordinate (a5) at (3,3);
\coordinate (a6) at (2.4,3.2);
\coordinate (a7) at (1,2);

\filldraw[black] (2,1) circle (2pt) node[anchor=west]{$0$};

\filldraw[blue] (2.45,0) circle (2pt) node[anchor=south west]{$w_1$};
\filldraw[red, fill=white] (-0.3,1.2) circle (2pt) node[anchor=east]{$v_{1}$};

\filldraw[red,fill= white] (1.3,1.2) circle (2pt) node[anchor=south  ]{$v_{3}$};

\filldraw[red,fill= white] (0.5,0.9) circle (2pt) node[anchor=north ]{$v_{2}$};

\filldraw[color=green, fill=white] (2.5,-1.08) circle (2pt) node[anchor=north east]{$v_{4}$};

\filldraw[black, thick] (3.2,3) node[anchor=west]{$S$};

\filldraw[blue, fill=blue] (3.5,-0.4) circle (2pt) node[anchor=south west]{$w_2$};;

\draw[red,thick]  (2,1)-- (1.3,1.2)-- (0.5,0.9) -- (-0.3,1.2);

\draw[blue,thick]  (2,1) -- (2.45,0);

\draw[blue,thick]  (2.45,0) -- (3.5,-0.4);

\draw[green, thick] (2.45,0) -- (2.5,-1.08);

\draw[use Hobby shortcut,thick] ([out angle=-90]a1)..(a2)..([in angle=-90]a3); 
\draw[use Hobby shortcut,thick] ([out angle=90]a3)..(a4)..(a5)..(a6)..(a7)..([in angle=90]a1);
\end{tikzpicture}
\caption{In this example, we consider an exploration of $\mathcal{Z}_{1}$. We start with two active particles at the origin, $\omega_{\mathbf{0}}^{1}$ (red) and $\omega_{\mathbf{0}}^{2}$ (blue). We first follow the trajectory of $\omega_{\mathbf{0}}^{1}$ until it exits the set $S$ for the first time or is deactivated. Since $\omega_{\mathbf{0}}^{1}$ left $S$, within its lifespan, we do not explore any of its vertices, and include $v_1, v_2,$ and $v_3$ as children of $\mathbf{0}$. Now, we follow the trajectory of the particle $\omega_{\mathbf{0}}^{2}$, which didn't exit the set $S$ within its lifespan. The vertices in its trajectory were labeled as $w_1$ and $w_2$ for convenience. The vertex $w_1$ has one particle, say $\omega_{w_1}^{1},$ which immediately left the set at the vertex $v_{4}$ represented by the green trajectory. Since the green trajectory exits $S$, we count $v_{4}$ as the fourth child of $\mathbf{0}$. Since $w_2$ has no particles, the exploration of $\mathbf{0}$ terminates. In this realization of the branching process, $\mathbf{0}$ has four children, so $\mathcal{Z}_{1}=4$. }\label{firststep}
\end{figure}
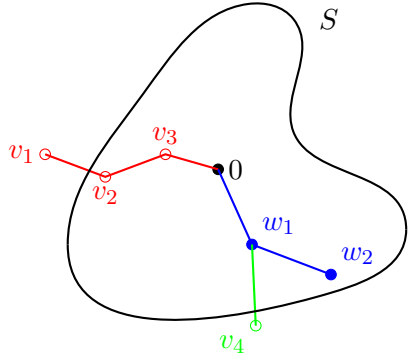

Now, for the exploration of $\mathcal{Z}_{2}$, we consider the vertices in $   \bigcup_{v \in \mathcal{K}(S)} \, \bigcup_{i \in A^{c}_{S}(v)} \, \mathrm{R}^{i}_v \setminus \{ v \},$ that is, the children belonging to the first generation. Arbitrarily order them, label the first one $v_1$, and let  $\Psi_{v_1}(S)$ be the ``translation" of $S$ at $v_{1}$. Now, for each vertex that had its trajectories revealed (that is, vertices belonging to $\mathcal{V}$ at the end of the previous iteration), we refresh them independently with a new $\text{Po}(\lambda)$ number of particles. For each particle, we independently sample a new continuous-time simple random walk trajectory with lifespan $t.$ We consider the iteration of the process starting at $v_{1}$ for the generation one, as the repetition of the previous iteration at $\mathbf{0}$, with $\mathcal{A}:=v_{1},$ $\mathcal{V} = \varnothing,$  and $\Psi_{v_1}(S)$ instead of $S$. It is clear that the expected number of children of $v_{1}$ is the same as the expected number of children of $\mathbf{0},$ which is less than one. See Figure \ref{seconstep} for an illustration of this exploration starting at $v_{1}$. 

This procedure can be iterated for the remaining children belonging to the first generation and later iterated for the remaining generations. This defines a subcritical branching process: the distribution of the number of children of each parent in each generation is always the same,  independent of other copies, and has an expected value strictly less than one. We now define the following coupling between the frog model and $\text{EP}(\mathbf{0})$
to show that the number of particles activated by the frog model $\bigcup_{x: \mathbf{0}\rightarrow x}\mathcal{P}(x),$ is dominated by the number of particles explored by $\text{EP}(\mathbf{0})$

 In the frog model, we have that the number of particles at each $x\in V$ is distributed as $\eta_x \sim \text{Po}(\lambda).$ For the exploration process, at the beginning, we define $\eta^{\text{EP(\textbf{0})}}_{x} := \eta_x,$ for each $x \in V$. Furthermore, for each $x \in V$ in the frog model, the particles $(\omega^{i}_{x})_{1\leq i \leq \eta_{x}}$ perform continuous-time simple random walks with lifespan $t.$ We also require that the particles of the exploration process perform continuous-time simple random walks according to the ones sampled for the frog model, that is, $(\omega^{i}_{x})^{\text{EP(\textbf{0})}}_{1\leq i \leq \eta^{\text{EP(\textbf{0})}}_{x}}:= (\omega^{i}_{x})_{1\leq i \leq \eta_{x}}.$

Suppose that the process has not yet terminated (otherwise we are done) and let $v_{1},...,v_{\ell}$ denote the children belonging to the first generation. For $v_{j}$ with $1 \leq j \leq \ell$, we consider the following. For sites $y \in V$ that were previously explored at the iteration starting at $v_{j-1}$ (for $v_{1}$ we consider the vertices that were explored at the generation zero, starting at $\mathbf{0}$), we refresh their number of particles with new independent $\text{Po}(\lambda)$ denoted by $\eta_{y}^{\prime}$. We also sample completely new and independent random walks for each particle $ (\omega^{i}_{y})_{1\leq i \leq \eta_{y}^{\prime}}^{\prime}.$ For every unexplored vertex $y$, we do not alter them and their respective particles' random walks.  We further iterate this procedure for the subsequent generations until the process terminates or continues indefinitely. Given this coupling, observe that in order to prove that $|\mathcal{C}_{\mathbf{0}}|$ is stochastically dominated by $|S|\sum_{i=0}^{\infty}\mathcal{Z}_{i}$, we need to prove that if $\mathbf{0}\rightarrow x$ in the frog model, then $x$ is explored by $\text{EP}(\mathbf{0})$.

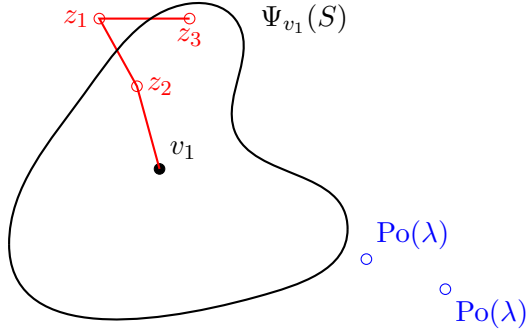
\begin{figure}[t!]
\centering
\begin{tikzpicture} 
\coordinate (a1) at (0,0);
\coordinate (a2) at (3.3,-0.7);
\coordinate (a3) at (4.5,0.2);
\coordinate (a4) at (3,1.5);
\coordinate (a5) at (3,3);
\coordinate (a6) at (2.4,3.2);
\coordinate (a7) at (1,2);

\filldraw[red, fill=white] (1.2,3) circle (2pt) node[anchor=east]{$z_{1}$};

\filldraw[red, fill=white] (2.4,3) circle (2pt) node[anchor=north]{$z_{3}$};

\filldraw[black, thick] (3.2,3) node[anchor=west]{$\Psi_{v_{1}}(S)$};

\filldraw[black] (2,1) circle (2pt) node[anchor=south west]{$v_1$};

\filldraw[blue, fill=white] (4.75,-0.2) circle (2pt) node[anchor=south west]{$\text{Po}(\lambda)$};

\filldraw[blue, fill=white] (5.8,-0.6) circle (2pt) node[anchor=north west]{$\text{Po}(\lambda)$};

\filldraw[red, fill=white] (1.7,2.1) circle (2pt) node[anchor= west]{$z_{2}$};

\draw[red,thick]  (2,1)   -- (1.7,2.1) -- (1.2,3)--(2.4,3) ;

\draw[use Hobby shortcut,thick] ([out angle=-90]a1)..(a2)..([in angle=-90]a3); 
\draw[use Hobby shortcut,thick] ([out angle=90]a3)..(a4)..(a5)..(a6)..(a7)..([in angle=90]a1);

\end{tikzpicture}
\caption{We now illustrate the iteration starting at $v_1$, which was one of the vertices counted as a child in the first generation. We shifted the set $S$ to $v_{1}$ and denoted that as $\Psi_{v_{1}}(S)$. Vertices that had their particles revealed at the previous iteration will be refreshed with $\text{Po}(\lambda)$ particles performing i.i.d. continuous time simple random walks, described in the drawing by the empty blue circles. In the example above, $v_{1}$ only had one child that left the set $S$. In this case, $z_{1},z_{2}$ and $z_{3}$ are counted as the children of $v_{1}$. The total size of $\mathcal{Z}_{2}$ will be the sum of the children of $v_{1},v_{2},v_{3}$ and $v_{4},$ each one of them having the same independent distributions with mean less than one. }\label{seconstep}
\end{figure}

\textbf{Proof of stochastic domination:} For the event $\mathbf{0} \rightarrow x$ to occur in the original frog model, there must exist a sequence of vertices $\{ x_{0},...,x_{k}\},$ with $x_{0} = \mathbf{0},$ and $x_{k} =x,$ such that for every $0 \leq j \leq k$ there exists a particle originally from $x_{j}$ that goes to $x_{j+1}$,  previously denoted as $x_{0} \Rightarrow x_{1} \Rightarrow ... \Rightarrow x_{k-1} \Rightarrow x_{k}$. We consider the following two cases. Either we have that $\mathbf{0} \xrightharpoonup{S} x,$ and then $x$ will be explored by the end of generation zero, concluding the claim. Otherwise, let $1\leq j\leq k$ be the first $j$ such that we have $\mathbf{0} \xrightharpoonup{S} x_{j-1},$ but we do not have that $\mathbf{0}  \xrightharpoonup{S} x_{j}.$ In this case, by the definition of the process, we have that $x_{j}$ will be one of the children belonging to the first generation. Now we consider at the first generation of $\text{EP}(\textbf{0}),$ the exploration process starting at $x_{j}$ and centered on $\Psi_{x_{j}}(S)$. It might be the case that the particles in $x_{j}$ have been previously explored by one of the other children belonging to the first generation. However, as previously observed, the frog model is an Abelian process, so the order in which we do the exploration does not matter. In any case, it only matters that when we begin the exploration at $x_{j}$, it is guaranteed that $x_{j}$ already had or will have its (original from the frog model) particles' trajectories explored, and thus, one of them will reach $x_{j + 1},$ since $x_{j} \Rightarrow x_{j+1}$. Therefore, by the end  of the first generation of the exploration process, we reveal at least one more vertex along the sequence $\{ x_{0},...,x_{k}\}$. This clearly can be further iterated, and hence the claim follows as $x_{k}$ is explored at the latest in the $(k+1)-$th generation of $\text{EP}(\mathbf{0})$. We leave the remaining details to the reader.

\begin{proof}[Proof of items (3) and (4) of Theorem \ref{mainsharpness}:]
We first prove item (3). Fix an arbitrary $t > 0$ and let $\lambda < \tilde{\lambda}_c(t)$.
Then, the exploration process $\text{EP}(\mathbf{0})$ can be described as a subcritical branching process, with the expected number of children of a single parent being at most $\tilde \phi_{\lambda,t}(S) < 1$. We denote the size of the $n$-th generation by $\mathcal{Z}_{n}$. 
Note that $\mathcal{Z}_1$ is stochastically dominated by $\sum_{x\in S} \sum_{i = 1}^{\eta_{x}} (\lvert \mathrm{R}^{i}_{x} \rvert {-} 1)$, whose probability generating function $f(s)$ satisfies that $f(s) \le e^{|S|\lambda(e^{t(s-1)}-1)}$ for all $s > 0$. Then, we conclude using Exercise 5.33 of \cite{MR3616205} that there exists some $\varepsilon$ such that $\bE_{\mathrm{EP(\mathbf{0})}}[(1+\varepsilon)^{\sum_{i} \mathcal{Z}_i}] < \infty$. Hence, due to the stochastic domination established above together with the Markov's inequality, we have
\begin{equation}\label{eq: exp_decay}
\bP_{\lambda,t}( |\mathcal{C}_{0}| \geq n) \le \bP_{\text{EP}(\mathbf{0})}\bigg(\sum_{i=0}^{\infty} \mathcal{Z}_{i} \geq \frac{n}{|S|}\bigg)\leq \frac{\bE_{\text{EP}(\mathbf{0})} \left[(1+\varepsilon)^{\sum_{i\geq0}\mathcal{Z}_{i}} \right] }{(1+\varepsilon)^{n/|S|}} \, ,
\end{equation}
proving item $(3)$ of Theorem \ref{mainsharpness}.




We now prove item (4) of Theorem \ref{mainsharpness}. Note that since, as previously commented, $\phi_{\lambda, t}(S)$ might not be monotone with respect to $t$, we overcome this by slightly modifying the argument.
Fix $\lambda>0$. Then, for any $t$ such that $t < \tilde{t}_c(\lambda)$, by the definition of $\tilde{t}_c(\lambda)$, there exists some $t_0 \in [t, \tilde{t}_c(\lambda)]$ such that $\phi_{\lambda, t_0}(S) \leq \mathrm{c}(\Delta,\lambda,t_0)$ for some finite set $S(t_{0})$. Analogously to the proof above, this implies that there exist constants $C = C(t_0)$ and $c = c(t_0)$ such that
\begin{equation*}
    \bP_{\lambda,t}\left(|\mathcal{C}_{\mathbf{0}}| \geq n\right) \leq \bP_{\lambda,t_0}\left(|\mathcal{C}_{\mathbf{0}}| \geq n\right) \leq Ce^{-cn}, \text{ for every $n\geq0$},
\end{equation*}
which concludes the proof of item (4) of Theorem \ref{mainsharpness}.
\end{proof}

This concludes the sharpness of the phase transition for directed vertex-transitive graphs. We now go back to the undirected graph setup for the proofs of the existence of the phase transition.

\section{Existence of the Phase Transition} \label{sec: existence}

The main objective of this section is to show that $t_{\mathrm{c}}(\lambda) < \infty$ holds for all $\lambda > 0$ on two types of graphs, namely quasi-transitive graphs of superlinear polynomial growth and non-amenable graphs.
In each of the following two subsections, we prove the desired statements through widely different strategies, that exemplify the graphs distinct behaviors. 

\subsection{Quasi-Transitive Graphs of Superlinear Polynomial Growth}\label{subsec_poly}

In this subsection, we show for quasi-transitive graphs of superlinear polynomial growth, that  $t_{\mathrm{c}}(\lambda) < \infty$ for all $\lambda >0.$
Recall that $g(n) := |B_{\boldsymbol{0}}(n)|$ and that $G$ is of polynomial growth if there exists a constant $C >0$ and $d \geq 1$ such that for all $n \geq 1$, we have $g(n) \leq C n^d$. To begin our discussion, we also need the definition of quasi-isometry:
\begin{definition}
    Let $(X_{1},d_{1})$ and $(X_{2},d_{2})$ be two metric spaces. A map $f: X \rightarrow Y$ is called a \textbf{quasi-isometry} if there exists $K>0$ such that 
    \begin{itemize}
        \item For all $z,y \in X_{1},$ we have that 
        $$ K^{-1}d_{1}(z,y) - K \leq d_{2}(f(z),f(y)) \leq Kd_{1}(z,y) + K \, . $$
        \item For all $x_{2} \in X_{2}$ there exists some $x_{1} \in X_{1}$ such that $d_{2}(x_{2},f(x_{1})) \leq K.$
    \end{itemize}
We say that $X_{1}$ is quasi-isometric to $X_{2}$ if there exists a quasi-isometry $f:X_{1}\rightarrow X_{2}.$
\end{definition}

The celebrated theorems of Gromov and Trofimov \cite{gromov_poly_growth, Trofimov1985GRAPHSWP} imply that a vertex-transitive graph of polynomial growth is always quasi-isometric to a Cayley graph of a finitely generated nilpotent group. Moreover, an important consequence of combining this with Bass-Guivarc'h formula (see \cite{bass, MR272943}) is that for any vertex-transitive graph of polynomial growth, there exists an integer \(d \) and a constant $C$ such that
\begin{equation}\label{gromov}
\frac{1}{C}n^d \leq g(n) \leq Cn^d , \quad  \forall n \geq 1 \, .
\end{equation}
The results above can be generalized to quasi-transitive \footnote{Equation \eqref{gromov} holds for $g_{x}(n)$ for every $x \in V.$ } graphs (for example, see Theorem 5.11 of \cite{MR1743100}). We call this integer \(d\) the \emph{growth exponent} of \(G\). A more detailed discussion on quasi-transitive graphs with polynomial growth can be found in the appendix. 

 Now, recall the definition of $p_{c}^{\text{site}}(G)$. For a graph $G:= (V,E),$ we consider site percolation on $G,$ that is, each vertex of $G$ is independently colored black (respectively white) with probability $p$ (respectively $1-p$), and denote the measure associated to this process as $\bP_{p}^{\text{site}}$. We denote by $\mathcal{C}_{x}^{\text{site}}$ the connected component of black vertices connected to $x$ (if $x$ is white, this set is empty). We define $p_{c}^{\text{site}}(G):= \inf \{p: \bP^{\text{site}}_{p}(\exists x, \hspace{0.5em} |\mathcal{C}^{\text{site}}_{x}|= \infty)>0 \}.$ We now introduce the idea of a group net, inspired by \cite{MR4529920}, which will play an important role in the proof of the existence of the phase transition.
 
\begin{definition} \label{def: nets}
Let $V_0 \subseteq V$. We say $V_0$ is \textbf{$a$-separated} if $d_G(x, y) \ge a$ for all $x, y \in V_0$. For $b\geq 1,$ we say that $V_{0}$ forms an \textbf{$(a, b)$-group net} of $G$ if the auxiliary graph $ G_{0}:= (V_{0},E_{0}),$ where
\begin{align*}
	E_0 := \big\{ \{x, y\}: x,y \in V_0, \, d_G(x, y) \le 4b \big\},
\end{align*}
satisfies that $p^{\mathrm{site}}_c(G_0) \le 3/4$.
\end{definition}

\begin{remark}
Although Definition \ref{def: nets} was inspired by the definitions of $(a, b)$-nets and controlled-nets introduced in \cite{MR4529920}, it requires fewer conditions. Unlike \cite{MR4529920}, we do not require the net to be dense, as this condition is not necessary for the proof. Additionally, while \cite{MR4529920} considers a sequence of graphs that converges locally, we focus on a single fixed graph $G$. In this regard, results from \cite{MR4529920} such as Lemma $2.3$ regarding the existence of nets also imply the existence of group nets. We also highlight that, as pointed out in \cite{MR4529920},  the choice of $3/4$ is arbitrary and can be replaced by any value strictly less than $1$.
\end{remark}

We now show the existence of $(a, b)$-group nets on quasi-transitive graphs of polynomial growth, for some intrinsic parameters $\beta$ depending only on the graph. Group nets will be fundamental to our renormalization-type argument, since they provide the necessary structure to be able to apply such ideas. The following proof strongly relies on the fact that we have a graph of polynomial growth and is essentially a consequence of Lemma $2.3$ of \cite{MR4529920}, which proved the result for Cayley graphs of nilpotent groups with $p^{\text{site}}_{c}(G)<1$.

\begin{proposition} \label{prop:net-polynomial}
Let $G$ be a quasi-transitive graph of superlinear polynomial growth. Then, we can find a constant $\beta = \beta(G)$ such that, for all $a \ge 1$, there exists an $(a, \beta a)$-group net of $G$. 
\end{proposition}

\begin{proof}

By Theorem 5.11 of \cite{MR1743100}, $G$ is quasi-isometric to a Cayley graph $\mathcal{H}$ of a finitely generated nilpotent group, which also has superlinear growth. By Lemma $2.3$ of \cite{MR4529920}, for all $a\geq1$ there exists an $(a,2a)$-group net of $\mathcal{H}$.  By the definition of quasi-isometry, there exists a map $\phi: \mathcal{H} \rightarrow G$ such that 
\begin{align}
    \frac{1}{K} d_H(x, y) - K \le d_G(\phi(x), \phi(y)) \le K d_H(x, y) + K \, , \label{eq:quasi-isometry-net}
\end{align}
for some $K \ge 1$. Fix $a \ge 1$ and consider a $(K(K+a), 2K(K+a))$-group net of $\mathcal{H}$, and its induced graph denoted by $G_0 = (V_0, E_0)$. Let $V_1 = \phi(V_0)$. It follows from the lower bound in Equation \eqref{eq:quasi-isometry-net} that $V_1$ is $a$-separated. Let $\{u, v\} \in E_0$, and observe that $d_H(u, v) \le 8K(K+a)$. By the upper bound in Equation \eqref{eq:quasi-isometry-net}, we have $d_G(\phi(u), \phi(v)) \le (8 K (K + a)) K + K \le 4 (4 K^3 + K) a$, where the last inequality uses the fact that $K+a \le 2Ka$ when $a, K \ge 1$. Define $E_1 := \{ \{x, y\} : x, y \in V_1, \, d_G(x, y) \le 4 (4 K^3 + K) a \}$ and let $G_1 = (V_1, E_1)$. Furthermore, observe that 
$$G_2 := (V_1, \{ \{ \phi(x), \phi(y) \} \}_{xy\in E_0})$$ 
is a subgraph of $G_1,$ and $G_0$ is isomorphic to $G_2$. Therefore, we must have $p^{\mathrm{site}}_c(G_1) \le p^{\mathrm{site}}_c(G_0) \le 3/4$.
We conclude that $V_1$ is an $(a, (4 K^3 + K) a)$-group net of $G$. 
\end{proof}

Let $V_0$ be an $(a, \beta a)$-group net of $G$. Since $V_0$ is $a$-separated, $\mathcal{B} = \{B_v(a/3)\}_{v \in V_0}$ forms a collection of disjoint balls. We now consider the induced graph $G_0 = (V_0, E_0)$, where $E_0 = \{ \{x, y\}: x, y \in V_0, \, d_G(x, y) \le 4\beta a \}$. 
Our goal is to construct an auxiliary supercritical Bernoulli percolation on the vertices $V_0$, so that their openness depends only on particles inside the ball $B_v(a/3)$. The nets will play a fundamental role in the proof of item $(2)$ of Theorem \ref{main_quasi:existence}. Indeed, they give us sufficient structure for a renormalization argument. The next bounds do not rely on the nets, but will be later combined with them to show the existence of the phase transition. We also remark that it is not known in general in which classes of graphs this type of renormalization structure can be applied, but the results of \cite{MR4529920} rely heavily on the polynomial growth assumption. We now introduce a few definitions to prove quantitative results that will be fundamental for our proof. Here, we will use results from the appendix regarding the trace of random walks on quasi-transitive graphs of superlinear polynomial growth.

Recall from \cref{introduction} that $x$ is connected to $y$ in $B$, denoted by $x\xrightarrow[]{B} y$, if there exists a chain of infection from $x$ to $y$ that only uses particles whose initial positions are in $B$. Let
\begin{align} \label{def: descendent_set_restricted} 
    \mathcal{A}_x^{ B}: = \{y \in B: x\xrightarrow[]{B} y\}
\end{align}
represent the collection of such vertices. Notice that this quantity depends on the parameters $\lambda$, $t$, but since they are clear from the context, we simply omit them.

\begin{definition} \label{def: nice_vertex}
    Let $B \subseteq V$ be a finite vertex set and $x \in B$. We say $x$ is a \textbf{good} (or $(B, \lambda, t)$-\textbf{good}) vertex if 
    \begin{align*}
    \Big|\mathcal{A}_x^{\lambda, t, B} \Big| \ge |B|/4 \, .
    \end{align*}
    Denote the set of $(B, \lambda, t)$-good vertices by $\mathcal{G}_{B}$.
\end{definition}

For a fixed ball $B_x(a)$, we study the existence of $\left(B_x(a), \lambda, a^2 \right)$-good vertices. The following proposition establishes that the probability of not finding a good vertex in a subset $A \subseteq B_x(a)$ decreases exponentially with respect to the size of $A$.

\begin{proposition} \label{prop: poly_internal} 
    Let $G$ be a quasi-transitive graph that is of superlinear polynomial growth. Then, there exist positive constants $C = C(\lambda, G)$ and $c = c(\lambda,  G)$ such that, for any $x \in V$, $a > C^{\star}(\lambda,G)>1$, and $A \subseteq B_x(a)$, we have
    \begin{align}
        \bP_{\lambda, a^2} \left( \mathcal{G}_{B_x(a)} \cap A = \varnothing \right) \le C \exp(-c \, |A|) \, . \label{ineq: int-dynamic_uppper_bound}
    \end{align}
\end{proposition}

In order to prove this proposition, we need the following lemma about the range of a single random walk trajectory. Recall that $\mathrm{R}_y(t) := \{ Y_y(s) : 0 \leq s \leq t \}$ denotes the range of a continuous-time simple random walk starting at $y$ up to time $t$, and $\bP_y$ denotes the probability measure associated with it. The proof of Lemma \ref{lemma: random_walk_range} will be given in the appendix.

\begin{lemma} \label{lemma: random_walk_range}
Let $G$ be a quasi-transitive graph of polynomial growth with growth exponent $d\geq2$.
Consider a continuous-time simple random walk on $G$. 
There exist positive constants $c_0 = c_0(G)$ and $c_1 = c_1(G)$ such that the following is true.
Let $x \in V$ and $a > 1$.
For any $H \subseteq B_x(a)$ with $|H| \leq \frac{3}{4}|B_{x}(a)|$  and any $y \in B_x(a)$, if we define $\widetilde{\mathrm{R}}^{\text{H}}_y(t) := \mathrm{R}_y(t) \cap (B_x(a) \setminus H)$, \footnote{Here we omit the dependency of $\mathrm{R}_{y}^{i}(t)$ on the particle label since this is true for any random walk starting at $y$. We also omit the dependency on $y$ when it is clear from the context.} then

\vspace{2mm}

\begin{align*}
    \bP_y \left( |\widetilde{\mathrm{R}}^{\text{H}}(a^2)| \geq c_1 \cdot \theta(a) \right) \geq c_0 \, ,
\end{align*}
where $\theta(a)$ is defined as 
\begin{align}\label{functionf}
    \theta(a) := \begin{cases}
        a^2 & \text{if } d \geq 3, \\
        a^2 / \log a & \text{if } d = 2.
    \end{cases}
\end{align}
\end{lemma}
Now, let us turn ourselves back to the proof of the proposition.
\begin{proof}[Proof of Proposition \ref{prop: poly_internal}]
    We first observe that it suffices to prove the inequality \eqref{ineq: int-dynamic_uppper_bound} for sets $A$ satisfying $|A| < |B_x(a)|/4$ and for $a$ larger than some constant $C'$ depending on $G$ and $\lambda$. Indeed, for sets with $|A| \geq |B_x(a)|/4$, we can simply apply the proposition to a subset of $A$ containing exactly $\lfloor|B_x(a)|/4\rfloor$ vertices. The inequality \eqref{ineq: int-dynamic_uppper_bound} then holds for all $A \subseteq B_x(a)$ with the constant $c$ in \eqref{ineq: int-dynamic_uppper_bound} being replaced by $c/5$, provided that $a \geq C'$ and $|B_x(a)| \geq 20$. We can then extend the proposition to all $a>1$  by appropriately adjusting the constant $C$ in \eqref{ineq: int-dynamic_uppper_bound}. 

    Consider the frog model with particle density $\lambda$ and lifespan $a^2$. Our goal is to find a vertex $z$ in $A$ such that 
    \begin{align*}
        |\mathcal{A}_z^{\lambda, a^2, B_{x}(a)}| \geq |B_{x}(a)|/4 \, ,
    \end{align*}
    which would make $z$ a good vertex. We shorten the notation $\mathcal{A}_z^{\lambda, a^2, B_{x}(a)}$ by $\mathcal{A}_z$ throughout the rest of this section.

  Fix an arbitrary order in the set $B_{x}(a).$ Let $\big\{v_{1},\dots,v_{|A|}\big\}$ be the ordered sequence of vertices in A. 
  We define an exploration process of the vertices of $B_{x}(a)$ to determine a subset $\mathcal{G}_{B_{x}(a)} \cap A$.  We start by considering the realization of the frog model at $v_{1}$ with no particles outside $B_{x}(a)$ and no particles inside $A.$ Initially, we only explore the descendants of $v_{1}$ outside of the set $A.$  When all its descendants (and descendants of descendants) are fully explored, we continue this exploration process by considering the particles of $v_{2}$ as active and exploring its descendants outside of $A$ and inside $B_{x}(a),$ now also not using the vertices that were revealed in the previous iteration. We further iterate this procedure until $v_{|A|}$. We now make this precise by defining the following three sets and, to avoid a heavier notation, we remark now that we only consider vertices inside $B_{x}(a)$. We start with the following verbal description, but we note that precise inductive definitions are also given throughout the proof:

   \begin{itemize}[left=20pt, itemsep=0.5em, topsep=0.5em, rightmargin=25pt] 
        \item[-] Let $\mathcal{V}_{k}$ denote the set of vertices that had their trajectories fully revealed upon the completion of stage $k-1$;
        \item[-] Let $\mathcal{M}_k$ denote the set of active vertices which will eventually have their trajectories fully revealed  after or at the end of stage $k$;
        \item[-] Let $\mathcal{I}_{k}$ denote the incremental set, that is, the set of vertices outside of $A$ that have been visited by the process for the first time upon completion of the stage $k-1$.
    \end{itemize}

In the beginning, we set $\mathcal{M}_{0} = \{v_{1}\},$ and $\mathcal{I}_{0}=\mathcal{V}_{0} = \varnothing.$ To aid our construction, we further introduce the sets $\mathcal{U}$ and $\mathcal{W}$ which are taken to be equal to $A$ in the beginning $\mathcal{U}_{0} = \mathcal{W}_{0} = A$.

For $k\geq 1$, after $k-1$ stages have been completed, we define the sets in the following way. If $\mathcal{M}_{k-1} \neq \varnothing$ we take $y_{k} \in \mathcal{M}_{k-1}$ and explore its range $\mathcal{R}_{y_{k}},$ further setting $\mathcal{V}_{k} = \{y_{1},...,y_{k}\}$. Otherwise, if $\mathcal{M}_{k-1} = \varnothing$, the process has terminated, and there are no current active vertices. 

In the case where the process has not terminated, we let $\mathcal{I}_{k} = \mathcal{R}_{y_k}\setminus \mathcal{W}_{k-1}$, representing the newly explored vertices, and we have that the auxiliary sets are given by $\mathcal{W}_{k}= \mathcal{W}_{k-1} \cup \mathcal{I}_k$ and $\mathcal{U}_{k} = \mathcal{W}_{k}\setminus \mathcal{V}_{k}.$ Finally, we observe the set of active vertices are given by $\mathcal{M}_{k} =\mathcal{U}_{k}\setminus A,$ if $\mathcal{U}_{k}\setminus A \neq \varnothing$, and the first vertex of the set $A\setminus \mathcal{V}_k$ otherwise. Observe that this process terminates after all vertices in $A$, and their respective descendants, have been explored.

 We will show that the size of the incremental set stochastically dominates a Bernoulli random variable multiplied by a positive constant. Moreover, the mean of this Bernoulli random variable depends only on $\lambda$ and $G$, while the multiplying constant can be made arbitrarily large by choosing $a$ sufficiently large. Therefore, for fixed $\lambda$, if $a$ is large enough, the mean increment becomes substantial. This observation is crucial to prove that the failure to find a good vertex in $A$ is a large deviation event when $|A|$ is sufficiently large.

    Let $\cF_0$ be the trivial $\sigma$-field and, for $k \ge 1$, let $\cF_k := \sigma(\{ \mathcal{I}_{j} \}_{j = 1}^{k})$. Thus, $\{ \cF_k \}_{k \ge 0}$ forms a filtration. We define two stopping times with respect to $\{ \cF_k \}_{k \ge 0}$
    \begin{align*}
        \tau_1 &:= \inf \{ k: \cU_k = \varnothing \} \, , \\
        \tau_2 &:= \inf \bigg\{ k: |\mathcal{W}_k| > \frac{3 |B_{x}(a)|}{4} \bigg\} \, ,
    \end{align*}
    and let $\tau := \min \{ \tau_1, \tau_2 \}$. We also let $\Delta_k := |\mathcal{I}_{k}|$, be the size of the $k$-th incremental set. Moreover, we define $S_k := \sum_{i=1}^{k\wedge \tau_{1}} \Delta_i$ as the total size of the incremental sets up to  $\min \{k,\tau_{1}\}$. We now provide a sufficient condition for the existence of a good vertex in $A$.

    \begin{claim} \label{claim: internal_exploration}
        Assume $A$ is non-empty and $S_k \ge k$ for all $|A|\leq k \leq \tau$. Then, there exists a good vertex $z^\star$ in $A$.
    \end{claim}
    
    \begin{proof}[Proof of Claim \ref{claim: internal_exploration}]
First, suppose that $\tau_{2}< |A|.$ 
Since $A$ is non-empty, the exploration process does not terminate immediately. 
Moreover, since we explore exactly one vertex in each stage, we must have $\tau_1 \geq |A|$, meaning the exploration process must complete at least $|A|$ stages, so in that case $\tau=\tau_{2}$. 
Furthermore, observe that 
$$|\mathcal{U}_{\tau} \setminus A| \geq S_{\tau} - \tau+1 \geq |A|+S_{\tau} - 2|A|+1 \geq \frac{1}{4}|B_{x}(a)| \, ,  $$
and since this implies that $|\mathcal{M}_{\tau}| \geq \frac{1}{4}|B_{x}(a)| $ we must have that there exists a good vertex in $A$.

      Now, consider the case when $\tau_{2} \geq |A|.$  We show that if $S_{k} \geq k$ for every $|A| \leq k \leq \tau$, then it is also true that $\tau = \tau_2$.  Observe that for any $|A|\leq k \leq \tau$ such that $S_k \geq k$, we have
        \begin{align*}
            |\mathcal{U}_k \setminus A| \geq S_k - k+1 \geq 1 \, ,
        \end{align*}
        which implies that the process continues to the $(k+1)$-th stage and therefore $\tau = \tau_2$. Furthermore, this also implies that for $|A| \leq k \leq \tau,$ the set of active particles is $\mathcal{M}_{k}= \mathcal{U}_{k} \setminus A$ since we can always pick a new vertex from $\mathcal{U}_{k} \setminus A,$ in particular having $\mathcal{M}_{|A|}= \mathcal{W}_{|A|}\setminus (\mathcal{V}_{|A|} \cup A)$. Since in the steps where $|A|\leq k\leq \tau$, the revealed vertices are necessarily outside of $A,$ we must have that they all belong to the exploration of the same vertex in $A,$ say $z^{\star},$ which implies that
        \begin{align*}
            \mathcal{W}_\tau \setminus \Bigl( \{y_1, \ldots, y_{|A|}\} \cup A \Bigr) \subseteq \mathcal{A}_{z^\star} \, .
        \end{align*}
        This concludes the proof, because we know that  $|\mathcal{W}_\tau| > 3|B_{x}(a)|/4$ which combined with the above claim implies that $|\mathcal{A}_{z^\star}| > 3|B_{x}(a)|/4 - 2|A| > |B_{x}(a)|/4$ since we assumed that $|A| < |B_{x}(a)|/4$.
\end{proof}

    We now continue the proof of Proposition \ref{prop: poly_internal}. Recall the definition of $\theta(a)$ in Equation \eqref{functionf}. By Lemma \ref{lemma: random_walk_range} and Poisson thinning on the event $\{ \tau \ge k \} \in \cF_{k-1}$, we have that
    \begin{align}
        \bP_{\lambda, a^2} \left( \Delta_{k} \ge c_1 \cdot \theta(a), \, \tau \geq k  \;\middle|\;
      \mathcal{F}_{k-1} \right) \ge 1 - \exp(-c_0 \lambda) \, , \label{ineq: single_exploration_step_increment_bound}
    \end{align}
    where $c_0$ and $c_1$ are constants depending only on $G$ (the same constants as in \cref{lemma: random_walk_range}). Let $q(\lambda) := 1 - \exp(-c_0 \lambda)$ and let $\{ X_k \}$ be a sequence of i.i.d. Bernoulli random variable with mean $q$. 
    Observe that on the event that $\tau \ge k$, we have that $\{\Delta_{i}\}_{1\leq i \leq k}$ stochastically dominates $\{c_1 \theta(a) X_i\}_{1\leq i \leq k}$. Define $T_k := c_1 \theta(a) \sum_{i=1}^{k} X_i$, then it follows from stochastic domination that 
    \begin{align}
        \bP_{\lambda, a^2}\big( S_k < k, \, \tau \ge k \big) \, \leq \, \bP\big( T_k < k\big) \label{eq: S_k_stoch_dom_T_k}
    \end{align}
    holds for any $k\geq 1$. Therefore, it follows from Claim \ref{claim: internal_exploration} that
    \begin{equation}\label{eq: intermediate_sum}
    \begin{split}
    	\bP_{\lambda, a^2} \left( \mathcal{G}_{B_{x}(a)} \cap A = \varnothing \right)  
        &\le \; \bP_{\lambda, a^2} \bigg( \exists \, k \in \big\{ |A|, ..., \tau \big\} \text{ s.t. } S_k < k  \bigg)  \; \\
        & \leq\bE_{\lambda, a^2} \left[ \sum_{k = |A|}^{|B_{x}(a)|} \mathbbm{1}_{\big\{S_k < k, \, \tau \geq k \big\}}  \right] \\
        &\le \; \sum_{k = |A|}^{|B_{x}(a)|} \bP\big( T_k < k \big) \, ,
        \end{split}
    \end{equation}
    where we use the fact that $\tau \le |B_{x}(a)|$ and  Equation \eqref{eq: S_k_stoch_dom_T_k}. Moreover, 
    \begin{align*}
        \bP \big( \, T_k < k \big) = \bP \left( \sum_{i=1}^{k} X_i < \frac{k}{c_1 \theta(a)} \right). 
    \end{align*}
    Now, choose $a$ sufficiently large (depending on $\lambda$ and $G$) such that $c_1 \theta(a) \cdot q \ge 2$. In particular, if $a \geq \max \big\{ \frac{8}{c_{1}c_{0} \lambda}, \frac{8} {c_{1}},8 \big\}$, then, for any $k$, we have
    \begin{align}\label{firstchernoff}
    	\bP \left( \sum_{i=1}^{k} X_i< \frac{k}{c_1 \theta(a)} \right) \le \bP \left( \sum_{i=1}^{k} X_i  < \frac{kq}{2} \right) \le \left(\frac{2}{e}\right)^{kq/2} ,
    \end{align} 
    where the last inequality is a direct application of Chernoff's inequality. Therefore, for any $a> C^{\star}(\lambda,G)$, \eqref{eq: intermediate_sum} can be upper bounded by:
    \begin{align*}
        \sum_{k = |A|}^{\infty} \left(\frac{2}{e}\right)^{kq/2} \leq C \exp({-c|A|}) \, ,
    \end{align*}
    for some choice of $C(\lambda,G),c(\lambda,G)>0$.

\end{proof}

To motivate the following proposition, we provide an overview of our strategy to construct auxiliary supercritical Bernoulli site percolation on an $(a, \beta a)$-group net.
Initially, we have $\mathrm{Po}(\lambda)$ number of particles at each vertex. Using Poisson thinning, we split these particles into two disjoint sets, $\mathcal{P}^{(1)}(x)$ and $\mathcal{P}^{(2)}(x)$, with each of them containing $\mathrm{Po}(\lambda/2)$ particles.

Let $V_0$ be an $(a, \beta a)$-group net of $G$. Thus, $\mathcal{B} = \{B_v(a/3)\}_{v \in V_0}$ is a collection of disjoint balls. We say that $B_v(a/3)$ and $B_w(a/3)$ are \textit{neighbors} if $\{v, w\} \in E_0$, recalling that we say $(v,w)\in E_0$ if $d(v,w)\leq 4\beta a$.
Our argument proceeds in two phases:

We begin by applying Proposition \ref{prop: poly_internal} to the particles of
$\mathcal{P}^{(1)}(x)$ in each of the balls $B_x(a/3)$, with the set $A$ being the
entire ball.
This implies that the probability of having no good vertices in
$B_x(a/3)$ is at most $e^{-c|B_x(a/3)|}$.
By taking $a$ large, we can make this as small as we wish, uniformly in $x$.

We next wish to show that a good vertex in $B_x(a/3)$ is very likely to
activate all vertices in nearby balls of the group net.
To this end, we will use the independent set of particles $\mathcal{P}^{(2)}(x)$.

Define
\[
   \widehat B_v := \bigcup_{\{v,w\}\in E_0} B_w(a/3) \, .
\]
Also, define the union of trajectories started from any vertex of a set
$D$ by
\[
   \mathcal{R}_D := \bigcup_{x\in D} \mathcal{R}_x \, .
\]

The next proposition implies that the second step has a high probability of success. Below, $\theta(a)$ is as in Equation \eqref{functionf}.

\begin{proposition} \label{prop: poly_neighboring}
    Let $G$ be a quasi-transitive graph of superlinear polynomial growth. 
    Let $V_0$ be an $(a, \beta a)$-group net on $G$. 
    Then, there exist positive constants $C' = C'(G, \lambda, \beta)$ and $c = c(G, \beta)$ such that for any $a > 2$, $v \in V_0$, and $D \subseteq B_v(a/3)$ with $|D| \geq |B_v(a/3)| / 4$, we have
    \begin{align} \label{eq: poly_neighboring_result}
        \bP_{\lambda, a^2} \left( \widehat{B}_v \not\subseteq \mathcal{R}_D \right) \leq C' e^{- c \lambda \theta(a)} \, .
    \end{align}
\end{proposition}

We will apply the above proposition to the second set of particles $\mathcal{P}^{(2)}(x)$, which has density $\lambda/2$, but we prove the theorem for an arbitrary $\lambda>0$. In order to prove Proposition \ref{prop: poly_neighboring}, we use the following auxiliary result for random walks on quasi-transitive graphs of superlinear polynomial growth.

\begin{lemma} \label{lemma: random_walk_expected_hitting_size_main_text}
    Let $G$ be a quasi-transitive graph of superlinear polynomial growth.
    Consider a continuous-time simple random walk on $G$. Let $K_1, K_2 > 1$. 
    Then, there exists a constant $c = c(G, K_1, K_2)$, such that for any $a \ge 2$, $x \in V$, and $D \subseteq B_x(K_1 a)$ with $|D| \ge |B_x(a)|/K_2$, we have

    \vspace{2mm}

\begin{align} \label{lemma: random_walk_expected_hitting_size_switched}
        \sum_{y \in D} \bP_y \left( \tau_x \le a^2 \right) \ge c\cdot \theta(a) \, .
\end{align}
\end{lemma}
Lemma \ref{lemma: random_walk_expected_hitting_size_main_text} follows directly from Lemma \ref{apx_lemma: random_walk_expected_hitting_size}. Indeed, applying Lemma \ref{apx_lemma: random_walk_expected_hitting_size} with the roles of $x$ and $y$ interchanged yields inequality \eqref{lemma: random_walk_expected_hitting_size_switched} with $x$ and $y$ swapped. Subsequently, since $G$ is quasi-transitive, we restore the desired ordering of $x$ and $y$ by paying a multiplicative constant depending on the maximum degree.
\begin{proof}[Proof of Proposition \ref{prop: poly_neighboring}]
     By Poisson thinning and a union bound, we have the following upper bound:
\begin{align} \label{eq: poly_neighboring_proof_result_1}
            \bP_{\lambda, a^2} \left( \widehat{B}_v \not\subseteq \mathcal{R}_D \right) \leq \sum_{x \in \widehat{B}_v} \exp \left(- \lambda \sum_{y \in D} \bP_y(\tau_x \le a^2 ) \right).
\end{align}
    Since $V_0$ is an $(a, \beta a)$-group net, for any $x \in \widehat{B}_v$, we have $D \subseteq B_x(4\beta a + a/3)$. Moreover, it follows from the fact that $G$ is of polynomial growth and quasi-transitive, that there exists a constant $K_3$ such that $B_v(a) \le K_3 B_v(a/3)$ for any $v$ and $a \ge 1$. Therefore, $|D| \ge |B_v(a/3)|/4 \ge |B_v(a)|/(4K_3)$. This allows us to apply Lemma \ref{lemma: random_walk_expected_hitting_size_main_text} with $K_1 = 4\beta + 1/3$ and $K_2 = 4K_3$. Consequently, the expression in \eqref{eq: poly_neighboring_proof_result_1} is upper bounded by
        \begin{align*}
            \big| B_v(4\beta a+ a/3) \big| \exp \left(- c \lambda
            \cdot\theta(a) \right),
        \end{align*}
    where $c$ is the same constant appeared in Lemma \ref{lemma: random_walk_expected_hitting_size_main_text}. Since $G$ is a quasi-transitive of polynomial growth and $\theta(a) \gg \log a$, we have the desired bound of $C' e^{-c'\lambda \theta(a)}$ for appropriate $C'$ and $c'$.
\end{proof}

We can now finally prove the main result of this subsection:

\begin{proof}[Proof of (2) of Theorem \ref{main_quasi:existence}]
    Fix $\lambda > 0$. 
    By Poisson thinning, for every vertex $x\in V$, we split $\mathcal{P}(x)$ into two disjoint sets, $\mathcal{P}^{(1)}(x)$ and $\mathcal{P}^{(2)}(x)$, such that $|\mathcal{P}^{(1)}(x)|$ and $ |\mathcal{P}^{(2)}(x)|$ are independent Poisson random variables, each with mean $\lambda/2$. Since each particle performs an independent random walk, the particles in the first set $\mathcal{P}^{(1)} := \cup_{x \in V} \mathcal{P}^{(1)}(x)$ evolve independently of those in the second set $\mathcal{P}^{(2)} := \cup_{x \in V} \mathcal{P}^{(2)}(x)$. 

    More precisely, we consider a realization the frog model where for each $x \in V,$ its particles are taken to be $\mathcal{P}^{(1)}(x)$ and its trajectories are taken in accordance to $\cup_{\omega \in \mathcal{P}^{1}(x)} (Y^{\omega}_{x}(s))_{0\leq s \leq t}$. We note that there is a slight abuse of notation in comparison to the one introduced in Section \ref{introduction}. We clarify that we mean the particles $\omega$ in the set of particles $\mathcal{P}^{(1)}(x)$.  By Poisson thinning, this defines a frog model with intensity $\lambda/2$ and lifespan $t.$ We do an analogous construction for the set $\mathcal{P}^{(2)}$.  

    By Proposition \ref{prop:net-polynomial} we know that there exists $\beta(G)>0$ such that for all $a \geq 1,$ there exists an $(a, \beta a )-$ group net of $G.$ We recall that we denote the group net (for an arbitrary $a$ to be determined below) by $G_{0}:= (V_{0},E_{0})$.
    Now, for every $v \in V_{0},$ we let the set of good vertices $\mathcal{G}^{(1)}_{B_{v}(a/3)} := \mathcal{G}^{\lambda/2, t}_{B_v(a/3)}$ to be the respective definition of the set of good vertices for the frog model defined by the particles of $\mathcal{P}^{(1)}(x)$. Analogously, for every $x\in V$ we define the set of vertices activated by $x$ with respect to $\mathcal{P}^{(1)}(x)$ as  $\mathcal{A}_x^{(1)} := \mathcal{A}_x^{\lambda/2,t, B_v(a/3)}$. 
    
    Additionally, let $D:= \mathcal{A}^{(1)}_x \cap B_v(a/3)$.  Observe that $D \subseteq B_v(a/3)$ with $|D| \geq |B_v(a/3)| / 4$, satisfying the conditions of Proposition \ref{prop: poly_neighboring}. Let $\mathcal{R}^{(2)}_D := \cup_{x \in D} \mathcal{R}^{(2)}_x$, where $\mathcal{R}^{(2)}_x := \cup_{\omega \in \mathcal{P}^{(2)}(x)}  \mathrm{R}^\omega_{x}$ the range of $D$  defined according to the frog model with particles in $\mathcal{P}^{(2)}.$ We say that $v$ is \textbf{open} if: 
    \begin{center}
        $\exists \, x \in \mathcal{G}^{(1)}_{B_{v}(a/3)}$ such that $\widehat{B}_v \subseteq \mathcal{R}^{(2)}_{\mathcal{A}^{(1)}_x \cap B_v(a/3)}$ .
    \end{center}
    We have that for every $v \in V_{0}$

    \begin{multline}  \label{eq: open_vertex}
        \bP_{\lambda, t} \left(v \textup{ is closed}\right) \leq \bP_{\lambda, t}\left(\mathcal{G}^{(1)}_{B_{v}(a/3)} = \varnothing\right)  + \\ \bP_{\lambda, t} \left( \forall x \in B_{v}(a/3), \hspace{0.5em} \widehat{B}_v \not\subseteq \mathcal{R}^{(2)}_{\mathcal{A}^{(1)}_x \cap B_v(a/3)}\, \big| \, |\mathcal{G}^{(1)}_{B_{v}(a/3)}| > 0\right). 
    \end{multline}
   Let $t:= a^{2}$. By Proposition \ref{prop: poly_internal} applied with $A=B_{v}(a/3)$ \footnote{Notice that due to the monotone coupling in $t$ (described in Sections \ref{introduction}  and \ref{sharpness})  the bound of Proposition \ref{prop: poly_internal} also applies to $a^{2},$ which is larger than $a^{2}/9.$ } and Proposition \ref{prop: poly_neighboring}, it follows that for any $\varepsilon$, there exists $a$ sufficiently large such that the right-hand side of \eqref{eq: open_vertex} is bounded by $\epsilon$ uniformly for all $v \in V_0$. We choose $a_{\star}$ such that $\sup_{v \in V_0} \bP_{\lambda, a_{\star}^2}\left(v \textup{ is open}\right) \ge 4/5,$ and recall that $p_{c}^{\text{site}}(G_{0})\leq 3/4$. Now, consider a $(a_{\star},\beta a_{\star})-$group net on $G$. Note that the $\{v \textup{ is open}\}$ and $\{w \textup{ is open}\}$ are independent events since $B_v(a_{\star}/3)$ and $B_w(a_{\star}/3)$ are disjoint and both events only depends on particles within the corresponding balls. Therefore, $\{ \mathds{1}_{v \textup{ is open}} \}_v$ stochastically dominates a supercritical Bernoulli site percolation on $G_0 = (V_0, E_0)$. 
    
    Let $\mathcal{D}_v \subseteq G_0$ be the connected open component of $v$ in the net $G_0$.  Let $\mathcal{C}_x$ be the (outward) cluster of $x \in V$ as defined in Section \ref{introduction}. From the definition of openness and a simple coupling argument, we have that for any $v \in V_0$, 
    \begin{align} \label{eq: stochastic_domination}
      \big\{|\mathcal{D}_v| = \infty \big\} \subseteq \big\{ \exists\, x \in B_v(a_{\star}/3), \, |\mathcal{C}_x| = \infty \big\} \, ,
    \end{align}
    as the existence of an infinite cluster $D_{v}$ implies that there exists a sequence of good vertices that conquer neighboring balls internally, starting a chain of activations of good vertices that go to infinity, and thus implying that $|\mathcal{C}_{x}| = \infty.$  Since almost surely there exists an infinite cluster on $G_0$, it follows from \eqref{eq: stochastic_domination} that, for some $x \in V$, we have $\bP_{\lambda, a_{\star}^2}(x \rightarrow \infty) > 0$.  We show that there exists some $x$ that reaches infinity, but we do not know which vertex it is. By the Harris-FKG inequality, since the events $\{\mathbf{0} \rightarrow x\}$ and $\{x \rightarrow \infty\}$ are increasing, we derive that $\bP_{\lambda,a_{\star}^{2}}(\mathbf{0} \rightarrow \infty) >0$ and conclude the result.
\end{proof}

Before we end this subsection, we provide a few details about the Remark \ref{remark_heavy_tails}. Indeed, the proof of existence relies on Proposition \ref{prop:net-polynomial} and Propositions \ref{prop: poly_internal} and \ref{prop: poly_neighboring}. The proof of Proposition \ref{prop:net-polynomial} remains unchanged since it is a fact about graphs of superlinear polynomial growth, and the latter propositions rely on controlling the range of the random walk that comes through as a consequence of Lemma  \ref{apx_lemma: random_walk_expected_hitting_size}. We can prove an analogous result to Lemma \ref{apx_lemma: random_walk_expected_hitting_size} in the case of $P(x,y)  \asymp 1/ d_{G}(x, y)^{d + \alpha}$ with $a^2$ substituted by $\ell(a)$ and $\theta^{\star}(a)$ instead of $\theta(a)$, where   
        \begin{align*}
          \theta^{\star}(a) := \begin{cases}
                    \ell(a) & \text{if } d \ge 3 \text{ or } (d = 2 \text{ and } \alpha < 2), \\
                    \ell(a) / \log \log \ell(a) & \text{if } d = 2 \text{ and } \alpha = 2, \\
                  \ell(a) / \log \ell(a) & \text{if } d = 2 \text{ and } \alpha > 2,
                 \end{cases}
         \end{align*}
        and where 
         \begin{align*}
             \ell(a) := \begin{cases} 
                 a^\alpha & \text{if } \alpha \in (0, 2), \\
                 a^2 & \text{if } \alpha > 2, \\
                 a^2 / \log a & \text{if } \alpha = 2.
                 \end{cases}
         \end{align*}

Although a version of this lemma that has been recently proved in \cite{angel2026heavytailedfrogmodel} is stated only for $\Z^d$, it is not hard to generalize the result for finitely generated groups of polynomial growth (see Theorem $1.7$ of \cite{Saloff-Coste_Zheng_2017}). We now begin the proof of item $(3)$ of Theorem \ref{main:existence}.

\subsection{Non-Amenable Graphs}
\label{sec:nonamenable}

In this section, we conclude the main results of the paper by proving item $(3)$ of \cref{main:existence}. In fact, we prove a stronger result by generalizing the frog model to run on an electrical network.
To state our result, we first introduce some (standard) notations.

Let $(G,W)$, be an electrical network, with positive edge weights $W=\big(w(x,y): E\to\R_+\big)$. 
We say that $G$ is weight-locally finite 
if $\pi(x):= \sum_{y\in V} w(x,y)< \infty$, for every $x\in V$.  Note that there may still be infinitely many non-zero terms for any $x$.
We assume implicitly that $G$ is connected.
We let $\mathrm{P}$ be the transition matrix on $V$ defined as $\mathrm{P}(x,y) = w(x,y)/\pi(x)$, and recall that $P$ is reversible w.r.t. $\pi$.
We will use a condition analogous to bounded degree, namely that for any $x, y \in V$, $\pi(x)/\pi(y) \leq \cK$ for some constant $\cK$, in which case we say $\pi$ is $\cK$-controlled. 

The \textbf{Cheeger constant} $\kappa = \kappa(\mathrm{P})$ of the electrical network is defined by \[\kappa(\mathrm{P}) := \inf_{A, \pi(A) <\infty} \Phi_{\mathrm{P}}(A),\] where
\begin{equation}
\Phi_{\mathrm{P}}(A):= \frac{\sum_{a \in A, b \not \in A} \pi(a) \mathrm{P}(a,b)}{\pi(A) } = \frac{\sum_{a \in A, b \not \in A}w(a,b)}{\sum_{a \in A} \pi(a)} \, ,
\end{equation}
We say that $G$ is non-amenable if $\kappa >0$.

Let $p^{n}(x,y) = \bP^{\mathrm{P}}_{x}(X(n)=y)$ for $x,y\in V$, be the heat kernel, and observe that $\mathrm{P}$ is a linear operator on the Hilbert space $\ell^{2}(V,\pi)$ with inner product $\langle f,g\rangle_{\pi}:= \sum_{x \in V} \pi(x) f(x)g(x)$.
Its operator norm is
\[\norm{\mathrm{P}}_{\pi} := \sup\bigg\{\frac{\norm{\mathrm{P}f}_{\pi}}{\norm{f}_{\pi}}; f\neq0\bigg\} \, ,
\]
and its \textbf {spectral radius} is denoted $\rho = \rho(\mathrm{P})$.

It is standard (e.g. Proposition $6.6$ of \cite{MR3616205}) that for any $x,y \in V$ we have that 

\begin{equation}\label{def: spectral_radius}
\norm{\mathrm{P}}_{\pi} = \limsup_{n \to \infty} p^{n}(x,y)^{1/n}=: \rho(\mathrm{P}) \, ,
\end{equation}
and that for every $n\geq 1$
\begin{equation} \label{eq: p_n_spectral_upper}
  p^{n}(x,y) \leq \sqrt{\frac{\pi(y)}{\pi(x)}} \norm{\mathrm{P}}^{n}_{\pi} \, .
\end{equation}
Kesten's criterion (see \cite{kesten_criteria,MR1743100} and Theorem 6.7 of \cite{MR3616205} for the version we are using) states that a weighted graph $G$ is non-amenable if and only if $\rho(\mathrm{P})<1$.

\medskip

We consider below the frog model where active particles follow a continuous-time random walk with lifespan $t$, and jump rates given by $P(x,y)$ (i.e., total jump rate 1 from any $x$).
We denote the associated probability measure by $\P_{\lambda,t}^{P}$.
The main result of this section states the following: 

\begin{theorem}\label{thm: non-amenable_phase_transition}
  Let $G$ be a non-amenable electrical network.
  Then, for any $\lambda > 0$, 
  \begin{equation}
    t_c(\lambda) \leq   \frac{200\mathcal{K}^{2}\left(\log_{\rho}\left(\frac{1-\rho}{32\cK}\right)+1\right)}{(1-\rho)^{2} \min\{1,\lambda\}} \, .
    \label{cond: cond_on_t_fix_lambda_supercrtical_nonamenable}
  \end{equation}
\end{theorem}

Item $(3)$ of \cref{main:existence} follows as an easy consequence of  \cref{thm: non-amenable_phase_transition}, since for graphs of bounded degree, we can set $w(x,y)=1$ for all $x,y \in V$,
which satisfies $\sup_{x,y \in V} \pi(x)/\pi(y) \leq \Delta,$ where $\Delta$ denotes the maximum degree of $G$.
From now on, we drop the dependency of $\bP_{x}$ and $\bP_{\lambda,t}$ on $\mathrm{P}$.

If we restrict our attention to $\lambda <1$, then \eqref{cond: cond_on_t_fix_lambda_supercrtical_nonamenable} says equivalently to $\lambda \cdot t_c(\lambda) \leq C$, for some $C$ depending on $\cK,\rho$.
This provides evidence to our intuition that $\lambda t$ is ``comparable'' to the number of offspring in the standard Bienaymé-Galton-Watson branching process.
We remark that the constant in \cref{thm: non-amenable_phase_transition} is far from optimal.
However, we do believe that a result in the direction of proving that, for regular graphs such as $\mathbb{Z}^{d}$ or $\mathbb{T}_{d}$, $\lambda \cdot t_{c}(\lambda) \rightarrow 1 $ as $d$ tends to infinity, for any $\lambda>0$, would be interesting.

The proof of \cref{thm: non-amenable_phase_transition} is mainly a consequence of a similar exploration argument as the one in the proof of \cref{prop: poly_internal}.
In the non-amenable case, the argument is a bit simpler, since we just need the exploration process to continue indefinitely.
To this end, we provide estimates for the trace of a random walk on non-amenable graphs, based on the following lemma due to \cite{localityconjecture}, which provides a uniform lower bound on the escape probability in terms of the spectral radius: 

\begin{lemma} \label{lemma: spectral_lower_bound_of_escape_probability}
  Let $(G,W)$ be a non-amenable electrical network as above with spectral radius $\rho$. Assume that $W$ is weight-locally finite.
  Let $A \subseteq G$ be a non-empty subset, and let $\pi_{A}(\cdot) := \pi(\cdot) / \pi(A)$ be the normalized restriction of $\pi$ to $A$.
  Also for a discrete-time random walk $X$, let $T^+_A := \inf \{n \ge 1: X(n) \in A \}$ be the \textbf{positive hitting time} of $A$.
  Then,
  \begin{align}
    \bP_{\pi_A} \left( T^+_A = \infty \right) := \sum _{x \in A}\pi_{A}(x) \bP_{x}(T^+_A = \infty) \geq 1 - \rho.
  \end{align}
\end{lemma}
\begin{lemma}\label{lemma: number_of_good_vertex_nonamenable}
    Let $(G, W)$ be a non-amenable electronic network with spectral radius $\rho$. Assume $W$ is weight-locally finite and $\cK$-controlled. 
    Let $A$ be a non-empty finite subset of $V$, then
    $$\left| \, \left\{ x \in A: \bP_x\left(T_A^{+} = \infty\right) \ge \frac{1-\rho}{2\mathcal{K}} \right\} \, \right|\geq \frac{1-\rho}{2\mathcal{K}} \cdot | A | \, .$$ 
\end{lemma}

\begin{proof}
    On the one hand, it follows from Lemma \ref{lemma: spectral_lower_bound_of_escape_probability} that 
    \begin{align*}
        \sum_{x \in A} \bP_x(T_A^{+} = \infty) \cdot \frac{\cK}{|A|} \ge \sum_{x \in A} \bP_x(T_A^{+} = \infty) \cdot \frac{\max_y \pi(y)}{\min_y \pi(y)|A|} \ge \bP_{\pi_A}(T^+_A = \infty) \ge 1 - \rho.
    \end{align*}
    On the other hand, we have
    \begin{align*}
        \sum_{x \in A} \bP_x(T_A^{+} = \infty) \le \left| \, \left\{ x \in A: \bP_x\left(T_A^{+} = \infty\right) \ge \frac{1-\rho}{2\mathcal{K}} \right\} \, \right| +  \frac{1-\rho}{2\cK} \cdot |A|.
    \end{align*}
    Combining those two observations, we conclude the proof.
\end{proof}

As the reader might have noticed, the above result is stated in the context of discrete-time Markov chains. However, Lemma \ref{lemma: spectral_lower_bound_of_escape_probability} is still valid for the continuous-time random walk associated with the discrete-time random walk, simply by considering the sequence of vertices visited by it. That is, let $\{Y(s)\}_{s\geq0}$ be the continuous-time simple random walk on $(G,W)$ with heat kernel $p_{s}(x,y):= e^{s(\mathrm{P}-I)}(x,y)$. If we consider $(X(n))_{n\geq 0} $ as the sequence of vertices traversed by $Y,$ then Lemma \ref{lemma: spectral_lower_bound_of_escape_probability} holds for $X$ with $\rho(\mathrm{P}).$

The following result will be fundamental in guaranteeing a high probability for the exploration process to succeed, and we remark that it is a general statement about random walks on non-amenable graphs. For $x\in V$ we make a slight abuse of notation \footnote{This was originally introduced for the simple random walk, but it can be analogously written for the random walk associated with $\mathrm{P}$.} and let $\mathcal{R}_{x}(t) := \cup_{1 \leq i \leq \eta_{x}} \mathrm{R}_{x}^{i}(t),$ where $\mathrm{R}_{x}^{i}(t)$ denotes the trace of $\{Y^{i}_{x}(s)\}_{0\leq s \leq t}$

\begin{lemma}\label{non_amenable_trace_lower}
  Let $(G,W)$ be a non-amenable electrical network, and $R(t)$ be the trace of $\{Y(s)\}_{0\leq s \leq t}.$  For every $\eps>0$ there exist $t_0,\alpha>0$ such that for $t\geq t_0$
  \[
    \bP_x(|R(t)|\leq \alpha t) \leq \varepsilon. 
  \]
\end{lemma}

\begin{proof}

There are two ways the range can fail to be large: either the walk makes few jumps, or many jumps are to previously visited vertices. We will use a union bound on those. 
  We first observe by Chernoff's bound for the Poisson random variable $N(t)$ that $\P(N(t)\leq t/2) \leq e^{-t/4}$, where $N(t)$ denotes the number of steps performed up to time $t$.

  Let $(X(n))_{n\leq N(t)}$ be the sequence of $N$ vertices visited by the random walk from $x$.
  The sequence $(X(n))$ is the initial $N(t)$ steps of a random walk independent of $N(t)$.
  Nearby values are likely to repeat, so for $m\geq 1$ we consider the subsequence $Y(i) = X(mi)$ of every $m$-th vertex.
  Let $I$ be the number of repetitions in this subsequence with $i=\lfloor t/2m \rfloor$.
 Equation \eqref{eq: p_n_spectral_upper} implies for $i<j$ that
 
  \[
    \bP_x\big(Y(j)=Y(i)\big) = \E_x\big[ \P_x\big(Y(j)=Y(i)|Y(i)\big)\big] \leq \rho^{m(j-i)}.
  \]
  Summing over $i< j\leq t/2m$, we find that
  \[
    \bE \big[I \big] \leq \frac{t}{2m}(\rho^m+\rho^{2m}+\dots) = \frac{t\rho^m}{2m(1-\rho^m)}. \]
  By Markov's inequality, $\P(I > t/4m) \leq \frac{2\rho^m}{(1-\rho^m)}$.

   If $N(t) > t/2$ and $I\leq t/4m$ the range is at least $t/4m$. Thus
   
  \begin{align*}
    \bP_x\left(|R(t)| \leq \frac{t}{4m}\right) & \leq e^{-t/4} + \frac{2\rho^m}{(1-\rho^m)}.
  \end{align*}
  Given $1\geq \eps>0$, we can choose $t_{0}:= 4 \log (2/\varepsilon),$ and  $m:= \big \lceil \log_{\rho}(\varepsilon/8)   \big\rceil $ so that each summand is at most $\eps/2$, and $\alpha=1/4m$ gives the claim. 
\end{proof}

Now,  for $\alpha:= \frac{1}{4\left \lceil\log_{\rho}\left(\frac{1-\rho}{32\cK}\right) \right\rceil },$ define:

\begin{align}\label{def_good_G}
    G_A := \left\{ x \in A :   \bP_x \left( |\mathrm{R}(t) \cap A^c| > \alpha \cdot t\right) \ge \frac{1-\rho}{4\mathcal{K}}\right\} \, ,
\end{align}

We want to estimate the size of $G_{A}$. For that, we show that a point $x \in A$ satisfying $\bP_{x}(T_{A}^{+} = \infty) \geq (1-\rho)/2\mathcal{K},$ belongs to $G_{A}$. Take such a point, and set $\varepsilon =(1-\rho)/4\mathcal{K}.$ We have by Lemma \ref{non_amenable_trace_lower} and Bayes' Theorem, that for any $t\geq t_{0}= 4 \log\left(\frac{8\mathcal{K}}{1-\rho}\right):$

 \begin{equation}
         \bP_x \left( |\mathrm{R}(t) \cap A^c| > \alpha t \right)  \geq \bP_{x}( \{ T_{A}^{+} = \infty \} \cap \{ \mathrm{R}(t)> \alpha t \} )\ge \frac{1-\rho}{4\mathcal{K}}\, ,
   \end{equation}

 From Lemma \ref{lemma: number_of_good_vertex_nonamenable}, we conclude that 

\begin{equation}\label{number_points_G}
    |G_A | \geq \frac{1-\rho}{2\mathcal{K}} \cdot |A| .
\end{equation}
We are now ready to prove Theorem \ref{thm: non-amenable_phase_transition}. Fix $t$ such that

$$t \geq \frac{200\mathcal{K}^{2}\left(\log_{\rho}\left(\frac{1-\rho}{32\cK}\right)+1\right)}{(1-\rho)^{2} \min\{1,\lambda\}}. $$
In particular, this implies that $t> t_{0},$ thus satisfying Equation \eqref{number_points_G}, and that:

\begin{equation} \label{ineq: non-amenable_condition}
    \alpha t \frac{(1-e^{-\lambda}) (1-\rho)}{4\mathcal{K}}> \frac{6\mathcal{K}}{1-\rho},
\end{equation}
which will be used later in the proof.

Similarly to the proof of Proposition \ref{prop: poly_internal}, the main idea to prove Theorem \ref{thm: non-amenable_phase_transition} is to construct an exploration process. A key difference is that we only explore vertices of the active set when there is a non-negligible probability that a random walk starting from said vertex visits an amount proportional to $t$ outside of the vertices that have been reached by the exploration process. That is, if $\mathcal{W}$ denotes the set of vertices reached by the frogs, we only explore the trajectories of vertices belonging to $G_{\mathcal{W}}\cap \mathcal{M},$ where $\mathcal{M}$ denotes the set of active vertices.
We now prove Theorem \ref{thm: non-amenable_phase_transition} and note that a similar argument was done in the previous subsection, and so we omit a few standard calculations:
\begin{proof}[Proof of \cref{thm: non-amenable_phase_transition}]
We use the Abelian property of the frog model, and reveal trajectories of activated frogs in an order chosen so that each frog revealed is likely to reach many new vertices.
Towards this, denote by $\mathcal{W}_k$ the set of vertices reached by frogs in the first $k$ steps of the process. Initially $\mathcal{W}_0 = \{\mathbf{0}\}$ is just the origin.
The set $\cW_k$ is split into $\cV_k \cup \cM_k$, where $\cV_k$ are the vertices $x$ for which we have already revealed the number and trajectories of frogs starting at $x$, and $\cM_k$ are those vertices whose number of particles and their trajectories have not yet been revealed.
Initially $\mathbf{0}\in\cM_0$.

Let $\cG_k = G_{\cW_k}$ be the set of vertices in $\cW_k$, from which a frog is likely to reach many vertices outside $\cW_k$, as in Equation \eqref{def_good_G}. 
If $\cG_k\cap\cM_k \neq \varnothing$, pick some vertex $y$ in this intersection, and reveal the frogs from $y$. 
This has the following effect for the sets at time $k+1$:
\begin{itemize}[nosep]
    \item $y$ is moved from $\cM_{k+1}$ to $\cV_{k+1}$,
    \item vertices in $\mathcal{R}_y \setminus \cW_k$ are moved to $\cM_{k+1}$. 
\end{itemize}
If at any time the intersection $\cG_k\cap\cM_k$ is empty, we give up and stop.
We shall see that there is a positive probability this never happens, in which case $\mathbf{0}\to\infty$ occurs.

To see this, note that from Equation \eqref{number_points_G}, we know that $|\cG_k| \geq \frac{1-\rho}{2\mathcal{K}}|\cW_k|$, whereas $|\cV_k|=k$.
Thus we wish to show that with positive probability $|\cW_k| > \frac{4 \cK}{1-\rho} k$ for all $k$.
The increment $\cW_{k+1} \setminus \cW_k$ is the set of vertices outside $\cW_k$ reached by a frog from $y$.
With probability $1-e^{-\lambda}$, there is at least one frog at $y$. 
If there is such a frog, with probability at least $(1-\rho)/(4\cK)$ it reaches at least $\alpha \cdot t$ vertices outside $\cW_k$. 
Thus the change of $|\cW_k|$ stochastically dominates a random variable that is $\alpha t$ with probability $(1-e^{-\lambda})(1-\rho)/4\mathcal{K}$, and is 0 otherwise. 
(i.e. $\alpha t$ times a Bernoulli random variable.)
It follows from \eqref{ineq: non-amenable_condition} that this random variable has expectation strictly greater than $\frac{6\cK}{1-\rho}$. 
Since the increments of $\cW_k$ dominate i.i.d. $\alpha t$ times Bernoulli variables, it follows that with positive probability $|\cW_k| >  \frac{4\cK}{1-\rho} k$ for all sufficiently large $k$. 
E.g., this can be seen by applying Chernoff's inequality to the sum of the Bernoulli variables. 
If the first step of the exploration process reveals a sufficiently large set, with positive probability, the process will survive for all time and even grow linearly as above, thereby concluding the proof.
\end{proof}

\section*{Acknowledgements.}
The authors would like to thank G\'{a}bor Pete for helpful remarks and comments, and specifically for pointing out the work of Pete and Rokob \cite{pete2025nonamenablepoissonzoo}, which provides an alternative approach to the existence of the phase transition for the non-amenable setup.   

\appendix
\renewcommand{\thesection}{\Alph{section}}
\setcounter{section}{0}

\section{Appendix}
We include here the supplementary results used throughout this paper. We subdivide them based on the conditions assumed on the graph $G$.

\subsection{Quasi-Transitive Graphs of Polynomial Growth}

The main goal of this subsection is proving Lemma \ref{lemma: sub-gaussian_cts_quasi}. For that, we need a few extra definitions. Observe that Equation \eqref{gromov} implies that the growth function satisfies the \textit{volume doubling} condition:

\begin{definition}
A graph $G=(V,E)$ satisfies the \textbf{volume doubling} condition if there exists $C > 0$ such that for all $x\in V$ and $n\geq1:$
\begin{equation}
g_{x}(2n) \leq C g_{x}(n) \, .
\end{equation}
\end{definition}

For quasi-transitive graphs of polynomial growth, this is an immediate consequence of Equation \eqref{gromov}. In order to prove that $t_{\mathrm{c}}(\lambda) < \infty,$ we use sub-gaussian estimates for the heat kernel of the continuous-time random walk  $p_{t}(x,y) := \bP_x(Y(t) = y)$ on $G$. In order to establish that, we first prove them in the context of $1/2-$lazy discrete time random walks. To distinguish from the prior notation, we denote the discrete heat kernel by $p^{n}(x,y)$. We shorten $p^{1}(x,y)$ as $p(x,y)$. For $f \in \R^{V},$ we define the length of the gradient as

\begin{equation}
    |\nabla f|(x)  = \bigg( \frac{1}{2} \sum_{xy \in E} |f(x)-f(y)|^{2} p(x,y)\bigg)^{1/2}.
\end{equation}

\begin{definition}
We say that a graph $G$ satisfies the Poincaré Inequality if there exist $C>0$ and $C^{\prime} \geq 1$ such that for any $f \in \R^{V}$ and $r> 0:$
\begin{equation}
    \sum_{y \in B_x(r)} |f(x) - f_{r} (x)|^{2} \leq Cr^{2} \sum_{B_x(C^{\prime}r)} |\nabla f|^{2}(y) \, ,
\end{equation}
where  
\begin{equation*}
    f_{r}(x):= \frac{1}{|B_x(r)|} \sum_{y \in B_x(r)}f(y) \, .
\end{equation*}
\end{definition}

An important fact is that quasi-transitive graphs of polynomial growth satisfy the Poincaré inequality. This is a consequence of the Poincaré inequality for Cayley graphs of finitely generated groups (see Theorem 2.2 of \cite{kleiner-poincare}) and the stability of Poincaré inequalities under quasi-isometries (see Proposition 3.14 of \cite{mathav-harnack} and \cite{Varopoulos_Saloff-Coste_Coulhon_1993}) in combination with Trofimov's theorem. As a conclusion of Delmotte's seminal paper (see Theorem 1.7 of \cite{Delmotte1999}) there exist $c^{\prime}_{1},c^{\prime}_{2},C^{\prime}_{1},C^{\prime}_{2}>0$ such that 

\begin{equation}\label{discrete_subgaussian}
\frac{c^{\prime}_{1}}{n^{d/2}} e^{-d_{G}^{2}(x,y)/c^{\prime}_{2}n} \leq p^{n}(x,y) \leq \frac{C^{\prime}_{1}}{n^{d/2}} e^{-d_{G}^{2}(x,y)/C^{\prime}_{2}n} \, , 
\end{equation}
for every $x,y \in V, n \geq 1$ such that $d_{G}(x,y) \leq n$ (for $d_{G}(x,y)>n,$ we have $p^{n}(x,y)=0$). Back to our context, we prove that an analogous bound holds for the continuous heat kernel $p_{t}(x,y)$.

\begin{lemma} \label{lemma: sub-gaussian_cts_quasi}
Let $G$ be a quasi-transitive graph of polynomial growth. Then, there exist constants $c_{1},c_{2},C_{1},C_{2}>0$ depending on $G$ such that
\begin{equation}
  \frac{c_{1}}{t^{d/2}} e^{-d_{G}^{2}(x,y)/c_{2}t} \leq p_{t}(x,y) \leq \frac{C_{1}}{t^{d/2}} e^{-d_{G}^{2}(x,y)/C_{2}t}
\end{equation}
for all $x,y \in V, t\geq 1$ with $d_{G}(x,y) \leq t$.
\end{lemma}

\begin{proof}[Proof of Lemma \ref{lemma: sub-gaussian_cts_quasi}]
First, observe that by the estimate in Equation $\eqref{chernoffupper}$, we have that for all $x > 1$:
\begin{equation}\label{chernofflower}
\bP\left( \text{Po}(t) < \left(1-\frac{1}{x}\right)t\right) \leq \exp \left( -\frac{t}{2x(x+1)} \right)\, .
\end{equation}
Let $x,y\in V$ and $t\geq1$ be chosen arbitrarily. For the upper bound, observe that Equations  \eqref{chernoffupper} and \eqref{discrete_subgaussian} imply:
\begin{align*}
    p_{t}(x,y) &= \sum_{k=d_{G}(x,y)}^{\infty}\bP( \text{Po}(2t) = k) p^{k}(x,y) \\ 
    &\leq \sum_{k= t}^{4t} \bP( \text{Po}(2t) = k) \frac{C^{\prime}_{1}}{k^{d/2}}e^{-d_{G}^{2}(x,y)/C^{\prime}_{2}k} + \frac{2 C^{\prime}_{1}}{d_{G}(x,y)^{d/2}}e^{-t/6}\\
    &\leq \frac{C_{1}}{t^{d/2}} e^{-d_{G}^{2}(x,y)/C_{2}t} \, ,
\end{align*}
for $C_{1},C_{2}>0.$ For the lower bound, combining Equation \eqref{discrete_subgaussian}  with the fact that $d_{G}(x,y) \leq t,$ and $t\geq 1$  yield:
\begin{align*}
    p_{t}(x,y) &= \sum_{k=d_{G}(x,y)}^{\infty}\bP( \text{Po}(2t) = k) p^{k}(x,y) \\
    &\geq \sum_{k= d(x,y)}^{4t} \bP( \text{Po}(2t) = k) \frac{c^{\prime}_{1}}{k^{d/2}}e^{-d_{G}^{2}(x,y)/c^{\prime}_{2}k} \\
    &\geq \frac{c_{1}}{t^{d/2}}e^{-d_{G}^{2}(x,y)/c_{2}t} \, ,
\end{align*}
for some choice of $c_{1},c_{2}>0.$
\end{proof}

\subsection{Quasi-Transitive Graphs of Superlinear Polynomial Growth}

In this subsection, we prove Lemma \ref{lemma: random_walk_range} using Lemma \ref{lemma: sub-gaussian_cts_quasi}. We first define the \textit{truncated green function} $\mathbf{G}_t: V \times V \rightarrow [0,\infty)$ as 
\begin{align}
    \mathbf{G}_t(x, y) := \int_0^t p_s(x, y) \, \mathrm{d}s \, ,
\end{align}
where $x, y \in V$ and $t \ge 0$. We begin by proving the following auxiliary lemma regarding the trace of a random walk:

\begin{lemma} \label{apx_lemma: random_walk_expected_hitting_size}
    Let $G$ be a quasi-transitive graph of superlinear polynomial growth.
    For any $K_1, K_2 \ge 1$, there exist constants $c = c(G, K_1, K_2)$ and $C = C(G)$, such that for any $a \ge 2$, $x \in V$, and $D \subseteq B_x(K_1 a)$ with $|D| \ge |B_x(a)|/K_2$, we have


    \begin{align}
        c \cdot \theta(a) \leq \sum_{y \in D} \bP_x \left( \tau_y \le a^2 \right) \leq C \cdot \theta(a) \, ,
    \end{align}
    where $\theta(a)$ is defined as 
    \begin{align*}
        \theta(a) := \begin{cases}
            a^2 & \text{if } d \geq 3, \\
            a^2 / \log a & \text{if } d = 2.
        \end{cases}
    \end{align*}
\end{lemma}

\begin{proof}
It follows from Lemma \ref{lemma: sub-gaussian_cts_quasi} that, for any $z \in V$ and any $t \ge 1$, we have
$$\mathbf{G}_{t}(z,z) \asymp \int_{0}^{1} p_s(z,z) \mathrm{d}s + \int_{1}^{t} \frac{1}{s^{d/2}} \mathrm{d}s \, .$$
By using that $0\leq p_{s}(x,y) \leq 1,$ we conclude that for any $z \in V$ and any $t \ge 2$, we have
\begin{align} \label{eq: on_diagonal_green_function_est}
    \mathbf{G}_{t}(z, z) \asymp \begin{cases}
        1 &\text{ if $d \ge 3$,} \\
        \log t &\text{ if $d = 2$},
    \end{cases}
\end{align}
where the implicit constant depends on the graph $G$. Let $\mathcal{N}_t(y)$ be the number of visits to $y$ up to time $t$. Then, 
\begin{align} \label{ineq: hit_to_Green_lower}
    \frac{\mathbf{G}_{a^2}(x, y)}{\bP_x(\tau_y \le a^2)} = \bE_x[\mathcal{N}_{a^2}(y) | \tau_y \le a^2] \, \le \bE_y[\mathcal{N}_{a^2}(y)] = \mathbf{G}_{a^2}(y, y) \, .
\end{align}
On the other hand, we have
\begin{align} \label{ineq: hit_to_Green_upper}
    \frac{\mathbf{G}_{2a^2}(x, y)}{\bP_x(\tau_y \le a^2)} \ge \bE_x[\mathcal{N}_{2 a^2}(y) | \tau_y \le a^2] \, \ge \bE_y[\mathcal{N}_{a^2}(y)] = \mathbf{G}_{a^2}(y, y) \, .
\end{align}
It follows from \eqref{ineq: hit_to_Green_lower} that
\begin{align} \label{ineq: expected_range_to_green_lower_bound}
    \sum_{y \in D} \bP_x(\tau_y \leq a^2) \ge \frac{\sum_{y \in D} \mathbf{G}_{a^2}(x, y)}{\mathbf{G}_{a^2}(y, y)} \, .
\end{align}
Moreover, we have 
\begin{align} \label{ineq: number_of_visits_lower_bound}
    \sum_{y \in D} \mathbf{G}_{a^2}(x, y) \ge \sum_{y \in D} \int_{\frac{a^2}{2}}^{a^2} p_s(x, y) \, ds \ge \frac{a^2 |B_x(a)|}{2K_2} \left(\frac{c_1}{a^d} \exp\left(-\frac{2 K_1^2}{c_2}\right)\right) \gtrsim_{G, K_1, K_2} a^2 \, ,
\end{align}
which combined with \eqref{eq: on_diagonal_green_function_est} and \eqref{ineq: expected_range_to_green_lower_bound} gives the desired lower bound. Similarly, \eqref{ineq: hit_to_Green_upper} gives that
\begin{align*}
    \sum_{y \in D} \bP_x(\tau_y \leq a^2) \le \frac{\sum_{y \in D} \mathbf{G}_{2a^2}(x, y)}{\mathbf{G}_{a^2}(y, y)} \le \frac{2a^2}{\mathbf{G}_{a^2}(y, y)} \le C \cdot \theta(a) \, ,
\end{align*}
where $C$ depends on the graph $G$. 
    
\end{proof}

\begin{proof}[Proof of Lemma \ref{lemma: random_walk_range}]
It follows from Lemma \ref{apx_lemma: random_walk_expected_hitting_size} with $D:= B_x(a)\setminus H, $ which satisfies $|D| \geq |B_{x}(a)|/4,$ that for any $y \in B_{x}(a):$ 
\begin{align} \label{ineq: lower_bound_expected_new_range}
    c \cdot \theta(a) \leq \bE_y\left[ \, \left| \widetilde{\mathrm{R}}^{\text{H}}(a^2) \right| \, \right]
    &= \sum_{z \in B_x(a) \setminus H} \bP_y(\tau_z \leq a^2) \leq C \cdot \theta(a) \, .
\end{align}

Let $b:= C \cdot \theta(a)$, and for \( k \in \N^+ \), define \( A(k) := \max_{y \in B_x(a)} \bP_y(|\widetilde{\mathrm{R}}^{\text{H}}(a^2)| \geq 2kb) \). Observe that by Markov's inequality \( A(1) \leq \frac{1}{2} \). Moreover, by the Strong Markov property, we have \( A(k_1 + k_2) \leq A(k_1)\, A(k_2) \) for any \( k_1, k_2 \in \N^+ \). This implies that 

\begin{equation}\label{secondmoment}
\bE_y\Big[|\widetilde{\mathrm{R}}^{\text{H}}(a^2)|^2\Big] = \sum_{\ell \geq 1} \bP_{y} \Big(|\widetilde{\mathrm{R}}^{\text{H}}(a^2)|^{2} \geq \ell \Big) \leq C' b^2 \, ,
\end{equation}
where \( C' \) is an absolute constant. The final result follows from the Paley–Zygmund inequality combined with the first and second moment estimates we obtained in Equations \eqref{ineq: lower_bound_expected_new_range} and \eqref{secondmoment}:
\begin{align}
    \mathbb{P}_y \left( \left| \widetilde{\mathrm{R}}^{\text{H}}(a^2) \right| \geq \frac{c \cdot \theta(a)}{2} \right) \geq \frac{\bE_y\left[ \left| \widetilde{\mathrm{R}}^{\text{H}}(a^2) \right| \right]^2}{4 \, \bE_y \left[ |\widetilde{\mathrm{R}}^{\text{H}}(a^2)|^2 \right]} \geq \frac{c^2}{4 C' C^2} \, , 
\end{align}
concluding the results in the appendix.
\end{proof}

\printbibliography

\end{document}

%% file: header.tex
\usepackage{amsmath,amssymb,amsfonts,amsthm}

\usepackage{enumitem}
\usepackage{graphicx}
\usepackage{color}
\usepackage{mathtools}
\usepackage{mathrsfs}
\usepackage[margin=3cm]{geometry}

\usepackage[protrusion=true,expansion=true]{microtype}

\usepackage[font=sf, labelfont={sf,bf}, margin=0.5cm]{caption}

\usepackage[colorlinks=true,linkcolor=black,citecolor=black,urlcolor=blue,pdfborder={0 0 0}]{hyperref}
\usepackage{thmtools}

\usepackage[nameinlink,capitalise]{cleveref}
  \crefname{theorem}{Theorem}{Theorems}
  \crefname{thm}{Theorem}{Theorems}
  \crefname{mainthm}{Theorem}{Theorems}
  \crefname{lemma}{Lemma}{Lemmas}
  \crefname{lem}{Lemma}{Lemmas}
  \crefname{remark}{Remark}{Remarks}
  \crefname{prop}{Proposition}{Propositions}
  \crefname{defn}{Definition}{Definitions}
  \crefname{corollary}{Corollary}{Corollaries}
  \crefname{section}{Section}{Sections}
  \crefname{figure}{Figure}{Figures}

\newtheorem{thm}{Theorem}[section]
\newtheorem{theorem}[thm]{Theorem}

\newtheorem{lemma}[thm]{Lemma}

\newtheorem{claim}[thm]{Claim}

\newtheorem{proposition}[thm]{Proposition}

\newtheorem{definition}[thm]{Definition}
\newtheorem{remark}[thm]{Remark}

\newtheorem{question}[thm]{Question}

\newtheorem{conjecture}[thm]{Conjecture}

\numberwithin{equation}{section}

\renewcommand{\P}{\mathbb P}
\newcommand{\Z}{\mathbb Z}
\newcommand{\E}{\mathbb E}
\newcommand{\R}{\mathbb R}
\newcommand{\N}{\mathbb N}
\newcommand{\eps}{\varepsilon}

\newcommand{\bP}{\P}

\newcommand{\bE}{\E}

\newcommand{\cU}{\mathcal{U}}
\newcommand{\cK}{\mathcal{K}}
\newcommand{\cF}{\mathcal{F}}
\newcommand{\cV}{\mathcal{V}}
\newcommand{\cW}{\mathcal{W}}
\newcommand{\cG}{\mathcal{G}}
\newcommand{\cM}{\mathcal{M}}
\newcommand{\T}{\mathbb{T}}
\newcommand{\1}{\mathbbm{1}}
\newcommand{\norm}[1]{\left\lVert#1\right\rVert}